\documentclass{amsart}

\usepackage{amssymb, amsmath,mathtools}
\mathtoolsset{showonlyrefs}
\usepackage{mathrsfs,mathabx}
\usepackage{amscd}
\usepackage{verbatim}
\usepackage{stmaryrd}
\usepackage[dvipsnames]{xcolor}
\usepackage{subfig}

\usepackage{enumerate}
\usepackage{stmaryrd,esint}
\usepackage[dvipsnames]{xcolor}
\usepackage{enumerate}
\usepackage[a4paper, margin=1in]{geometry}

\usepackage{booktabs} 
\usepackage[colorlinks,linkcolor={blue},citecolor={blue},urlcolor={blue},]{hyperref}
\usepackage{graphicx}
\usepackage[colorinlistoftodos,prependcaption,textsize=tiny]{todonotes}

\usepackage{forest}
\usepackage{tikz}
\tikzset{>=latex}

\usepackage{tabularx}
\newcolumntype{L}{>{\arraybackslash}X}
\usepackage{multirow}

\theoremstyle{plain}
\newtheorem{theorem}{Theorem}[section]
\theoremstyle{remark}
\newtheorem{remark}[theorem]{Remark}

\theoremstyle{plain}

\newtheorem{proposition}[theorem]{Proposition}
\newtheorem{definition}[theorem]{Definition}

\newtheorem{assumption}[theorem]{Assumption}

\numberwithin{equation}{section}

\def\N{{\mathbb N}}
\def\Z{{\mathbb Z}}
\def\Q{{\mathbb Q}}
\def\R{{\mathbb R}}

\newcommand{\E}{{\mathbb E}}
\renewcommand{\P}{{\mathbb P}}
\newcommand{\F}{{\mathscr F}}

\newcommand{\g}{\gamma}

\newcommand{\om}{\omega}
\renewcommand{\O}{\Omega}

\newcommand{\I}{I}

\newcommand{\Tor}{\mathbb{T}}
\newcommand{\T}{\mathbb{T}}

\newcommand{\A}{{\mathcal A}}
\newcommand{\loc}{{\rm loc}}
\newcommand{\tr}{\mathrm{tr}}
\newcommand{\id}{\mathrm{Id}}

\newcommand{\diam}{\mathrm{diam}}

\newcommand{\wt}{\widetilde}

\newcommand{\Borel}{\mathcal{B}}
\newcommand{\Progress}{\mathcal{P}}
\newcommand{\embed}{\hookrightarrow}
\newcommand{\one}{{{\bf 1}}}
\newcommand{\lb}{\langle}
\newcommand{\rb}{\rangle}
\newcommand{\supp}{\mathrm{supp}\,}
\renewcommand{\emptyset}{\varnothing}

\newcommand{\vp}{\varphi}
\newcommand{\Lip}{\mathrm{Lip}}
\newcommand{\dd}{\mathrm{d}}
\newcommand{\data}{\mathcal{D}}
\renewcommand{\square}{Q}
\renewcommand{\tilde}{\wt}

\newcommand{\ellip}{\nu}

\allowdisplaybreaks

\renewcommand{\S}{\rotatebox[origin=c]{90}{\tikz{\draw[thick](0,0)--(0.25,0)--(0,0.25)--cycle;}}\,}
\newcommand{\invS}{\,\rotatebox[origin=c]{270}{\tikz{\draw[thick](0,0)--(0.25,0)--(0,0.25)--cycle;}}}
\begin{document}

\author{Antonio Agresti}
\address{Department of Mathematics Guido Castelnuovo\\
	Sapienza University of Rome \\ P.le Aldo Moro 5\\ 00185 Rome\\ Italy.} 
\email{antonio.agresti92@gmail.com}

\author{Max Sauerbrey}
\address{Max Planck Institute for Mathematics in the Sciences\\
Inselstr. 22 \\ 04103 Leipzig \\ Germany.} \email{maxsauerbrey97@gmail.com}

\author{Mark Veraar}
\address{Delft Institute of Applied Mathematics\\
Delft University of Technology \\ P.O. Box 5031\\ 2600 GA Delft\\The
Netherlands.} \email{M.C.Veraar@tudelft.nl}

\thanks{The first author is a member of GNAMPA (IN$\delta$AM). The third author has received funding from the VICI subsidy VI.C.212.027 of the Netherlands Organisation for Scientific Research (NWO)}

\date\today

\title[Stochastic flow approach to De Giorgi--Nash--Moser estimates]{A stochastic flow approach to De Giorgi--Nash--Moser estimates for SPDE\lowercase{s} with smooth transport noise}




\begin{abstract}
The celebrated De Giorgi--Nash--Moser theory ensures that solutions to uniformly elliptic or parabolic PDEs are bounded and H\"older continuous, even with merely bounded measurable coefficients. For parabolic SPDEs with transport noise, boundedness has recently been established, but H\"older continuity remains a key open problem in the regularity theory of parabolic SPDEs. In this work, we resolve this question under the assumption that the noise coefficients are sufficiently regular in space.

Our approach relies on Kunita’s stochastic method of characteristics, which allows us to transform the original SPDE---via a stochastic flow of diffeomorphisms---into a random PDE to which the classical De Giorgi--Nash--Moser estimates apply. This program is accomplished through new a-priori estimates for the inverse of stochastic flows of diffeomorphisms, and a novel version of the It\^o--Wentzell formula adapted to rough random fields. To demonstrate the applicability of our results, we establish the existence of global, regular solutions to quasilinear SPDEs with transport noise.
\end{abstract}

\maketitle

\section{Introduction} \label{sec:intro}
The De Giorgi–Nash–Moser estimates are one of the cornerstones of the modern theory of PDEs. They were originally found independently and based on completely different methods by De Giorgi \cite{DeGiorgi_57} and Nash \cite{nash_58}, and were soon after reproved by Moser \cite{moser_60} using yet another approach, based on an iteration technique.
They roughly state that any \emph{weak} solution to a uniformly elliptic or parabolic linear PDE are locally bounded and even \emph{H\"older continuous}. In the absence of source terms in the given PDE and for positive solutions, H\"older continuity can be obtained via Harnack inequalities, see \cite{Kru64,MoHarnack61,MoHarnack64}. 
Parallel to their theoretical value, the De Giorgi--Nash--Moser estimates play a crucial role in interpreting weak solutions as pointwise--and thus, physically meaningful--objects. Since their discovery, the De Giorgi--Nash--Moser estimates have been extended to various other classes of equations as, e.g, quasilinear equations \cite{serrin_QL}, degenerate parabolic equations \cite{CF_porous,chen_degenrate} or merely hypoelliptic PDEs \cite{hypoelliptic_orig}, see also the monographs \cite{dibendetto, GT83, LSU} for a streamlined presentation. One of the most significant consequences of such estimates is that they often lead to \emph{global smooth solutions} of quasilinear parabolic PDEs (see, e.g., \cite[Chapter 15]{TayPDE3}). This is closely related to the original motivation of De Giorgi and Nash, who proved such estimates to solve the 19th Hilbert problem on the regularity of minimizers of variational problems (thus, a quasilinear elliptic PDE). As discussed in more detail below, our primary motivation for studying stochastic variants of the De Giorgi--Nash--Moser estimates lies in their applicability to \emph{quasilinear parabolic SPDEs}, which has not been accomplished so far.

\smallskip

In this manuscript, we establish De Giorgi--Nash--Moser estimates for \emph{SPDEs with transport noise} which apply, among others, to the following SPDE
\begin{equation}
	\label{eq:SPDE_lin_intro}
	\left\{
	\begin{aligned}
			&\dd u \, - \, \nabla \cdot (a\cdot \nabla u) \,\dd t \, = \,\textstyle{\sum}_{n\ge 1} (b_n\cdot \nabla) u \,\dd w^n_t, & & \text{ on }[0,T]\times \Tor^d,\\
		&u(0)\,=\,u_0,& & \text{ on }\Tor^d,
	\end{aligned}
	\right.
\end{equation}
with solely \emph{bounded measurable} diffusion matrix $a$ and \emph{sufficiently regular} noise coefficients $b_n$, as long as \eqref{eq:SPDE_lin_intro} is parabolic: There exists $\ellip>0$ such that almost surely 
\begin{equation}
    \label{eq:stoch_para_intro}
    \textstyle
    \forall t\in [0,T], x\in \Tor^d,  \eta\in \R^d \ \ \text{ it holds that } \ \    \eta^\top \big( a(t,x)-\frac{1}{2} \sum_{n\ge 1}b_n(t,x) \otimes  b_n(t,x) \big) \eta \,\geq \, \ellip | \eta|^2.
\end{equation}
Here, and throughout the manuscript, $\T^d$ denotes the $d$-dimensional torus and the sequence $(w^n)_{n\geq 1}$ consists of independent $\F$-Brownian motions on a probability space $(\Omega,\A , \P )$ equipped with a complete filtration $\F$. 

\smallskip

Although De Giorgi--Nash--Moser estimates for parabolic PDEs have been extensively studied, extending them to stochastic PDEs such as \eqref{eq:SPDE_lin_intro}---in a form robust enough to establish global existence of smooth solutions to quasilinear SPDEs---remains an open problem. In this manuscript, we aim to partially fill this gap by proving suitable versions of De Giorgi--Nash--Moser estimates for \eqref{eq:SPDE_lin_intro} together with an application to a quasilinear SPDE, see Theorems \ref{t:DNM_intro} and \ref{t:global_quasilinear_intro} below. 
Under the assumption \eqref{eq:stoch_para_intro}, regularity of solutions to the SPDE \eqref{eq:SPDE_lin_intro} was already investigated in \cite{DG15_boundedness,DG17_continuity}. In \cite{DG15_boundedness}, by extending the Moser iteration to SPDEs, the authors showed boundedness of solutions to \eqref{eq:SPDE_lin_intro} under the sole condition \eqref{eq:stoch_para_intro} even in the presence of sufficiently integrable source terms. 
This resulted in a refinement of the pioneering work  \cite{LMS}, in which boundedness is shown also for noises with Lipschitz dependence on $(u,\nabla u)$, but where the parabolicity condition \eqref{eq:stoch_para_intro} is replaced by a more restrictive version. 
Following this result, the same authors showed in \cite{DG17_continuity} that this bounded solution is almost surely continuous at each point $(t_0,x_0)$ in the interior of the domain by means of a Harnack inequality. However, in contrast to the deterministic case, whether such a solution is almost surely continuous at any such point\footnote{This does not follow from the previous, as it involves the intersection of uncountably many events of probability one.} or even in some H\"older class is, in general, an open problem. The boundedness result of \cite{DG15_boundedness} was subsequently generalized to quasilinear and degenerate parabolic SPDEs such as stochastic porous media equations in \cite{DG19_supremum_degenerate}. It is worth mentioning a few results on SPDEs where $\sum_{n\geq 1}(b_n\cdot\nabla) u \,\dd{w}^n_t$ is replaced by a lower-order noise of the form $\sum_{n\geq 1} f_n(t,x,u)\,\dd w^n_t$ with $\|(f_n(t,x,u))_{n\geq 1}\|_{\ell^2}\lesssim 1+|u|$ (i.e., $f=(f_n)_{n\geq 1}$ is sublinear in $u$ as an $\ell^2$-valued function), in which case much more can be said. Indeed, as demonstrated in \cite{HWW17,KNP02}, it is possible to reduce to the deterministic De Giorgi--Nash--Moser theory from which then boundedness and H\"older continuity follow. Finally, a Harnack-type inequality in the case of lower-order noise as above was proved in \cite{W18_Harnack}. We also remark that while the H\"older-regularity of solutions to \eqref{eq:SPDE_lin_intro} is currently open, other types of gain in regularity for solutions to coercive SPDEs can be proved rather generically, cf.\ \cite{Agresti_meyer,Auscher_Portal,bechtel_extrapolation_2025}.

The presence of transport noise in \eqref{eq:SPDE_lin_intro} is a central obstruction to establishing De Giorgi–Nash–Moser estimates for parabolic SPDEs. At the same time, there are several compelling reasons to consider transport noise. First, from a scaling perspective (see, e.g., \cite[Subsection 1.2]{AV25_survey}), the transport term $(b_n\cdot\nabla) u \,\dd{w}^n_t$
shares the same local scaling as the diffusive term $\nabla \cdot (a\cdot \nabla u)\,\dd t$, owing to the classical scaling properties of Brownian motions. In particular, transport noise is \emph{critical} from a PDE point of view, and this already suggests why lower-order noise terms are analytically more tractable. Second, starting from Kraichnan's influential works \cite{K68,K94}, transport-type noise is by now a well-established model in stochastic fluid dynamics, where it can effectively model the behavior of turbulent flows advecting passive scalars--e.g., in the case where $u$ represents concentrations, see \cite[Subsection 1.3]{RD_AV23} and \cite{F15_book}.
In this manuscript, we tackle the challenge posed by transport noise by using the \emph{stochastic flow method} in combination with a \emph{new distributional It\^o--Wentzell formula}; see Subsection \ref{ss:proof_strat} for details on our strategy. 
As the result below shows, we are able to cover a large range of regular Kraichnan noise, see \cite[Section 5]{GY_stabilization} for the terminology.

\smallskip

Next, we state a simplified version of our main result, Theorem \ref{t:DeGiorgiNashMoser}. 
A progressively measurable process $u$ with paths in $ L^2(0,T;H^1(\Tor^d))\cap C([0,T];L^2(\Tor^d))$ almost surely is said to be a \emph{weak solution} to \eqref{eq:SPDE_lin_intro} if for all $\phi\in C^\infty(\T^d)$ and almost surely
	\begin{align*}
		\langle
		u(t),  \phi\rangle \,-\, \langle u_0,\phi \rangle  \,=\, &  - \int_0^t \bigl(a\nabla u ,\nabla \phi\bigr)_{L^2(\T^d;\R^d)} \,\dd s \,
+\, \sum_{n\ge 1}\int_0^t \bigl(  b_n \cdot  \nabla  u,\phi \bigr)_{L^2(\T^d)}  \,\dd w_s^n,
\end{align*}
	for all $t\in [0,T]$, where the right-hand side is well-defined if, e.g., for some $M<\infty$, almost surely
    \begin{align}\label{eq:bddness_intro}
       \sum_{i,j=1}^d |a^{ij}(t,x)|+ \sum_{i=1}^d \|(b^i_n(t,x))_{n\geq 1}\|_{\ell^2} \,&\leq\, M ,
    \end{align}
    for all $(t,x)\in [0,T]\times \T^d$. We note that \eqref{eq:stoch_para_intro} and \eqref{eq:bddness_intro} guarantee for any $\F_0$-measurable $u_0:\Omega \to  L^2(\T^d)$ the existence of a unique weak solution $u$ to \eqref{eq:SPDE_lin_intro}, see Proposition \ref{prop:Existence_Uniqueness} for a more general version of this fact. 

Below, we say that $v:(0,T)\times \T^d\to \R$ is H\"older continuous if $v\in C^{\gamma}((0,T)\times \T^d)$ for some $\gamma>0$. 

\begin{theorem}[Stochastic De Giorgi--Nash--Moser estimate -- simplified version]\label{t:DNM_intro}
Let the coefficient fields $a\colon \Omega \times [0,T]\times \T^d \to \R^{d\times d}$ and $b_n \colon  \Omega \times [0,T]\times \T^d \to \R^{d}$ for $n\ge 1$ be progressively measurable and suppose that for some $\gamma_0>0$, $u_0\colon \Omega \to  C^{\gamma_0}(\T^d)$ is strongly $\F_0$-measurable. Suppose that there exist $\delta\in (0,1)$ and $\nu, M, C \in (0,\infty)$ such that \eqref{eq:stoch_para_intro}, \eqref{eq:bddness_intro} and 
\begin{equation}
    \label{eq:additional_assumption_intro}
    \|(b_n)_{n\ge 1}\|_{C^{3+\delta}(\T^d;\ell^2(\R^d)) }\,\le\, C,\qquad t\in [0,T],
\end{equation}
holds almost surely. Then, the weak solution $u$ to \eqref{eq:SPDE_lin_intro} satisfies
\begin{equation}\label{eq:statement_DNM_intro}
\P\big(\,u :(0,T)\times \T^d \to \R \,\text{\normalfont{ is H\"older continuous}}\big)\,=\,1.
\end{equation}
\end{theorem}
For a more general version of Theorem \ref{t:DNM_intro} accounting also for lower order terms and non-trivial source terms in \eqref{eq:SPDE_lin_intro}, see Theorem \ref{t:DeGiorgiNashMoser}.
It should be remarked that unlike in the well-established deterministic theory concerning \eqref{eq:SPDE_lin_intro} with $b_n=0$ where the H\"older regularity of $u$ depends only on  $\nu$, $M$  and $\gamma_0$ it may very well happen that the $\gamma>0$ for which the above guarantees $u\in C^\gamma((0,T)\times \T^d) $ may degenerate on a set of small probability, see Theorem \ref{t:DeGiorgiNashMoser_quantitative} for a quantified version of this. 
Additional   conditions on  $(b_n)_{n\ge 1}$ under  which we can actually take $\gamma>0$ uniformly in $\omega$ are derived in Proposition \ref{prop:sufficient_uniform_gamma}. The reader is referred to Subsection \ref{ss:open_problems} for a list of open problems related to Theorem \ref{t:DNM_intro}.
A comparison of these results to the existing literature on the De Giorgi--Nash--Moser estimates for SPDEs reviewed above is given in Table \ref{table}.

\begin{table}[h!]
	\centering
	\caption{Existing and our results on the De Giorgi--Nash--Moser estimates for SPDEs}
	\label{table}\begin{tabular}{p{3cm}  p{5cm} p{3.5cm}  p{3cm}}
		\toprule
		{Result} & {Leading Order Operator} & {Type of SPDEs} &  {References} \\
		\midrule
			$L^\infty$\&$C^\gamma$-estimate    &    $\nabla \cdot (a \nabla u)\,\dd t $  &  parabolic &   
		\cite{HWW17,KNP02}  \\
		$L^\infty$-estimate  &   $ \nabla \cdot (a \nabla u)\,\dd t  \,+\, b(u,\nabla u) \,\dd w$   & restricted parabolic & \cite{LMS}  \\
		$L^\infty$-estimate  &  $\nabla \cdot (a \nabla u)\,\dd t  \,+\, b\cdot \nabla u \,\dd w$  &  parabolic   & \cite{DG15_boundedness}  \\
		$L^\infty$-estimate    & $\nabla \cdot (a(u) \nabla u)\,\dd t  \,+\, b(u)\cdot   \nabla u \,\dd w $ &  degenerate parabolic  & \cite{DG19_supremum_degenerate}  \\
		continuity    & $\nabla \cdot (a \nabla u)\,\dd t  \,+\, b\cdot \nabla u \,\dd w$   & parabolic  &    \cite{DG17_continuity}  \\
		$C^\gamma$-estimate  &  $\nabla \cdot (a \nabla u)\,\dd t  \,+\, b\cdot \nabla u \,\dd w$ & parabolic, regular $b$   &current manuscript\\
		\bottomrule
	\end{tabular}
\end{table}

\begin{remark}\label{remark_improved_condition}
The role of the regularity assumption \eqref{eq:additional_assumption_intro} is to guarantee regularity of the stochastic flow of diffeomorphisms induced by the Stratonovich SDE with coefficients $(-b_n)_{n\ge 1}$, see Subsection \ref{ss:proof_strat} for a more detailed discussion. Indeed, \eqref{eq:additional_assumption_intro} guarantees that almost surely
\begin{equation}\label{Eq:alternative_ass}
\textstyle	\|(b_n)_{n\ge 1}\|_{C^{2+\delta}(\T^d;\ell^2(\R^d))} \,+\, \bigl\|\sum_{n\ge 1}(b_n \cdot \nabla) b_n  \bigr\|_{{C^{2+\delta}(\T^d; \R^d) }}\,\le \, C,\qquad t\in [0,T],
\end{equation}
up to enlarging $C$, which then allows us to argue that the said Stratonovich SDE generates a stochastic flow of $C^2$-diffeomorphisms. As a consequence, the results stated in this introduction and Sections \ref{app:DeGiorgi_Nash_Moser}--\ref{s:application_to_quasi} remain valid when replacing \eqref{eq:additional_assumption_intro} by \eqref{Eq:alternative_ass}.
\end{remark}

Despite our statement on H\"older continuity of the solution $u$ to \eqref{eq:SPDE_lin_intro} being fairly weak, it turns out sufficient to establish the existence of regular solutions to \emph{quasilinear, parabolic SPDEs}, as we will demonstrate in our forthcoming work \cite{ASV25_quasi}. As a proof of concept, we consider in Section \ref{s:application_to_quasi} the prototypical SPDE:
\begin{equation}
	\label{eq:SPDE_quasi_intro}
	\begin{cases}
		\displaystyle{
			\dd U \, - \, \nabla \cdot (A(U)\cdot \nabla U) \,\dd t \, =  \textstyle{\sum}_{n\ge 1}\bigl(B_n\cdot \nabla\bigr) U \,\dd w^n_t,} &\ \ \ \text{ on }[0,\infty)\times \Tor^d,\\
		U(0)\,=\,U_0,&\ \ \  \text{ on }\Tor^d.
	\end{cases}
\end{equation}
While weaker solution concepts, as entropy and kinetic formulations, have recently been proven very successful in establishing well-posedness of even degenerate parabolic quasilinear SPDEs with nonlinear gradient noise, see, e.g.,  \cite{DGG19, DG20,FG_ARMA,RSZ24_MATH}, the question whether the regularity of solutions improves subject to more restrictive conditions on the coefficients is unknown. For lower-order noise terms, this was achieved in \cite{DMH_regularity}, where a bootstrap argument allowed the authors to deduce smoothness in space of the unique kinetic solution constructed in \cite{DHV16}. To deal with the transport noise in \eqref{eq:SPDE_quasi_intro} we take a slightly different approach: We start with a local regular solution $U$ to \eqref{eq:SPDE_lin_intro}, which can be constructed based on earlier results \cite{AV19_QSEE1,AV19_QSEE2,AV21_SMR_torus} by the first and last authors, and show that $U$ does not blow up in finite time using Theorem \ref{t:DNM_intro}. This leads to Theorem \ref{t:global_quasilinear}, a special case of which can be stated as follows.

\begin{theorem}[Global regular solutions to quasilinear SPDEs -- simplified version]
\label{t:global_quasilinear_intro} Let  $A:\R\to \R^{d\times d}$ be bounded and locally Lipschitz continuous, and for some $\g_0>0$,  $U_0 \colon \Omega \to C^{\g_0}(\T^d) $  be strongly $\F_0$-measurable.
Suppose that there exist $\delta\in (0,1)$ and $\nu>0$   such that 
\begin{equation*}
\forall  {y }\in \R,  x\in \T^d,  {\eta} \in \R^d \quad \text{ it holds that } \quad 
    \textstyle
    A( {y}) {\eta}\cdot {\eta}- \frac{1}{2}\sum_{n\geq 1} |B_n(x) \cdot  {\eta }|^2\,\geq\, \nu | {\eta}|^2.
    \end{equation*}
    as well as 
    $(B_{n})_{n\geq 1}\in C^{3+\delta}(\T^d;\ell^2(\R^d))$.
Then there exists a unique solution $U$ to \eqref{eq:SPDE_quasi_intro} and the latter satisfies 
$$
U\in C_{\loc}^{1/2-,1-}((0,\infty)\times \T^d) \ \text{almost surely}. 
$$
\end{theorem}
We omit the precise Definition \ref{defi:sol} of a solution here. The reader is referred to Section \ref{s:application_to_quasi} and the notation section at the end of the introduction for more information and the convention for function spaces used in the above, respectively. As for the case of linear equations, the resulting space-time continuity of the solution is important when it comes to interpreting the random field $U$ as a real-world object, as well as for its numerical approximation. 
We remark moreover that under more restrictive assumptions on the diffusion matrix $A$, the solution $U$ can even be shown to become \emph{smooth in space}. More precisely, under the assumptions of Theorem \ref{t:global_quasilinear_intro}, for all $\delta>0$, it holds that 
$$
 A\in C^{3+\delta}_{\loc}(\R;\R^{d\times d})
\quad \Longrightarrow \quad
U\in C_{\loc}^{1/2-,(4+\delta)-}((0,\infty)\times \T^d) \ \text{almost surely}.
$$
In particular, if $B$ and $A$ are smooth, then $U$ is smooth in space. In the aforementioned companion paper \cite{ASV25_quasi}, we will prove the above result for much larger classes of systems of quasilinear SPDEs. The reader is referred to \cite[Subsection 2.4]{AS24_thin_film} for a similar result in the case of fourth-order quasilinear SPDEs. However, we emphasize that, even for smooth $A$, the central step in the proof of global existence of smooth solutions to quasilinear SPDEs as \eqref{eq:SPDE_quasi_intro} remains the application of Theorem \ref{t:DNM_intro} to obtain sufficient control of the quasilinearity $U \mapsto \nabla\cdot (A(U)\cdot \nabla U)$, see Subsection \ref{ss:proof_global_ql} for details.

\subsection{Proof strategy for Theorem \ref{t:DNM_intro} and challenges}\label{ss:proof_strat}
Our main  strategy is to transform (more general versions of) the SPDE \eqref{eq:SPDE_lin_intro} into a random PDE using Kunita's \emph{stochastic method of characteristics} \cite{kunita_book}, which, since its introduction, has proven to be useful in many applications, see, e.g.,
\cite{DaPratoTubaro,Tubaro} and the celebrated work \cite{FGP10}.
More precisely, we let $\xi$ be the \emph{stochastic flow of diffeomorphisms} solving the Stratonovich SDE
\begin{equation}\label{eq:eqn_xi_intro}
\dd \xi_{t}(x)\,=\,  - \textstyle{\sum}_{n \ge 1} b_n(\xi_t(x )) \circ \,\dd w^n_t,\qquad \xi_0(x)\,=\, x,
\end{equation}
for each $x\in \R^d$ (where we identify the vector fields $b_n \colon \T^d \to \R^d$ with their periodic extensions to the whole space to apply chain rule type arguments more directly). 
Then the \emph{It\^o--Wentzell formula} yields formally that if $u$ satisfies \eqref{eq:SPDE_lin_intro}, then the composition $v(t,x)=u(t, \xi_t(x))$ solves  the $\omega$-dependent PDE
\begin{equation}\label{eq:eqn_v_intro}
	\partial_t v\,=\, \nabla \cdot  \bigl( \alpha  \nabla v 
		\bigr)   \,+\, \beta \cdot  \nabla  v,
\end{equation}
where the new transport coefficient $\beta $ depends only on  $\xi$ (see the formula
 \eqref{eq:defi_As}), and 
\begin{equation}\label{eq:new_diff_coeff_intro}
\alpha(t,x) \, = \, (D\xi_t(x))^{-1,\top} \Bigl( 
a(t , \xi_t(x) ) \,-\, \tfrac{1}{2} \textstyle{\sum}_{n\ge 1} \bigl(b_n \otimes b_n )(t, \xi_t(x))
\Bigr) (D\xi_t(x))^{-1}
\end{equation}
features in particular the stochastic parabolicity matrix \eqref{eq:stoch_para_intro}, but also the Jacobian of the stochastic flow.
As a key tool to analyze  \eqref{eq:eqn_v_intro}, we derive \emph{regularity of $\xi$} and its inverse, complementing the seminal results from Kunita's theory of stochastic flows of diffeomorphisms \cite{Kunita}. 
As soon as we succeed in this, we can first conclude that $v $ is H\"older continuous using the \emph{deterministic version of the De Giorgi--Nash--Moser estimates}, and in a second step, that also $u(t,x) = v(t,\xi_t^{-1}(x))$ is H\"older continuous.

The above program faces several technical difficulties, which we address in the preliminary Section \ref{Sec:prep_results}: 
Firstly, with our intention to apply the results from \cite{Kunita} to SPDEs of the form \eqref{eq:SPDE_lin_intro}  in mind, we review some generalizations of said theory to  \emph{equations driven by infinitely many Brownian motions} like \eqref{eq:eqn_xi_intro} in Subsections \ref{ss:Rd_flows} and \ref{sec:periodic_flows}. 
To obtain later on also a quantitative version, Theorem \ref{t:DeGiorgiNashMoser_quantitative}, of our main result, we prove, among other things,  an estimate on the H\"older norm of the inverse flow $\xi_t^{-1}$. The method to achieve the latter seems to be {new} and is based on deriving an a-priori type bound for a fixed-point equation in the spirit of the inverse function theorem.
With the necessary properties of $\xi$ at our disposal, we secondly need to rigorously justify the transformation of \eqref{eq:SPDE_lin_intro} into \eqref{eq:eqn_v_intro}. We face the problem that, on the one hand, classical versions of the required It\^o-Wentzell formula assume more regularity of the random field $u$  which we only know to lie in $C_t(L^2_x)\cap L^2_t (H^1_x)$ a-priori, cf.\ \cite[Theorem I.8.3]{Kunita}, \cite[Theorem 3.3.1]{kunita_book} or  \cite[Proposition 2]{Tubaro}, whereas Krylov's version \cite{KrylovIto} tailored to distributional solutions to SPDEs on the other hand assumes the flow to depend additively on the initial value, i.e., that $\xi_t(x) =x + Y_t$ for some stochastic process $Y_t$ that is independent of $x$. Since neither is sufficient for our purposes, we derive in Subsection \ref{ss:ito-wentzell} a new version of the It\^o--Wentzell formula applicable in our situation. Our key ingredient to give  meaning to the  composition of $\xi$ with the \emph{potentially rough right-hand side $\nabla  \cdot (a \nabla u)\in L^2_t(H^{-1}_x)$} is the distributional composition rule
\begin{equation}\label{eq:distr_comp_intro}
\langle 
f\circ \xi, \phi 
\rangle \,=\,\langle 
f ,( \phi\circ \xi^{-1} ) |\det(D\xi^{-1})|
\rangle.
\end{equation}
Additionally, we repeatedly make use of the derived properties of $\xi$ to make sure that emerging factors, e.g., on the right-hand side of \eqref{eq:distr_comp_intro}, can be estimated appropriately. We remark that while we only require periodic stochastic flows and an It\^o--Wentzell formula on $\T^d$ to prove our main theorem, 
we state and provide versions of these auxiliary results on the whole space, as they may be fruitful also for other applications. 
By completing these steps, we are ready to provide a qualitative (Theorem \ref{t:DeGiorgiNashMoser}) and quantitative version (Theorem \ref{t:DeGiorgiNashMoser_quantitative}) of the stochastic De Giorgi--Nash--Moser theorem in Section \ref{app:DeGiorgi_Nash_Moser}. While the former is based on the qualitative aspects of the stochastic flow from the previous sections, the latter relies on corresponding quantitative bounds. We also mention that to deal with \emph{additional lower order terms} in  \eqref{eq:SPDE_lin_intro}, a \emph{subsequent transformation}  of the equation for $v$ is employed, which is similar to the proof of \cite[Theorem 4.2]{HWW17}.

The fact that the transformed diffusion matrix $\alpha$ from \eqref{eq:new_diff_coeff_intro} involves not only the original coefficients but also the inverse Jacobian $\xi$ already indicates why our approach cannot yield a uniform Hölder exponent—neither in the simplified Theorem \ref{t:DNM_intro}, nor in Theorems \ref{t:DeGiorgiNashMoser}--\ref{t:DeGiorgiNashMoser_quantitative}.
Indeed, whenever the flow disproportionally stretches some regions of $\T^d$ while compressing other parts, the \emph{ellipticity ratio} of $\alpha$ may be arbitrarily large on a set of positive probability,
so we can indeed only conclude $v\in C^{\gamma}((0,T)\times \T^d)$ where $\gamma>0$ is a random variable that can potentially be arbitrarily small on a set of positive probability. This phenomenon is further explored in the subsequent Section \ref{sec:unif_constants}, where we provide examples of $(a,(b_n)_{n\ge 1})$  for which the transformed coefficient admits either an \emph{exploding or uniformly bounded ellipticity ratio}. While the former shows that our approach can at best yield results for which the H\"older exponent may degenerate in $\omega$, the latter provides additional assumptions under which it can indeed be chosen uniformly; see Section \ref{sec:unif_constants} for details.

\subsection{Open problems}
\label{ss:open_problems}
The results presented here, while establishing H\"older regularity for the linear SPDE \eqref{eq:SPDE_lin_intro} (or more generally $\eqref{eq:SPDE_lin}$) under regularity conditions for the transport noise coefficients, suggest several open challenges related to De Giorgi--Nash--Moser estimates for parabolic SPDEs.
These challenges primarily stem from the constraints imposed by the stochastic flow method used in our analysis. We detail the most relevant open problems below.

\begin{itemize}
\item {\textsc{The case of rough transport noise}.}
Comparing Theorem \ref{t:DNM_intro} with the classical De Giorgi--Nash--Moser estimates \cite{DeGiorgi_57,moser_60,nash_58}, one can expect that the regularity conditions \eqref{eq:additional_assumption_intro} (or the weaker version in \eqref{Eq:alternative_ass}) can be relaxed or omitted. Indeed, the deterministic case suggests that such estimates should hold only assuming parabolicity and boundedness, i.e., \eqref{eq:stoch_para_intro} and \eqref{eq:bddness_intro}. 
However, the approaches via Harnack inequalities or the De Giorgi method to control oscillation of solutions to \eqref{eq:SPDE_lin_intro} over a decreasing family of balls do not seem to work, see \cite{DG17_continuity} for the case of Harnack inequalities. 
The main obstruction is that the moment estimates available in the stochastic setting are too weak to perform the above-mentioned iteration arguments over decreasing families of balls. 
\item {\textsc{Uniform H\"older exponent}.}
It is unclear whether one can choose a sequence of positive numbers $(\g_m)_{m}$ in Theorem \ref{t:DNM_intro} that is independent of $m$ even in the case of smooth transport noise $(b_n)_n$. Partial results in this direction can be found in Proposition \ref{prop:sufficient_uniform_gamma}.
\item {\textsc{Moment bounds}.}
Our proof of Theorem \ref{t:DNM_intro} by using flows as described in Subsection \ref{ss:proof_strat} does not provide any information on moments of the random variable $\|u\|_{C^\gamma((0,T)\times \T^d)}$, where $\g>0$ might also depend on $\om$. As discussed above, in our approach, the main obstacle is that the transformed matrix coefficient defined in \eqref{eq:new_diff_coeff_intro} might lose the uniform ellipticity. This and the delicate dependence of the H\"older exponent and constants in the De Giorgi--Nash--Moser estimates for \eqref{eq:eqn_v_intro} on the ellipticity ratio (see \cite[Theorem III.10.1]{LSU} and the proof of Theorem \ref{t:DeGiorgiNashMoser}) prevents us from asserting any quantitative bounds on moments of H\"older-type norm for $u$.
\end{itemize}

Let us point out that addressing the first problem in the above list would allow us to strengthen the global regularity results for quasilinear SPDEs (Theorem $\ref{t:global_quasilinear_intro}$), also including SPDEs with \emph{nonlinear} transport noise, see e.g., \cite{bechtel_extrapolation_2025,CSZ19,DG20,FG_ARMA} for examples.

\subsection{Organization of the manuscript}
The rest of this manuscript is separated into four sections: In the following Section \ref{Sec:prep_results}, we collect preparatory results concerning stochastic flows of diffeomorphisms as well as the It\^o--Wentzell formula. 
We start in Subsection \ref{ss:Rd_flows} by reviewing some generalizations of Kunita's theory \cite{Kunita} of $\R^d$-valued stochastic flows to equations driven by infinitely many Brownian motions with coefficients depending only measurably on $(t,\omega)$. Due to the unboundedness of the domain $\R^d$, additional facts regarding $\xi$ and its inverse stated in this subsection are of a qualitative nature. In the following Subsection \ref{sec:periodic_flows} we provide possible improvements for spatially periodic stochastic flows, which will play a role in the proof of the quantitative version of our main result, Theorem \ref{t:DeGiorgiNashMoser_quantitative}. The last preparatory Subsection \ref{ss:ito-wentzell} is dedicated to the proof of the It\^o--Wentzell formula, which is the key ingredient to perform the stochastic characteristic method with a merely weak solution to an SPDE. We remark once more that our version, Proposition \ref{prop:ItoWentzel}, of the It\^o--Wentzell formula is purposely stated on the whole space and may be of interest also for other applications.

We begin the next Section \ref{app:DeGiorgi_Nash_Moser}  with rigorously stating our main results, Theorem \ref{t:DeGiorgiNashMoser} and Theorem \ref{t:DeGiorgiNashMoser_quantitative}, which apply to more general equations than \eqref{eq:SPDE_lin_intro}. 
Their proofs are deferred to the subsequent Subsection \ref{ss:proofs_main}. The whole Section \ref{sec:unif_constants} is dedicated to a more detailed discussion regarding the diffusion matrix of the transformed equation  \eqref{eq:new_diff_coeff_intro}. In Proposition \ref{prop:blow_up_elliptic_ratio}, we provide an explicit example where, with positive probability, the ellipticity ratio of $\alpha$ explodes, even at each space-time coordinate individually. Sufficient conditions to avoid such a blow-up resulting in a strengthened but more restrictive version of Theorem \ref{t:DeGiorgiNashMoser} are provided in Proposition \ref{prop:sufficient_uniform_gamma}. Proofs concerning this section are once more postponed to the following part, Subsection \ref{ss:proofs_uniform_gamma}. Analogously, the more general version  Theorem \ref{t:global_quasilinear} of Theorem \ref{t:global_quasilinear_intro} together with its assumptions is stated in Section \ref{s:application_to_quasi} and its proof is given in Subsection \ref{ss:proof_global_ql}.

\noindent
Before all that, we collect some notation that will be used throughout the manuscript.

\subsection*{Notation}
\emph{General notation.}
The parameter  $T\in (0,\infty)$ will stand for the time horizon and $d\in \N$ for the dimension of the spatial domain. Here and throughout, $\N$ denotes the positive integers, starting from $1$. Moreover, we indicate by $C$ a universal constant, and if it depends on other choices of parameters $(a,b,\dots)$, we write $C_{(a,b,\dots)}$. We remark that, as usual, such constants may vary from line to line, and we sometimes also indicate the relation $A\le C B$ by writing  $A \lesssim B$  or $A \lesssim_{(a,b,\dots) } B$ if the universal constant depends on parameters. The Euclidean $d$-dimensional space is denoted by $\R^d$, and $\T^d$ stands for the $d$-dimensional torus of unit length. By $B_r(y)$ we denote the ball of radius $r$ centered at $y$ in $\R^d$.

\emph{Calculus.}
We make use of the classical notation for differential operators, so $\nabla u$ denotes the gradient of $u$, and for the divergence we write $\nabla \cdot f $. We also write $D f=(\partial_j f^i)_{i,j=1}^d$ for the Jacobian of $f$ and use the partial derivative notation $\partial_i$ and $\partial_t$ to denote the derivative in the direction of the $i$-th spatial coordinate or the time variable. Here, all derivatives may be understood in the weak, distributional, or classical sense, depending on the regularity of the functions involved. We also make use of the \emph{Einstein summation convention}, if necessary: A repeated index in a formula will be summed over. We also recall the topological/geometric notions of a \emph{homeomorphism}, i.e., a continuous bijection of space with continuous inverse, and the notion of a \emph{($C^k$-) diffeomorphism}, i.e., a ($k$-times) continuously differentiable homeomorphism with ($k$-times) continuously differentiable inverse.

\emph{Probability.} Throughout, let $(\O,\A,\P)$ be a probability space with a complete filtration $\F=(\F_t)_{t\geq 0}$, i.e., we assume that $\F_0$ contains all $\P$-null sets. The sequence $(w^n)_{n\ge 1}$ will stand for a sequence of independent  $\F$-Brownian motions, $\P(\Lambda)$ for the probability of an event $\Lambda\in \A$ and $\E[X]$ for the expectation of a random variable $X$. As usual, we write 
\[
\sum_{n\ge 1}  \int_0^t g_{n,s} \dd w_s^n 
\]
for the \emph{It\^o-integral} of a progressively measurable $g\in \ell^2$ against the $\ell^2$-valued cylindrical Wiener process $(w^n)_{n\ge 1}$, see \cite{DPZ,LiuRock} for more information on Hilbert-space valued stochastic integration.

\emph{Stochastic flows}. We will use  the letters $\mu$ and $\sigma = (\sigma_n)_{n\ge 1}$ for coefficients of a \emph{stochastic flow of diffeomorphisms}.  To this end, $\mu:[0,\infty)\times\Omega\times\R^d\to \R^d$ and $\sigma:[0,\infty)\times\Omega\times\R^d\to \ell^2(\R^d)$ are throughout assumed to be $\mathcal{P}\otimes \mathcal{B}(\R^d)$-measurable and $\xi$ will denote the induced flow, i.e., $\xi_{s,t}(x)$ is assumed to satisfy 
\begin{equation}\label{eq:SDE_2}
	\left\{
	\begin{aligned}
		\dd\xi_{s,t }(x) &= \mu_t(\xi_{s,t }(x) )\, \dd t + \displaystyle{\sum_{n\ge 1}}\sigma_{n,t}(\xi_{s,t }(x) )\, \dd w_t^n, \qquad t\in [s,T],  \\
		\xi_{s,s}(x) &= x.
	\end{aligned}
	\right.
\end{equation}
Under suitable assumptions on $(\mu,\sigma)$, $\xi$ exists uniquely and can be modified to be almost surely continuous in $(s,t,x)$, see \cite{Kunita}, in which case we always choose this modification. As this continuous modification is uniquely determined up to a null-set, it will also be the modification which is H\"older continuous or differentiable, whenever such a property holds as well.
The collection of initial and terminal times is denoted by
\begin{equation}\label{Eq33}
\S=	\{ (s,t) \in [0,T]\times [0,T] \,|\, s\leq t\}.
\end{equation}
We abbreviate moreover $\xi_{t} = \xi_{0,t}$ and if applicable denote the inverse of $\xi_{s,t}$ (as a function $\R^d\to \R^d$) by $\Psi_{t,s}$ defined for tuples  $(t,s) $ from $ \invS =\{(t,s)|(s,t) \in \S\} $. Accordingly we set $\Psi_t = \Psi_{t,0} = \xi_t^{-1}$ and  we use $\psi_t$ to denote the inverse  Jacobian $(D\xi_t)^{-1}$.

\emph{Norms and spaces.} We write $L^p(S,\nu;\mathscr{X})$ for the \emph{Bochner space} of strongly measurable, $p$-integrable $\mathscr{X}$-valued functions for a measure space $(S,\nu)$ and a Banach space $\mathscr{X}$. 
If $\mathscr{X}=\R$, we write $L^p(S,\nu)$, and if it is clear which measure we refer to, we also leave out $\nu$. Moreover, if $S$ is countable and equipped with the counting measure, we write $\ell^p(S;\mathscr{X})$ instead of $L^p(S,\nu;\mathscr{X})$ and for $S=\N$ just $\ell^p(\mathscr{X})$. If on the other hand $I=(s,t)$ is an interval and $w$ a density, we write $L^p(s,t;w;\mathscr{X})$ for $L^p(I,w\, \dd t;\mathscr{X})$. In the last Section \ref{s:application_to_quasi}, we use in particular power weights of the form $w_\kappa (t)=|t|^\kappa$ and if $w=1$ we just write $L^p(s,t;\mathscr{X})$. Lastly, in any of the cases above, we denote by $L^p_{\mathfrak{G}}(S,\nu;\mathscr{X})$   the closed subspace of strongly $\mathfrak{G}$-measurable functions in $L^p(S,\nu;\mathscr{X})$.

Regarding \emph{Sobolev spaces} we use the standard notation $W^{k,q}(\mathcal{O})$ for $k\in \N$ and $q\in [1,\infty]$ for either $\mathcal{O}\in \{\R^d,\T^d \}$, and for the Hilbert space case $q=2$ we just write $H^k(\mathcal{O}) = W^{k,2}(\mathcal{O})$. We also recall  the vector-valued  \emph{Bessel potential} scale $H^{s,q}(\mathcal{O};\mathscr{X}) $ for $s\in \R$ and $q\in (1,\infty)$, defined through 
\[
\|f\|_{H^{s,q}(\mathcal{O};\mathscr{X})} \,=\,  \| (1-\Delta)^{s/2} f\|_{L^q(\mathcal{O};\mathscr{X})},
\]
where we write $H^{s,q}(\mathcal{O})$, if $\mathscr{X} =\R$. Also here we set in the Hilbert case $H^{s}(\mathcal{O};\mathscr{X}) = H^{s,2}(\mathcal{O};\mathscr{X})$, which is in line with the previous convention since $H^{k,q} (\mathcal{O}) = W^{k,q}(\mathcal{O})$ with equivalent norms for $k\in \N$.
In Section \ref{s:application_to_quasi} we encounter moreover the periodic \emph{Besov space} $B^{s}_{q,p}(\T^d)$ and refer for its definition to \cite{schmeisser1987topics}.

If $(S,d)$ is a metric space we write $C(S;\mathscr{X})$ for the bounded continuous functions and $C^{\theta}(S;\mathscr{X})$ for the subset of \emph{$\theta$-H\"older continuous functions} for $\theta \in (0,1)$, where we leave $\mathscr{X}$ out again if $\mathscr{X}=\R$.  For $\theta=1$, we write $\Lip (S; \mathscr{X})$ for the bounded and Lipschitz continuous functions with the special case $\Lip(S) =\Lip(S; \R)$. If $\tilde{S}$ is another metric space  we define the \emph{anisotropic H\"older space }
\begin{equation}\label{eq:defi_torus_holder}
	C^{\theta_1,\theta_2}(S\times \tilde{S})\,=\, \bigg\{ f \in C(S\times \tilde{S}) \;\biggr|\;
[f]_{C^{\theta_1 ,\theta_2}(S\times \tilde{S})}
\,=\,
\sup_{t\ne s,\tilde{t}\ne \tilde{s}}
\frac{|f(t,\tilde{t})-f(s,\tilde{s})|}{|t-s|^{\theta_1}+|\tilde{t} -\tilde{s}|^{\theta_2}}\,<\, \infty \bigg\},\quad \theta_1,\theta_2\in (0,1].
\end{equation}
If the underlying spaces are open subsets $S, \tilde{S}\subseteq \mathcal{O}$, where again $\mathcal{O}\in \{\R^d,\T^d \}$, then we also have the spaces of \emph{$k$-times continuously differentiable} and \emph{$(k+\theta)$-H\"older  functions} with their usual norms
\[
\|f\|_{C^{k}(S)} \,=\, \sum_{|\beta|\le k} \| \partial_\beta f\|_{C(S)}, \qquad 
\|f\|_{C^{k+\theta}(S)} \,=\, \|f\|_{C^{k-1}(S)}  \,+\, \sum_{|\beta|= k} \| \partial_\beta f\|_{C^\theta(S)},
\]
for $k\in \N$ and $\theta\in (0,1)$ with suitable modifications in the anisotropic setting.

Finally, we recall the \emph{local version $F_\loc(S)$} of any of the above spaces $F$, which are generally defined as the collection of functions $f$ such that $f|_O \in F(O) $ for any compactly contained, relatively open subset $O\subset S$.  An elegant way to define the topology on local Sobolev spaces on $\R^d$ is laid out at the beginning of Subsection \ref{ss:ito-wentzell}. Similarly, if $p\in \R$ is a parameter in the definition of $F$, we write $F^{p-}(S)$ for the collection of all $f\in \cap_{q<p}F^q(S)$ with the induced locally convex topology.

 \emph{Distributional composition.} Due to its importance in our version of the It\^o--Wentzell formula proved in Subsection \ref{ss:ito-wentzell} we  recall that for any $C^2$-diffeomorphism $\xi\colon \R^d \to \R^d$ one can define the distributional composition with some $f\in H^{-1}_{\loc}(\R^d)$ by
\begin{equation*}
\langle 
f\circ \xi, \phi 
\rangle \,=\,\langle 
f ,( \phi\circ \xi^{-1} ) |\det(D\xi^{-1})|
\rangle_{H^{-1}(B_k(0)) \times {H^1_0(B_k(0))}} ,\qquad \phi\in C_c^\infty(\R^d),
\end{equation*}
for any $k\in \N$ such that $\supp ( \phi\circ \xi^{-1})$ is compactly contained in $B_k(0)$.
We remark that since $( \phi\circ \xi^{-1} ) \det(D\xi^{-1})$ is continuously differentiable it is an element of $C_c^1(B_k(0))$ and the above dual pairing is well-defined.
 
\emph{Stochastic PDE.} In Section \ref{app:DeGiorgi_Nash_Moser} we are concerned with linear stochastic PDE of the form \begin{equation}\label{Eq120}
	\begin{cases}
		\displaystyle{
			\dd u \, - \, Au \,\dd t \, = \, f \, \dd t \,+\, \bigl(B_n u \,+\,g_n\bigr) \,\dd w^n_t,}
	\end{cases}
\end{equation}
where we employ the shorthand notations
\begin{align}\label{Eq20}
	A u &\,= \,\partial_i(a^{ij}\partial_j u) +  a^i \partial_i u + a^0 u,
	\qquad  B_n u \,= \, b^i_n \partial_i u \, + \,b^0_n u,
	\qquad  f  \,=\, f^0 \,+\,  \partial_i f^i.
\end{align}
The variables for the transformed equations will be denoted by $v$ and $z$ and the new coefficients $(\alpha,F,G)$ and $(\alpha,\overline{F})$ will take the respective role of $(a,f,g)$ in \eqref{Eq120}--\eqref{Eq20}.

\section{Preparatory results}
\label{Sec:prep_results}
As indicated in the introductory section, this part of the manuscript is dedicated to providing auxiliary results revolving around stochastic flows of diffeomorphisms and the It\^o--Wentzell formula. The former describes a suitably modified field of solutions $\xi_{s,t}(x)$ to the stochastic differential equations
\eqref{eq:SDE_2}
with initial value $x\in \R^d$ attained at time $s\in [0,\infty)$, and will be the subject of the Subsections \ref{ss:Rd_flows}--\ref{sec:periodic_flows}. Throughout, we assume that the coefficient functions $\mu:[0,\infty)\times\Omega\times\R^d\to \R^d$ and $\sigma:[0,\infty)\times\Omega\times\R^d\to \ell^2(\R^d)$ are $\mathcal{P}\otimes \mathcal{B}(\R^d)$-measurable. The latter notion refers to a version of It\^o's formula for the situation that one composes a semimartingale not with a $(t,x)$-dependent function but with a random field and is investigated in Subsection \ref{ss:ito-wentzell}. 

\subsection{Stochastic flows on $\R^d$}\label{ss:Rd_flows}
In this subsection, we collect some properties of stochastic flows on the whole space $\R^d$ by closely following the seminal work \cite{Kunita}.
  As a start, we consider the situation that \eqref{eq:SDE_2} has Lipschitz continuous coefficients with linear growth, i.e., we impose the following assumption.
\begin{assumption}\label{Ass:Lip}
There exists a constant $C \in (0,\infty)$ such that a.s.
	\begin{align}\begin{split}\label{eq:lipschitz_cond}
	   |\mu_t(x) - \mu_t(y)| \, +\, \|\sigma_t(x) - \sigma_t(y)\|_{\ell^2(\R^d)} \,&\leq \,C|x-y| ,\\
       |\mu_t(x) |\,+\, \|\sigma_t(x)\|_{\ell^2(\R^d)} \,& \leq \,C \bigl(1+|x|\bigr),
	\end{split}\end{align}
	for all $ t\in [0,T]$ and $ x,y\in \R^d$.
\end{assumption}
Classically, the above assumption implies the existence and uniqueness of solutions $\xi_{s,\cdot}(x)$ to \eqref{eq:SDE_2}. 
Additionally, by the Kolmogorov--Chentsov theorem jointly in $(s,t,x)$, one can interpret (suitable modifications of) the solutions as a random flow on $\R^d$ as laid out in \cite{Kunita}.  For later reference, we recall the quantitative version of the Kolmogorov--Chentsov theorem on metric spaces from  \cite[Theorem 8.2]{Cox_vWinden}, see also  \cite[Theorem 1.1]{KU_kolmogorov} for a more general result. 
\begin{theorem}[Kolmogorov--Chentsov]\label{Thm_KC} Let $(M,d_M)$ be a metric space satisfying the following conditions:
\begin{enumerate}[(i)]
    \item \emph{(Finite diameter)}  We have $\diam_M \coloneqq \sup_{x,y \in M} d_M(x,y) <\infty$.
    \item \emph{(Finite Minkowski dimension)} There exist $D\in (0,\infty)$ and $C<\infty$ such that $M$ can be covered by no more than $C r^{-D}$ open balls in $(M,d_M)$  of radius $r$, for any $r\in (0,\diam_M]$.
    \item \emph{(Finite doubling number)} There exist $n\in \N$ such that any open ball in $(M,d_M)$ with radius $r$ can be covered by $n$ open balls in $(M,d_M)$ of radius $r/2$.
\end{enumerate}
    Then for $p \in (D,\infty) \cap [1,\infty)$, $\alpha\in (D/p,1)$ and $\beta \in (0,\alpha - D/p)$ and any measurable
    $Z \colon \Omega\times M \to \R^d$ with 
    \begin{equation}\label{Eq91} C_Z \,\coloneq\, 
    \sup_{x\in M} \E[|Z(x)|^p]^{1/p} \,+\,
    \sup_{x,y \in M} \frac{\E [ |Z(x) - Z(y)|^{\alpha p }]^{1/p} }{d_M(x,y)^\alpha  } \,<\, \infty,
    \end{equation}
    there exists a continuous modification $\tilde{Z}$ of $Z$ satisfying the bound 
    \begin{equation}\label{Eq92}
    \E\biggl[ \sup_{x\in M} |Z(x)|^p \,+\, 
     \sup_{x,y \in M} \frac{|Z(x) - Z(y)|^p }{d_M(x,y)^{\beta p} }
    \biggr]^{1/p}\,
    \lesssim_{ (\alpha,\beta,p, M)}  \, C_Z.
    \end{equation}
\end{theorem}
For us, the key point of the above is that it allows us to deduce a quantitative version of anisotropic Kolmogorov--Chentsov theorems like \cite[Appendix A]{DKN_hitting} and \cite[Theorem I.10.1]{Kunita}. 
Indeed, estimates on solutions to \eqref{eq:SDE_2} typically scale differently in the time and space variable, essentially due to the scaling behavior of the driving Brownian motions.
Thereby, a typical choice of the metric space $M$ would be $M=[0,T]\times V$ for $V\subset \R^d$ bounded, equipped with parabolic distance 
\[
d_M((t,x), (t',x')) \,=\, \max\{ |t-t'|^{1/2} ,|x-x'|\}.
\]
Then $M$ has Minkowski dimension $D=d+2$, since it can be covered by the product of $\sim r^{-2}$ subintervals of  $[0,T]$ with length less than or equal to $r^2$, and $\sim r^{-d}$ Euclidean balls covering $V$. One sees similarly that the doubling number of $(M,d_M)$ is finite. Thereby, a bound on \eqref{Eq91} results in an estimate on the left-hand side of \eqref{Eq92}, which is seen to be equivalent to  
\begin{equation}\label{Eq93}
\E\bigl[ \|
Z
\|_{C^{\beta/2,\beta} ([0,T] \times V;\R^d)}^p\bigr]^{1/p}.
\end{equation}

If, as in \cite{Kunita}, one wants to deduce properties of a random field defined on the whole space $Z\colon\Omega\times [0,T]\times \R^d \to\R^d$, one can of course apply Theorem \ref{Thm_KC} on a sequence of balls $V_1 \subset V_2 \subset \dots $ exhausting $\R^d$. Instead of an estimate  on \eqref{Eq93}, the result is then the existence of a modification $\tilde{Z}$ with $\tilde{Z}\in C_\loc^{\beta/2,\beta} ([0,T] \times \R^d;\R^d)$, $\P$-almost surely.

We proceed to state some of the results obtained in \cite{Kunita} for which we recall the set $\S$ from \eqref{Eq33} consisting of all admissible initial and terminal time instances. The following summarizes \cite[Theorem II.2.2, Theorem II.4.3]{Kunita}. 
\begin{theorem}\label{Thm_homeom}Under Assumption \ref{Ass:Lip} there exist modifications of $\xi_{s,\cdot}(x)$ such that a.s. the following is satisfied:
	\begin{enumerate}[(i)]
		\item The mapping
		\[
		\xi\colon 
		\S \times \R^d \to \R^d ,\quad (s,t,x) \mapsto \xi_{s,t}(x)
		\]
		lies in $C_{\loc}^{1/2-, 1-}(\S\times \R^d;\R^d)$. 
		\item The map 
		\begin{equation*}\label{Eq:homeomorph}
		\xi_{s,t}(\cdot) \colon \R^d \to \R^d
		\end{equation*}
		is a homeomorphism for all $(s,t)\in \S$.
		\item It holds the flow property, i.e.,
		\[
		\xi_{r,t}(x) \,=\, \xi_{s,t} ( \xi_{r,s} (x)),
		\]
		for all $0\le r\le s \le t\le T$ and $x\in \R^d$.
	\end{enumerate}
\end{theorem}
\begin{remark}
We emphasize that in \cite[Chapter II]{Kunita} the situation of finitely many Brownian motions $w^n$ is considered and the coefficients are assumed deterministic and to depend continuously on the time variable. An analysis of the proofs there, which rely on combining Lipschitz and linear growth estimates on the right-hand side of \eqref{eq:SDE_2} with the Kolmogorov--Chentsov theorem from \cite[Theorem I.10.1]{Kunita}, cf.\ Theorem \ref{Thm_KC}, reveals, however, that the results carry over to our situation.
\end{remark}

Under a more restrictive assumption on the coefficient functions, the statements of the previous theorem can be strengthened, see \cite[Theorem II.3.3, Theorem II.4.4]{Kunita}.
\begin{assumption}\label{Ass:Ck_alpha}
Let $k\in \N$ and $\alpha\in (0,1]$. We assume that there exists a constant $C\in (0,\infty)$ such that a.s. for all $t\in [0,T]$ and multiindices $0\le |\gamma|\le k$ the derivatives $\partial_\gamma \mu_t$ and $\partial_\gamma\sigma_t$  exist and satisfy
\begin{align}\begin{split} \label{Eq23}
|\partial_\gamma\mu_t(x) - \partial_\gamma\mu_t(y)|+ \|\partial_\gamma\sigma_t(x) - \partial_\gamma\sigma_t(y)\|_{\ell^2(\R^d)}\,&\leq\, C|x-y|^\alpha, \qquad |\gamma| =k,
\\
|\partial_\gamma \mu_t(x) | \,+\, \|\partial_\gamma \sigma_t(x)\|_{\ell^2(\R^d)} \,&\leq \,C,\qquad \qquad\quad\; \,1\le  |\gamma| <k,
\\
       |\mu_t(x) | \,+\, \|\sigma_t(x)\|_{\ell^2(\R^d)} \,&\leq \,C \bigl(1+|x|\bigr),
\end{split}
\end{align}
for  all $x,y\in \R^d$.
\end{assumption}

\begin{theorem}\label{Thm:Ck}
	We assume that Assumption \ref{Ass:Ck_alpha} holds for $k\in \N$ and $\alpha\in (0,1]$ and consider the modifications of $\xi_{s,\cdot}(x)$ from Theorem \ref{Thm_homeom}. Then, a.s., the following holds:
	\begin{enumerate}[(i)]
		\item \label{Item:regularity} The mapping 
		\[\xi\colon 
		\S \times \R^d \to \R^d ,(s,t,x) \mapsto \xi_{s,t}(x)
		\]
		lies in $C_{\loc}^{\alpha/2-, (k+\alpha)-}(\S\times \R^d; \R^d)$.
		\item \label{Item:Ck_diff} The map 
		\[
			\xi_{s,t}(\cdot) \colon \R^d \to \R^d
		\]
		is a $C^k$-diffeomorphism for all $(s,t)\in \S$.
	\end{enumerate}
\end{theorem}

\begin{remark}
	 As for Theorem \ref{Thm_homeom}, the fact that we are dealing with infinitely many Brownian motions and coefficients which are random and only measurable in time does not affect the analysis, and hence \cite[Theorem II.4.4]{Kunita} yields \eqref{Item:Ck_diff}. Concerning \eqref{Item:regularity} some additional remarks are in order:
	The spatial regularity is stated in  \cite[Theorem II.3.3]{Kunita}. To obtain also the temporal regularity, we consider first the situation that $k=1$, in which case the spatial regularity is shown in \cite[Theorem II.3.1]{Kunita} by applying the Kolmogorov--Chentsov theorem \cite[Theorem I.10.1]{Kunita}, cf.\ Theorem \ref{Thm_KC},  to the random field
	\begin{align}\label{eq:defi_eta}
		\eta_{s,t}(x,y) =  \frac{1}{y} \bigl(
		\xi_{s,t}(x+ye_l) - \xi_{s,t}(x)
		\bigr),\qquad y\,\ne\, 0,
	\end{align}
to deduce that it admits a.s.\ a continuous extension at $y=0$. For this
the estimate 
\begin{align}\begin{split}\label{Eq_est_eta}&
	\E \bigl[
	\bigl|
	\eta_{s,t}(x,y) - \eta_{s',t'}(x',y')
	 \bigr|^p
	\bigr]
	\\&\quad 
	\lesssim_{(\alpha,p, C_{\mathrm{reg}}, T)}
	|x-x'|^{\alpha p} +|y-y'|^{\alpha p } +(1+|x|+|x'|)^{\alpha p} \bigl(
	|s-s'|^{\alpha p /2} +|t-t'|^{\alpha p /2}
	\bigr),
    \end{split}
\end{align}
for $p\in (2,\infty)$
is proved in \cite[Lemma II.3.2]{Kunita}, where $C_{\mathrm{reg}}=C$ is the constant for which the conditions \eqref{Eq23} in Assumption \ref{Ass:Ck_alpha} hold. 
 Since $p$ can be chosen large, the Kolmogorov--Chentsov theorem  yields the existence of versions of $\eta_{s,t}(x,y)$ such that a.s.\
\[
\eta \in C_{\loc}^{\beta/2, \beta, \beta}(\S\times \R^d\times \R ),\qquad \beta<\alpha.
\]
In particular, $\eta_{\cdot, \cdot}(\cdot , 0)$ lies in $C_{\loc}^{\beta/2, \beta}(\S\times \R^d )$ for each $\beta<\alpha$ resulting in the claimed regularity $\xi \in C_{\loc}^{\alpha/2-,(1+\alpha)-}(\S\times \R^d)$ of $\xi$.
For $k>1$, the same line of argument is applied to the derivative processes $\partial_{\gamma} \xi$. In particular, the Kolmogorov--Chentsov theorem yields again temporal regularity as a byproduct, resulting in the asserted $\xi \in C_{\loc}^{\alpha/2-, (k+\alpha)-}(\S\times \R^d)$.
\end{remark}

Later we also require  H\"older continuity of the backwards flow $\Psi_{t,s}$ defined by the inverse mappings $\xi_{s,t }^{-1}\colon \R^d\to \R^d$, for $(t,s) \in \invS$. We remark that at least for $(t,s)$ fixed, spatial regularity follows from Theorem \ref{Thm:Ck} \eqref{Item:Ck_diff} under Assumption \ref{Ass:Ck_alpha}, so that the point of the following result is to also deduce H\"older continuity in the time variables.

\begin{proposition}\label{Prop_Psi}
	We assume that Assumption \ref{Ass:Ck_alpha} holds for $k=1$ and $\alpha\in (0,1]$. Then, $\P$-a.s., the inverse flow 
	\[
	\Psi \colon \invS \times  \R^d \to \R^d, (t,s,y) \mapsto \Psi_{t,s}(y) \coloneq \xi_{s,t}^{-1}(y)
	\]
	lies in $C_{\loc}^{\alpha/2-,1}(\invS\times \R^d ;\R^d)$.
\end{proposition} 
\begin{remark}
    In the deterministic case, the natural barrier $1/2$ for the temporal regularity of $\xi$ is not present, and thus the argument below can be significantly simplified. 
    Indeed, if $\xi_{s,t}$ is for all times $(s,t)$ a diffeomorphism, then the Jacobian $D\xi_{s,t}(x)$ is invertible for all $(s,t,x)\in \S\times \R^d$. But if it is also continuously differentiable in $(s,t)$, then the continuously differentiable  $(s,t,x)\mapsto (s,t,\xi_{0,t}(x))$  has invertible Jacobian 
    \[
   \left( \begin{matrix}
        1 &0 & 0\\
        0 &1 & 0\\
        \partial_s \xi_{s,t}(x)  &
        \partial_t \xi_{s,t}(x) & D \xi_{s,t}(x)
    \end{matrix}\right)
    \]and therefore is a diffeomorphism itself with inverse $(s,t,y)\mapsto (s,t,\Psi_{t,s}(y))$. 
    As a result we readily  conclude that $\Psi\in C_{\loc}^{1,1}(\invS\times \R^d;\R^d)$ in the deterministic setting. 
    As soon as the temporal regularity of $\xi$ is, however, less than $1$, this line of argument breaks down since H\"older regularity is not preserved under taking inverses (consider, e.g., $[0,1] \mapsto [0,1], t\mapsto t^2$). 
    
    Alternatively, one can also use in the deterministic case that the inverses $\Psi_{t,s}$ satisfy the same equation backwards in time, which  transfers to the stochastic setting if the coefficients do not depend on $\omega$, as is shown in \cite{BW_Kunita,kunita_jump_book}.  If the latter is not satisfied, however, issues related to adaptedness arise so that neither of these methods are applicable in our situation. We resolve the situation in the proof below by instead deriving an a priori-type estimate on the inverse flow $\Psi$. This approach can even be made quantitative in the periodic setting, see Proposition \ref{cor:quant_inverse} below.     
\end{remark}
\begin{proof}[Proof of Proposition \ref{Prop_Psi}]
	We essentially aim to show that as a consequence of the H\"older continuous dependence of the homeomorphisms $\xi_{s,t}$ on $(s,t)$, also their inverses depend H\"older continuously on these variables. As this is essentially a deterministic argument,  we allow throughout the proof all constants to depend on $T$ and  $\omega$, which we choose such that all almost sure properties of $\xi$ that we use hold.  The only additional property of $\xi$ that we need is that, $\P$-a.s.\ the inverse Jacobian
	\begin{equation}\label{Eq15}(
	D \xi_{(\cdot,\cdot)})^{-1}(\cdot) \colon \S\times \R^d \to \R^{d\times d}
	\end{equation}
	depends continuously on $(s,t,x)$,  as shown in the last step in the proof of \cite[Theorem II.4.4]{Kunita}. Let us also remark that $\Psi$ has the backwards flow property that 
	\begin{equation}\label{BW_flow}
		\Psi_{t,r}(y) \,=\, \Psi_{s,r}   (\Psi_{t,s}(y) ),\qquad 0\le r\le s \le t\le T,\,y\in \R^d,
	\end{equation}
	as a consequence of the flow property of $\xi$. Applying $\xi_{0,s}$ to the above with $r=0$ yields 
	\[
	\Psi_{t,s}(y) \,=\, \xi_{0,s} (\Psi_{t,0}(y)),
	\] 
	and thereby it suffices to show that $\Psi_{\cdot,0 }\in C_{\loc}^{\alpha/2-,1}([0,T]\times \R^d ;\R^d)$ by the regularity of $\xi$ stated in Theorem \ref{Thm:Ck}. To this end, we recall the notation $\Psi_t = \Psi_{t,0} =\xi_t^{-1}$ and for later use that
	\begin{equation}\label{Eq_Felix}
		\Psi_t ( \xi_{s,t}(x)) \,=\, \Psi_s(x),
	\end{equation}
	as 
	follows from \eqref{BW_flow} with $r=0$ evaluated at $ \xi_{s,t}(x)$. We first show that $\Psi_t(y)$ is continuous in $(t,y)$ and then the desired H\"older continuity in the following three steps.

	\emph{Step 1 (Continuity at $(t,y)$---Reduction). } 
	Now we consider the one-parameter family $\Psi_t$ evaluated at two time instances $t$ and $t'$ and consider first the case that $t'\le t$. 
	Then, for any $y,y'\in \R^d$, we obtain from \eqref{Eq_Felix} that
	\[
	\Psi_{t'}(y') - \Psi_{t}(y) = \Psi_{t}( \xi_{t',t}(y'))  - \Psi_{t}(y ).
	\]
	The latter converges to $0$ as $(t',y')\to (t,y)$ by the space-time continuity of $\xi$ stated in Theorem \ref{Thm:Ck} \eqref{Item:regularity}, the fact that $\xi_{t,t}=\id$ and the continuity of $\Psi_t$ in the space variable.
	If on the contrary  $t\le  t'$, we deduce 
	\[
	\Psi_{t'}(y') - \Psi_{t}(y) \,=\,	\Psi_{t}  ( \Psi_{t',t}(y') ) -  \Psi_{t} (y),
	\]
	from \eqref{BW_flow} 
	and hence it suffices to show that  $ \Psi_{t',t}(y') \to  y$ as $(t',y')\to (t,y)$ due to continuity of $\Psi_t$ in space.
	To this end, we let $\epsilon>0$ and claim that there exists a $\delta>0$ such that the mapping 
	\begin{equation}\label{Eq2}
	f_{t',y'}\colon \widebar{B_{\epsilon/2}(y)} \to \widebar{B_{\epsilon/2}(y)}, \quad x \mapsto x  - (\xi_{t,t'}(x) -y')
	\end{equation}
	is a Lipschitz contraction
	as soon as $|t'-t|  + |y' - y|<\delta$. This implies, by the contraction mapping principle, that it admits a fixed point, which must coincide with $\Psi_{t',t}(y')$ 
	by the definition of $f_{t',y'}$. As a result, it follows that $|\Psi_{t',t}(y')-y|\le \epsilon/2 <\epsilon$ for $|t'-t|  + |y' - y|<\delta $, and by the above considerations the continuity of $\Psi$.
	
	\emph{Step 2 (Continuity at $(t,y)$---Claims on \eqref{Eq2}).}
	It remains to verify that $f_{t',y'}$ is a contraction and maps  $\widebar{B_{\epsilon/2}(y)}$ into itself for 
	\[
	|t'-t|  + |y' - y|<\delta,
	\]
	where we specify $\delta>0$ below. Regarding the contractivity, we notice that
	\begin{align}\begin{split}
			\label{Eq3}
		&|f_{t',y'} (x) - f_{t',y'} (x')| 
		\\&\quad  = \bigl|
		x-x' - (\xi_{t,t'}(x) - \xi_{t,t'}(x'))
		\bigr|
			\\&\quad 
			 = \biggl|
			x-x' - \int_{0}^1 D\xi_{t,t'}(x' + r(x-x'))(x-x') \, \dd r
			\biggr|
				\\&\quad 
			= \biggl| \int_{0}^1\bigl( D\xi_{t,t'}(x' + r(x-x')) - \id \bigr)(x-x')\, \dd r
			\biggr|,
		\end{split}
	\end{align} for $x,x'\in \R^d $
	by the fundamental theorem of calculus and the differentiability of $\xi$ in space, cf.\ Theorem \ref{Thm:Ck} \eqref{Item:Ck_diff}.
	Restricting ourselves to $x,x' \in \widebar{B_{\epsilon/2}(y)}$, we use that 
	$\xi\in C_{\loc}^{\beta/2, (1+\beta)}(\S\times \R^d)$
	for $\beta<\alpha$ by Theorem \ref{Thm:Ck} \eqref{Item:regularity} to deduce that
	\[
	\bigl| D\xi_{t,t'}(x' + r(x-x')) - \id \bigr| \le C_{\epsilon,y} |t'-t|^{\beta/2} 
	\]
	Inserting this in \eqref{Eq3} results in the estimate
	\[
	|f_{t',y'} (x) - f_{t',y'} (x')| \le C_{\epsilon,y} |t'-t|^{\beta/2}  |x-x'|,
	\]
	which becomes a contraction estimate, e.g., for $|t'-t|<( 2 C_{\epsilon,y})^{-2/\beta}$.
	To assure also the required mapping property $f_{t',y'}\colon \widebar{B_{\epsilon/2}(y)} \to \widebar{B_{\epsilon/2}(y)}$, we calculate
	\[
	|f_{t',y'}(y)-y| \,=\, |\xi_{t,t'}(y) - y'|
	\,\le\, |\xi_{t,t'}(y) - y| + |y-y'| \,\le\, C_{y}
	|t'-t|^{\beta/2} + |y-y'|.
	\]
	Whenever $|t-t'|< (\epsilon/(8C_{y}))^{2/\beta}$ and $|y-y'|< \epsilon/8$, the above can be estimated by $\epsilon/4$.
	Finally, we set 
	\[
	\delta =\min \bigl\{
	( 2 C_{\epsilon,y})^{-2/\beta}, (\epsilon/(8C_{\epsilon,y}))^{2/\beta}, \epsilon/8
	\bigr\},
	\]
	so that all the previous estimates hold. Then 
	\[
	|f_{t',y'}(x)-y|\, \le\, |f_{t',y'}(x)-f_{t',y'}(y)|\, +\, |f_{t',y'}(y) - y|\, <\, \tfrac{1}{2}|x-y| \,+ \,\epsilon/4 \,<\, \epsilon/2
	\]
	for $x\in \widebar{B_{\epsilon/2}(y)}$ and therefore $f_{t',y'}(x)\in \widebar{B_{\epsilon/2}(y)}$. We conclude that $\Psi$ is continuous at $(t,y)$ as desired. 
	
	\emph{Step 3 (H\"older continuity).}
	As a last step, we show that even $\Psi\in C_{\loc}^{\alpha/2-,1}([0,T]\times \R^d)$ and let $(t_0,y_0)\in [0,T]\times \R^d$ so that it suffices to show that there exists $\delta>0$ with \begin{equation}\label{Eq9}\Psi \in C^{\beta/2, 1}(\zeta_\delta(t_0,y_0)),\quad \zeta_\delta(t_0,y_0) = {[0\vee( t_0-\delta), T\wedge (t_0+\delta)] \times B_\delta(y_0)}, 
	\end{equation}
	for all $\beta \in (0,\alpha)$. We allow the parameter $\delta>0$ to depend on $(t_0,y_0)$ since the H\"older constants can be chosen uniformly on bounded subsets of $[0,T]\times \R^d$ by compactness. For now we just fix some $\delta_0>0$ and remark that $\Psi (\zeta_{\delta_0}(t_0,y_0))\subset \R^d$ is bounded by the continuity of $\Psi$ shown in the previous steps.  
    Recalling that \eqref{Eq15} depends continuously on $(s,t,x)$, we obtain for $(t,y), (t',y') \in \zeta_{\delta_0}(t_0,y_0)$ and $r\in [0,1]$ that
	\begin{align}&\label{Eq4}
		|(D\xi_{0,t})^{-1}(\Psi_t(y))| \le C_{{\delta_0},t_0,y_0}
		\\&\label{Eq6}
		|D\xi_{0,t}(\Psi_{t'}(y') 
		+r(\Psi_{t}(y) - \Psi_{t'}(y'))
		) - D\xi_{0,t}(\Psi_t(y))  | \le  C_{{\delta_0},t_0,y_0}|\Psi_{t'}(y') - \Psi_t(y)|^\beta
		\\&\label{Eq5}
		|\xi_{0,t}(\Psi_{t'}(y')) - \xi_{0,t'}(\Psi_{t'}(y'))| \le C_{{\delta_0},t_0,y_0}|t'-t|^{\beta/2}
	\end{align}
	for a constant $C_{{\delta_0},t_0,y_0}$ when invoking also that $\xi\in C_{\loc}^{\beta/2, (1+\beta)}(\S\times \R^d)$, cf.\
	 Theorem \ref{Thm:Ck} \eqref{Item:regularity}. 
	 To proceed, we trivially write 
	 \begin{align*}&
	 	\Psi_{t}(y)
	 	= - (D\xi_{0,t})^{-1}(\Psi_t(y))
	 	\bigl[
	 	\xi_{0,t}(\Psi_t(y)) -y  - D\xi_{0,t}(\Psi_t(y))[\Psi_t(y)]
	 	\bigr],
	 	\\&
	 			\Psi_{t'}(y')
	 		= - (D\xi_{0,t})^{-1}(\Psi_t(y))
	 		\bigl[
	 		\xi_{0,t'}(\Psi_t'(y')) -y'  - D\xi_{0,t}(\Psi_t(y))[\Psi_{t'}(y')]
	 		\bigr],
	 \end{align*}
 	where $(D\xi_{0,t})^{-1}(\Psi_t(y))$ is the evaluation of  \eqref{Eq15} at $\Psi_t(y)$, i.e., the inverse of the matrix $D\xi_{0,t}(\Psi_t(y))$, and the squared brackets stand for the application of a matrix to a vector. Accordingly, we  expand
 	\begin{align}\begin{split}\label{Eq8}
 			&
 		\Psi_t(y) - \Psi_{t'}(y')\, =\,-
 		 (D\xi_{0,t})^{-1}(\Psi_t(y))
 		\bigl[
 		(y'-y) +(
 		\xi_{0,t}(\Psi_{t'}(y')) - \xi_{0,t'}(\Psi_{t'}(y'))
 		) 
 		\bigr]
 		\\&
 		-  (D\xi_{0,t})^{-1}(\Psi_t(y))
 		\bigl[
 		\xi_{0,t}(\Psi_{t}(y)) - \xi_{0,t}(\Psi_{t'}(y'))
 		- D\xi_{0,t}(\Psi_t(y)) [\Psi_{t}(y) - \Psi_{t'}(y') ]
 		\bigr]
 		\\& = \I_1 + \I_2 
 		\end{split}
 	\end{align}
 	and by \eqref{Eq4} and \eqref{Eq5}, we can estimate
 	\begin{equation}\label{Eq7}
 	|\I_1| \le 
 	C_{{\delta_0},t_0,y_0}\bigl(
 	|y'-y| + C_{{\delta_0},t_0,y_0}|t'-t|^{\beta/2}
 	\bigr).
 	\end{equation}
 	For the remaining term, we apply the fundamental theorem of calculus, \eqref{Eq4} and \eqref{Eq6} to bound
 	\begin{align*}
	|\I_2|
	&
	 \le C_{{\delta_0},t_0,y_0} \biggl|\int_0^1 \bigl(
	D\xi_{0,t}(\Psi_{t'}(y')+r (\Psi_t(y) -\Psi_{t'}(y'))) - 
	D\xi_{0,t}(\Psi_{t}(y)) \bigr)[\Psi_{t}(y) - \Psi_{t'}(y') ]
	\dd r
	\biggr|
	\\&
	\le C_{{\delta_0},t_0,y_0}^2 | 
	\Psi_{t}(y) - \Psi_{t'}(y')
	|^{1+\beta}.
 	\end{align*} 
	By the continuity of $\Psi$, we can choose $\delta\in (0,\delta_0)$ such that
	\[
	\Psi(\zeta_{\delta}(t_0,y_0)) \subset B_{\epsilon_0}(\Psi_{t_0}(y_0)),\quad 
	\epsilon_0 = \frac{1}{2}\bigl( 2C_{{\delta_0},t_0,y_0}^2 \bigr)^{ -1/\beta},
	\]
	and hence 
	\[
	| 
	\Psi_{t}(y) - \Psi_{t'}(y')
	|^{\beta} \le \bigl({2 C_{{\delta_0},t_0,y_0}^2}
	\bigr)^{-1}\]
	if additionally $(y,t), (y',t') \in \zeta_{\delta}(t_0,y_0)$. In this case, we find
	\[
	|\I_2| \le \tfrac{1}{2} 	| 
	\Psi_{t}(y) - \Psi_{t'}(y')
	|
	\]
	and inserting also \eqref{Eq7} in \eqref{Eq8}, we deduce that
	\[
	\tfrac{1}{2} 	| 
	\Psi_{t}(y) - \Psi_{t'}(y')
	| \le 	C_{{\delta_0},t_0,y_0}\bigl(
	|y'-y| + C_{{\delta_0},t_0,y_0}|t'-t|^{\beta/2}
	\bigr).
	\]
	Hence, \eqref{Eq9} follows, and the proof is complete.
\end{proof}
\subsection{Periodic stochastic flows}\label{sec:periodic_flows} With our goal in mind to apply the theory laid out in the last subsection to stochastic PDEs on a periodic domain, we collect possible improvements in the periodic setting. Most notably due to the compactness of $\T^d$, the qualitative results from Theorem \ref{Thm_homeom}, Theorem \ref{Thm:Ck}, and Proposition \ref{Prop_Psi} can be cast in quantitative form.  
To this end, we identify coefficient functions for a stochastic flow on $\T^d$ with their periodic extensions satisfying
\begin{align}\label{eq:periodicity_assumption} \mu_t(x+j)\, =\, \mu_t(x),\qquad \sigma_t(x+j) \,=\,\sigma_t(x),\qquad j\in \Z^d,\, x\in \R^d.\end{align} 
Under Assumption \ref{Ass:Lip}, by Theorem \ref{Thm_homeom}, there exists then a stochastic flow of homeomorphisms, which is periodic itself  due to pathwise uniqueness for \eqref{eq:SDE_2}: a.s.,
\begin{equation*}\label{eq:periodic_xi_copy}
		\xi_t(x+j) \,=\, \xi_t(x) + j, \qquad j\in \Z^d,\,x\in \R^d, \, t\in [0,T].
\end{equation*}
Thereby, it induces a.s.\ a homeomorphism of the torus $\xi_t\colon \T^d \to\T^d$, for all $t\in [0,T]$.
If moreover Assumption \ref{Ass:Ck_alpha} holds for $k\ge 1$ we have additionally that $
D\xi(x) = D\xi(x+j)$ for $j\in \Z^d$ and all its higher derivatives. Due to periodicity of $\xi$, it suffices in the following to estimate its H\"older seminorm on a sufficiently large bounded subset of $V\subset \R^d$.
\begin{proposition}\label{prop:periodic_flows_I}
    Let Assumption \ref{Ass:Lip} and the periodicity condition \eqref{eq:periodicity_assumption} be satisfied, 
    and $C_{\mathrm{reg}}$ be a constant such that \eqref{eq:lipschitz_cond} holds. 
    Then the stochastic flow of diffeomorphisms satisfies
    the bound 
    \begin{equation}
        \E \bigl[
       \| \xi \|_{C^{\beta/2,\beta}([0,T]\times V; \R^d)}^p
        \bigr]\,\le\, K_{{(\beta, C_{\mathrm{reg}}, p,V, T)}},
    \end{equation}
    for each $\beta\in (0,1)$, $p\in (1,\infty)$ and bounded $V\subset \R^d$.
\end{proposition}
\begin{proof}  We observe that the periodicity condition \eqref{eq:periodicity_assumption} together with linear growth actually implies boundedness. Therefore, 
recalling the improved  bound 
\begin{equation}\label{Eq25}
\E \bigl[
|\xi_{s,t}(x) -\xi_{s',t'}(x')|^p
\bigr]\,\lesssim_{(C_{\mathrm{reg}},  p,T)}\, |x-x'|^p \,+\, |s-s'|^{p/2}\,+\, |t-t'|^{p/2}
\end{equation}
from the remark after \cite[Theorem II.2.1]{Kunita}, it only remains to apply Theorem \ref{Thm_KC} as laid out in the comments preceding \eqref{Eq93}.
\end{proof}
\begin{proposition}
    \label{prop:periodic_flows}
      Let Assumption \ref{Ass:Ck_alpha} with $k=1$ and $\alpha\in (0,1]$ and the periodicity condition \eqref{eq:periodicity_assumption} be satisfied, 
    and $C_{\mathrm{reg}}$ be a constant such that \eqref{Eq23} holds. Then the stochastic flow of diffeomorphisms $\xi_t = \xi_{0,t}$ satisfies the bounds 
    \begin{align}\label{Eq26}
     \E \bigl[
       \| D\xi \|_{C^{\beta/2,\beta}([0,T]\times V; \R^{d\times d})}^p
        \bigr]\,&\le\, K_{{(\beta, C_{\mathrm{reg}}, p,V, T)}},
        \\
        \label{Eq28}
        \E \bigl[
       \| (D\xi(\cdot) )^{-1}\|_{C^{\beta/2,\beta}([0,T]\times V; \R^{d\times d})}^p
        \bigr]\,&\le\, K_{{(\beta, C_{\mathrm{reg}}, p,V, T)}},
    \end{align}
    for each $\beta\in (0,\alpha)$, $p\in (0,\infty)$ and bounded $V\subset \R^d$. {Moreover, if Assumption \ref{Ass:Ck_alpha} holds also for $k>1$, then additionally to \eqref{Eq26} we have 
    \begin{equation}\label{eqn:xi_higher_der}
        \E \bigl[
       \| D^k\xi \|_{C^{\beta/2,\beta}([0,T]\times V; (\R^{d})^{\otimes k})}^p
        \bigr]\,\le\, K_{{(\beta, C_{\mathrm{reg}}, p,k,V, T)}}.
    \end{equation}
   }
\end{proposition}
\begin{proof}
Inspecting the proof of \cite[Lemma II.3.2]{Kunita} we find that the terms $|x|$ and $|x'|$ on the right-hand side of \eqref{Eq_est_eta} enter via bounds on differences of $\xi$ itself, which by the boundedness of the coefficients under the above assumptions satisfies the improved estimate \eqref{Eq25}. Therefore, we have for $\eta$ defined in \eqref{eq:defi_eta} that
\begin{align}\begin{split}\label{Eq50}&
	\E \bigl[
	\bigl|
	\eta_{s,t}(x,y) - \eta_{s',t'}(x',y')
	 \bigr|^p
	\bigr]
	\\&\quad 
	\lesssim_{(\alpha,p, C_{\mathrm{reg}}, T)}
	|x-x'|^{\alpha p} +|y-y'|^{\alpha p } \,+\, 
	|s-s'|^{\alpha p /2} +|t-t'|^{\alpha p /2} ,
    \end{split}
\end{align}
which by Theorem \ref{Thm_KC} and the comments before \eqref{Eq93}  yields already \eqref{Eq26}. 

Next, we recall that the proof of H\"older continuity of \eqref{Eq15} from \cite[Theorem II.4.4]{Kunita} relies on the linear equation
\begin{align*}
 (D\xi_{t}(x))^{-1} \,=\, &\id_{\R^d} \,-\, \int_0^t (D\xi_{s}(x))^{-1}D\mu_{s}( \xi_{s}(x))\,\dd s\\& -\, \frac{1}{2}\sum_{n\ge 1}
\int_0^t (D\xi_{s}(x))^{-1}D\sigma_{n,s}(\xi_{s}(x))D\sigma_{n,s}(\xi_{s}(x))\,\dd s\\ &  -\, \sum_{n\ge 1} \int_0^t (D\xi_{s})^{-1}(x)D\sigma_{n,s}(\xi_{s}(x)) \,\dd w^n_s,
\end{align*}
satisfied by the inverse Jacobian of $\xi$. 
To proceed, it suffices then to use the regularity of the coefficients, for which in our special case holds the improved \eqref{Eq25}, resulting in the estimate 
\begin{align}\begin{split}&
	\E \bigl[
	\bigl|
	(D\xi_t(x))^{-1} - (D\xi_{t'}(x'))^{-1}
	 \bigr|^p
	\bigr]
	\,
	\lesssim_{(\alpha,p, C_{\mathrm{reg}}, T)}\, 
	|x-x'|^{\alpha p}\,+\, |t-t'|^{\alpha p /2} ,
    \end{split}
\end{align}
again without any $|x|$ or $|x'|$-dependent constants. Using Theorem \ref{Thm_KC} we obtain \eqref{Eq28}.

{Lastly, regarding \eqref{eqn:xi_higher_der}, we 
observe that the eliminated $x$-dependence in the estimate  
\eqref{Eq50} compared to the previous \eqref{Eq_est_eta}
can be iterated when obtaining estimates on the higher order derivatives as in \cite[Theorem II.3.3]{Kunita}.
}
\end{proof}
\begin{proposition}\label{cor:quant_inverse}
    Under the assumptions of Proposition \ref{prop:periodic_flows} we have
    \begin{align}\label{Eq111}
     \E \bigl[
       \| \Psi \|_{C^{\beta/2,1}([0,T]\times V; \R^{ d})}^p
        \bigr]\,&\le\, K_{{(\beta, C_{\mathrm{reg}}, p,V, T)}},
    \end{align}
    for each $\beta\in (0,\alpha)$, $p\in (0,\infty)$ and bounded $V\subset \R^d$, where  $\Psi_t = \xi_t^{-1}$ is the inverse stochastic flow.
\end{proposition}
\begin{remark}
    Before giving the proof of the above proposition, let us comment on the measurability of $\Psi$, which can be a subtle issue when considering a general, random inverse mapping. In our situation, however, due to the completeness of $\F$ and the continuity properties stated in Theorem \ref{Thm_homeom} and Proposition  \ref{Prop_Psi}, we can argue that $\Psi(x) \colon \Omega \times [0,T] \to \R^d$ defines a progressive process for all $x\in \R^d$. Indeed, by the completeness of $\F$, we can consider versions of the inverses $\xi_t$ and $\Psi_t$, which are continuous in $(t,x)$ for every $\omega\in \Omega$. Thereby, to deduce that $\Psi(x)$ is progressive, we only need to show adaptedness. To this end, we let $(q_k)_{k\in \N}$ be a numeration of $\Q^d$, and we define $\F_t$-measurable mappings $Y_{N} \colon \Omega \to \R^d$ by setting 
    \[Y_N(\omega) \,=\, q_{k_N^*(\omega)}
    \qquad \text{for}\qquad
   k_N^*(\omega) \,=\, \min \bigl\{
    k\in \N \,\big| \, 
    |\xi_t(\omega, q_k) -x |\,<\, 1/N
    \bigr\}.
    \]
    Using that $\Psi_t   = \xi_t^{-1}$, we deduce moreover
    \[Y_N(\omega)\,-\,
    \Psi_t(\omega,x)  \,=\,\Psi_t(\omega,\xi_t (\omega, q_{k_N^*(\om)})) \,-\, \Psi_t(\omega,x) ,
    \]
    and the right-hand side tends to $0$ as $N\to\infty$, by continuity and the definition of $Y_N$. Therefore, we see that the $\omega$-wise limit $\Psi_t(x)$ of $\F_t$-measurable mappings is $\F_t$-measurable itself, as desired. 
    Using the space-time continuity of $\Psi$ stated in Proposition \ref{Prop_Psi}, we see in particular that the left-hand side of \eqref{Eq111} is the expectation of a random variable.
\end{remark}
\begin{proof}[Proof of Proposition \ref{cor:quant_inverse}]
The improved assertions from the previous Propositions \ref{prop:periodic_flows_I}--\ref{prop:periodic_flows} allow us to also cast the proof of Proposition \ref{Prop_Psi} into a quantified form. To this end, we let $M\ge 1$ and restrict ourselves for now to the event 
\begin{align} \begin{split}
\Lambda_M \,=\, \Bigl\{&
\|\xi\|_{C^{\beta/2,\beta} ([0,T]\times Q; \R^d) } \vee 
\|D \xi\|_{C^{\beta/2,\beta} ([0,T]\times Q; \R^{d\times d}) } 
\\&\vee 
\|(D \xi(\cdot ))^{-1}\|_{C^{\beta/2,\beta} ([0,T]\times Q; \R^{d\times d}) } \,\le\, \sqrt{M}
\Bigr\},\end{split}
\end{align}
where $Q=[-1,2)^d$ is sufficiently large to contain the torus.
Following the third step of the proof of Proposition \ref{Prop_Psi} we obtain then on $\Lambda_M$ that
\begin{align}\label{Eq30}
    |\Psi_t(y) - \Psi_{t'}(y') |\,\le\, M 
    |\Psi_t(y) - \Psi_{t'}(y') |^{1+\beta} \,+\, M \bigl(
    |y'-y|+ |t'-t|^{\beta/2}
    \bigr),
\end{align}
due to the periodicity of $\xi$.
To capitalize on the latter estimate, we introduce the concave auxiliary function $\phi_{M,\beta}(r) = r-Mr^{1+\beta}$ starting at $\phi_{M,\beta}(0)=0$ and admitting its maximum at $R_{\max}\coloneq((1+\beta)M)^{-1/\beta}$. The latter is in particular larger than
\[
R^*_{M,\beta}\,=\, (2M)^{-1/\beta}.
\]
Moreover, for all $r\le R^{*}_{M,\beta} $ we have $Mr^{1+\beta}\le r/2$ and therefore $r/2 \le \phi_{M,\beta}(r)$ . We conclude from \eqref{Eq30} that for all $\delta>0$ and for all $t,t',y,y'$ with $    |y'-y|+ |t'-t|^{\beta/2} \le \delta/M$ 
\begin{equation}\label{Eq222}
\phi_{M,\beta}\bigl( 
|\Psi_t(y) - \Psi_{t'}(y') |
\bigr)\,\le\, \delta.
\end{equation}
Taking now $\delta
    \le 
R^*_{M,\beta}/2 \le 
\phi_{M,\beta}(R^*_{M,\beta})$, we claim that $|\Psi_t(y) - \Psi_{t'}(y') | \le R^*_{M,\beta}$ for all $t,t',y,y'$ as before. Before we prove the claim, note that it can be combined with the previous considerations to obtain  
\begin{equation}\label{Eq31}
    \frac{1}{2}|\Psi_t(y) - \Psi_{t'}(y') |\,\le\, 
\phi_{M,\beta}\bigl( 
|\Psi_t(y) - \Psi_{t'}(y') |
\bigr)\,\le\, M\bigl( |y'-y|+ |t'-t|^{\beta/2}\bigr)
\end{equation}
as soon as  \[   |y'-y|+ |t'-t|^{\beta/2}
    \le R^*_{M,\beta}/(2M) \,=\, (2M)^{-(\beta+1)/\beta}.\]
Now, to prove the claim, it suffices to consider the case $\delta = R^*_{M,\beta}/2$ for which we fix $t',y'$ and let 
\[
B'\,=\,\bigl\{ 
(t,y)\in [0,T] \times \R^d \,\big|\,
|y'-y|+ |t'-t|^{\beta/2} \le R^*_{M,\beta}/(2M)
\bigr\}.
\]
We moreover let 
$G$  be the set of 'good points' in $B$, i.e.,
\[G = \{(t,y)\in B': |\Psi_t(y) - \Psi_{t'}(y') | \le R^*_{M,\beta}\}.\]
We will show that $G = B'$ by proving that $G\ne\emptyset$ is open and closed in the connected set $B$. Since $\Psi$ is continuous thanks to Proposition \ref{Prop_Psi}, the closedness of $G$ is immediate. To show that $G$ is open in $B'$, fix $(t,y)\in G$. Since $(t,y)\in G$ and $\Psi$ is continuous we can choose $\varepsilon>0$ so small that 
\[|\Psi_s(z) - \Psi_{t'}(y') | < R_{\max} \qquad \text{for all} \qquad (s,z)\in B_{\varepsilon}(t,y) \coloneq \{(s,z)\in B':|z-y|+|s-t|^{\beta/2}<\varepsilon\}.\]
For $(s,z)\in B_{\varepsilon}(t,y)$, we have by \eqref{Eq222} at the same time 
\[\phi_{M,\beta}(|\Psi_s(z) - \Psi_{t'}(y') |)\,\leq\, R^*_{M,\beta}/2 \, \le \,
\phi_{M,\beta}(R^*_{M,\beta}).\]
Since $\phi_{M,\beta}$ is strictly increasing on $[R^*_{M,\beta}, R_{\max}]$ we can conclude that
$|\Psi_s(z) - \Psi_{t'}(y')| \leq R^*_{M,\beta}$, and thus $(s,z)\in G$ as well.  This proves that $G$ is open in $B'$, and therefore establishes the claim.   

To proceed, we observe that \eqref{Eq31} clearly implies that for all $t,t',y,y'$ with $|y'-y|+ |t'-t|^{\beta/2} \le  (2M)^{-(\beta+1)/\beta}$,
\begin{equation}\label{Eq313}|\Psi_t(y) - \Psi_{t'}(y') |
\,\le\, 2M\bigl( |y'-y|+ |t'-t|^{\beta/2}\bigr).
\end{equation}
We would like to derive a similar estimate for an arbitrary choice of $t,t',y,y'$ by a chaining argument. For this purpose, we choose sequences $y_0=y$, $t_0 = t$ and $y_l = y'$, $t_k = t'$ with
$|y_i - y_{i-1}| \le (2M)^{-(\beta+1)/\beta}$, $|t_j -t_{j-1}|^{\beta/2} \le (2M)^{-(\beta+1)/\beta}$ and additionally 
\begin{align*}&
\sum_{i=1}^l |y_i - y_{i-1}| \,=\, |y'-y| 
\end{align*}
and
\begin{align}\label{eqn:condition_tjs} 
\sum_{j=1}^k |t_j - t_{j-1}|^{\beta/2} \,\le \, \begin{cases}
|t-t'|^{\beta/2} , & |t-t'|^{\beta/2} \le (2M)^{-(\beta+1)/\beta}, \\
2(2M)^{(2-\beta)(\beta+1)/\beta^2 }
|t'-t|,
& |t-t'|^{\beta/2} > (2M)^{-(\beta+1)/\beta}
.\end{cases}
\end{align}
Here, the first case of the preceding condition simply requires us to take $k=1$, if possible.
The left-hand side of \eqref{eqn:condition_tjs} can, in any case, be bounded by 
\[
C_{\beta,T} M^{(2-\beta)(\beta+1)/\beta^2 }|t'-t|^{\beta/2},
\]
and applying the estimate \eqref{Eq313} to each of the summands of
\[
|\Psi_t(y) - \Psi_{t'}(y') |\,\le\, \sum_{i=1}^l
|\Psi_{t_0}(y_i) - \Psi_{t_0}(y_{i-1}) | \,+\, \sum_{j=1}^k 
|\Psi_{t_j}(y_l) - \Psi_{t_{j-1}}(y_l) |,
\]
results then in the bound
\[
|\Psi_t(y) - \Psi_{t'}(y') | \,\le\, C_{\beta,T} M^{(2+\beta)/\beta^2}\bigl( |y'-y|+ |t'-t|^{\beta/2}\bigr),
\]
for arbitrary $t,t',y,y'$. We conclude that
\[
[\Psi]_{C^{\beta/2,1}( [0,T] \times \R^d; \R^d) } \,\le\, C_{\beta,T}
M^{(2+\beta)/\beta^2}\,\qquad \;\text{on } \;\Lambda_M,
\]
for $M\ge 1$. Thereby, the layer cake representation theorem  yields
\begin{align}\begin{split}\label{Eq36}&
    \E \Bigl[
    [\Psi]_{C^{\beta/2,1}( [0,T] \times \R^d; \R^d) }^p
    \Bigr]
    \,=\,\int_0^\infty 
    \P \bigl( [\Psi]_{C^{\beta/2,1}( [0,T] \times \R^d; \R^d) } > \lambda^{1/p} \bigr)
    \,\dd \lambda 
    \\&\quad =\,  pC_{\beta,T}^p\frac{2+\beta}{\beta^2}\int_0^\infty  \P \bigl( [\Psi]_{C^{\beta/2,1}( [0,T] \times \R^d; \R^d) } > C_{\beta,T}M^{(2+\beta)/\beta^2)} \bigr) M^{p(2+\beta)/\beta^2 -1}
    \,\dd M 
    \\&\quad 
    \lesssim_{\beta,p,T} \,1 \,+\, \int_0^\infty 
    \P \bigl( \Omega\setminus \Lambda_M \bigr) M^{p(2+\beta)/\beta^2 -1}
    \,\dd M 
    \\&\quad=\,1 \,+\, \int_0^\infty 
    \P \bigl( X_{\max} >\sqrt{M} \bigr) M^{p(2+\beta)/\beta^2 -1} \,\dd M
    \,\eqsim_{p,\beta}\, 1 \,+\, \E \bigl[X_{\max}^{2p(2+\beta)/\beta^2}  \bigr],\end{split}
\end{align}
where 
\[
X_{ \max} \,=\,
\|\xi\|_{C^{\beta/2,\beta} ([0,T]\times Q; \R^d) } \vee 
\|D \xi\|_{C^{\beta/2,\beta} ([0,T]\times Q; \R^{d\times d}) } 
\vee 
\|(D \xi(\cdot ))^{-1}\|_{C^{\beta/2,\beta} ([0,T]\times Q; \R^{d\times d}) } ,
\]
and we reverted in the last estimate to the earlier transformation. It remains to use the uniform bounds on the right-hand side of \eqref{Eq36}
in terms of the data provided by the earlier Propositions \ref{prop:periodic_flows_I}--\ref{prop:periodic_flows}.
\end{proof}

\subsection{An It\^o--Wentzell formula}\label{ss:ito-wentzell}
The aim of this subsection is to provide a version of the It\^o--Wentzell formula applicable to random fields of low regularity. While a formula for the distribution-valued case is in principle available in \cite{KrylovIto}, the flow considered there depends on the initial value via $\xi_t(x) =x + Y_t$. Since we want to take $\xi$ to be the flow induced by the Stratonovich SDE \eqref{eq:eqn_xi_intro}, the latter is, however, not sufficient for our purposes, see also the discussion in Subsection \ref{ss:proof_strat}. We also remark that a flow as in \cite{KrylovIto} depends in particular smoothly on $x$ and thereby a composition with any distribution-valued random field can be justified. Due to the more complicated dependence of $\xi$ on $x$ via \eqref{eq:eqn_xi_intro}, we do require some regularity of the random field, which is however below the $C^2$-regularity in space required by traditional It\^o--Wentzell formulae, cf.\ \cite[Theorem I.8.3]{Kunita}, \cite[Theorem 3.3.1]{kunita_book} or  \cite[Proposition 2]{Tubaro}. Most importantly, the resulting assumption is satisfied by weak solutions to parabolic SPDEs.

To this end, we recall the local Sobolev spaces $H_\loc^\gamma(\R^d;\mathscr{X})$, which can for $\gamma\in \R$ and a Banach space $\mathscr{X}$ be  defined as
\[
H^\gamma_{\loc}(\R^d;\mathscr{X}) \,=\, \bigl\{
f\in \mathcal{D}'(\R^d;\mathscr{X}) \big| \forall\lambda>0: (\eta_\lambda f )\in H^\gamma(\R^d;\mathscr{X}) 
\bigr\},
\]
where $\eta$ is a smooth cutoff function with $\eta\equiv 1$ on $B_1(0)$, $\supp(\eta)\subset B_2(0)$ and $\eta_\lambda(x) = \eta(x / \lambda)$. For $s=0$ we naturally write  $L^2_{\loc}(\R^d;\mathscr{X})\,=\, H^0_{\loc}(\R^d;\mathscr{X})$. In any case, the resulting locally convex topology is induced by the metric
\[
d_{H^\gamma_{\loc}(\R^d;\mathscr{X})}(f,g) \,=\, \sum_{k\in \N} 2^{-k} \min \bigl\{
1, \|\eta_k (f-g)\|_{H^\gamma(\R^d;\mathscr{X}) }
\bigr\}
\]
and independent of the choice of $\eta$. With slight abuse of notation, we write
\[
f\in L^p ( 0,T;H^\gamma_{\loc}(\R^d;\mathscr{X})) \quad\iff\quad \forall \lambda>0: \, \eta_{\lambda}f \in L^p ( 0,T;H^\gamma(\R^d;\mathscr{X}))
\] 
for measurable $f\colon [0,T]\to H^\gamma_{\loc}(\R^d;\mathscr{X})$ and say  that $f_n \to f$ in $L^p ( 0,T;H^\gamma_{\loc}(\R^d;\mathscr{X}))$ whenever $\eta_{\lambda}f_n\to \eta_\lambda f $  in  $L^p ( 0,T;H^\gamma(\R^d;\mathscr{X}))$ for all $\lambda>0$.
In particular, if a progressively measurable $g\colon \Omega\times [0,T] \to H^\gamma_{\loc}(\R^d;\ell^2)$ lies almost surely in $ L^2 ( 0,T;H^\gamma_{\loc}(\R^d;\ell^2))$, then we can define the \emph{stochastic integral} 
\begin{equation}
\sum_{n\ge 1}
\int_0^\cdot  g_{n,s}\, \dd w^n_s \,\in \, C([0,T]; H^\gamma_{\loc}(\R^d))
\end{equation}
as follows: For each $k\in \N$, the stochastic integral $
\sum_{n \ge 1}
\int_0^\cdot \eta_k g_{n,s}\, \dd w^n_s$
exists as an adapted, continuous  process in $H^\gamma(\R^d)$ for all $k\in \N$. Moreover, for all $l\ge k$ and $\phi\in C_c^\infty(B_k(0))$ and $t\in [0,T]$, we have the consistency property
\begin{equation}\label{Eq22}
\biggl\langle \sum_{n\ge 1}
\int_0^t \eta_l g_{n,s}\, \dd w^n_s,\phi
\biggr\rangle\,=\, \sum_{n\ge 1}
\int_0^t \langle  g_{n,s}, \phi \rangle\, \dd w^n_s \,=\, \biggl\langle \sum_{n\ge 1}
\int_0^t \eta_k g_{n,s}\, \dd w^n_s,\phi
\biggr\rangle,
\end{equation}
such that we can define 
\[
\sum_{n\ge 1}
\int_0^t  g_{n,s}\, \dd w^n_s \colon \Omega \to \mathcal{D}'(\R^d)\]
as the  $\F_t$-measurable random variable whose action on $\phi\in C_c^\infty(B_k(0))$ is as in \eqref{Eq22}.
Continuity of the resulting process in $H^\gamma_{\loc}(\R^d)$ follows by the continuity of \eqref{Eq22} in $H^\gamma(\R^d)$ for each $k\in \N$. 
In the same fashion, we can define 
\[
\int_0^\cdot  f_{s}\, \dd s \,\in \, C([0,T]; H^\gamma_{\loc}(\R^d))
\]
for $f\in L^1(0,T;H^\gamma_{\loc}(\R^d))$.

Finally, we assume that $\xi\colon \R^d \to \R^d$ is a $C^2$-diffeomorphism and define the \emph{distributional composition} with some $f\in H^{-1}_{\loc}(\R^d)$ by
\begin{equation}\label{eq:distr_comp}
\langle 
f\circ \xi, \phi 
\rangle \,=\,\langle 
f ,( \phi\circ \xi^{-1} ) |\det(D\xi^{-1})|
\rangle_{H^{-1}(B_k(0)) \times {H_0^1(B_k(0))}} ,\qquad \phi\in C_c^\infty(\R^d),
\end{equation}
where we choose $k\in \N$ such that $\supp ( \phi\circ \xi^{-1})$ is compactly contained in $B_k(0)$, see \cite[Section 6.1]{hormander_I}. Since $\xi^{-1}$ twice continuously differentiable, $( \phi\circ \xi^{-1} ) |\det(D\xi^{-1})|$ is continuously differentiable and therefore an element of $C_c^1(B_k(0))$. Here we use that even though the modulus is not differentiable, $\det(D\xi^{-1})$ never reaches $0$, and as a consequence the above dual pairing is well-defined. With this at hand, we are ready to state our version of the It\^o--Wentzell formula.
\begin{proposition}\label{prop:ItoWentzel}
	Let $u\colon \Omega\times (0,T]\to H^1_{\loc}(\R^d) $, $f\colon \Omega\times [0,T] \to H^{-1}_{\loc}(\R^d)$ and $g\colon \Omega\times [0,T] \to L^2_{\loc}(\R^d;\ell^2)$ be progressively measurable with $u\in L^2(0,T; H^1_{\loc}(\R^d))$, $f\in L^1(0,T; H^{-1}_{\loc}(\R^d) )$ and $g\in L^2(0,T; L^2_{\loc}(\R^d;\ell^2) )$,  almost surely. We assume that almost surely
	\begin{equation*}
		u_t \,=\, u_0 \,+\, \int_0^t f_s\,\dd s\,+\, \sum_{n\geq 1}\int_0^t g_{n,s}\, \dd w_s^n
	\end{equation*}
	for all $t\in [0,T]$
	for an $\F_0$-measurable $u_0\colon \Omega \to L^2_{\loc}(\R^d)$.
	Additionally, let $\mu:[0,\infty)\times\Omega\times\R^d\to \R^d$ and $\sigma:[0,\infty)\times\Omega\times\R^d\to \ell^2(\R^d)$ be $\mathcal{P}\otimes \mathcal{B}(\R^d)$-measurable coefficient functions satisfying Assumption \ref{Ass:Ck_alpha} for $k=2$ and some $\alpha\in (0,1]$  and $\xi$ be the associated stochastic flow of $C^2$-diffeomorphisms satisfying  \eqref{eq:SDE_2} provided by Theorem \ref{Thm:Ck}. Then it holds almost surely,  for all $t\in [0,T]$
		\begin{align}
		\begin{split}\label{eq:ItoWentzell1}
			u_t\circ \xi_t &= u_0 + \int_0^t f_s\circ \xi_s\, \dd s + \sum_{n\geq 1}\int_0^t g_{n,s}\circ \xi_s\, \dd w_s^n\\ & \quad +
			\sum_{j=1}^d  \int_0^t [(\partial_j u_s)\circ \xi_s]  \mu^j_s(\xi_s)  \,\dd s + \sum_{n\geq 1}\sum_{j=1}^d \int_0^t [(\partial_j u_s)\circ \xi_s] \, \sigma^j_{n,s}(\xi_s)  \,\dd w_s^n
			\\ &\quad + \sum_{n\geq 1} \sum_{j=1}^d  \int_0^t [(\partial_j g_{n,s})\circ \xi_s]\sigma_{n,s}^j(\xi_s) \,\dd s
			\, + \frac12 \sum_{n\geq 1} \sum_{i,j=1}^d  \int_0^t [(\partial_i \partial_j u_s)\circ \xi_s] \sigma_{n,s}^i(\xi_s)\sigma_{n,s}^j(\xi_s) \,\dd s,
		\end{split}
	\end{align}
	as continuous, $H_{\loc}^{-1}(\R^d)$-valued processes.
\end{proposition}
    We note that the distributional composition rule \eqref{eq:distr_comp} is used in the first, fifth and sixth integrals on the right-hand side of \eqref{eq:ItoWentzell1}. 
    In the above and below, we employ the short-hand notation $\xi_t\,=\, \xi_{0,t}$.
    We also remark that all the integrals on the right-hand side of \eqref{eq:ItoWentzell1} converge almost surely in $C([0,T];H^{-1}_{\loc}(\R^d))$ by the properties of $\xi$ collected in the previous subsections, as is the content of the the second step in the following proof.

\begin{proof}[Proof of Proposition \ref{prop:ItoWentzel}]
	To fix the intuition, we let $(\vp_\kappa)_{\kappa>0}$ be a smooth and compactly supported  approximate identity so that after defining $v^{\kappa} \,=\, \vp_\kappa * v$ for a distribution $v$, we have
	\[
	u_t^\kappa \,=\, u_0^\kappa \,+\, \int_0^t f_s^\kappa\,\dd s\,+\, \sum_{n\geq 1}\int_0^t g_{n,s}^\kappa\, \dd w_s^n.
	\]
	In particular, as for a fixed initial value $x\in \R^d$ 
	\[
	\xi_{t}(x ) \,=\, x\,+\, \int_0^t \mu_s(\xi_{s }(x) ) \,\dd s \,+\, \displaystyle{\sum_{n\ge 1}} \int_0^t \sigma_{n,s}(\xi_{s }(x) )\, \dd w_s^n,
	\]
	 we can apply a more classical version of the It\^o--Wentzell formula like \cite[Theorem 3.1]{KrylovIto} to deduce that almost surely
	\begin{align}\begin{split}\label{eq:ItoWentzell_pw}
			u_t^\kappa(\xi_{t}(x )) \,=\, &u_0^\kappa(x)\,+\, \int_0^t f^\kappa_s(\xi_s(x)) \,\dd s
			\,+\, \sum_{n\ge 1} \int_0^t  g_{n,s}^\kappa (\xi_s(x))\, \dd w_s^n
			\\&+\,  
			\sum_{j=1}^d  \int_0^t \mu^j_s(\xi_s(x)) \partial_j u_s^{\kappa}(\xi_s(x))   \,\dd s + \sum_{n\geq 1}\sum_{j=1}^d \int_0^t \sigma^j_{n,s}(\xi_s(x)) \partial_j u_s^\kappa (\xi_s(x)) \,\dd w_s^n
			\\&+\,  \sum_{n\geq 1} \sum_{j=1}^d  \int_0^t \sigma_{n,s}^j(\xi_s(x))\partial_j g_{n,s}^\kappa( \xi_s(x)) \,\dd s
			\\&+\, 
			\frac12 \sum_{n\geq 1} \sum_{i,j=1}^d  \int_0^t \sigma_{n,s}^i(\xi_s(x))\sigma_{n,s}^j(\xi_s(x)) \partial_{ij} u_s^\kappa( \xi_s(x)) \,\dd s,
		\end{split}
	\end{align}
	 for all $t\in [0,T]$. In particular, the above is a pointwise version of the desired \eqref{eq:ItoWentzell1} and the remainder of this proof is devoted to taking $\kappa\to 0$ in the above.

	\emph{Step 1 (Observations regarding the distributional composition with $\xi$).} Firstly, using the inverse function theorem, we have for any $\phi\in C_c^\infty(\R^d)$ that 
	\begin{equation}\label{Eq12}
	\widebar{\phi}\colon 
	[0,T]\times \R^d \to \R ,\, (t,y)\mapsto \phi(\xi_t^{-1}(y)) \det (D\xi_{t}^{-1}(y))\,=\, \bigl[\phi \cdot {\det}^{-1}(D\xi_t) \bigr] (\xi_t^{-1}(y))
	\end{equation}
	has compact support due to the continuity of $\xi$.  We remark that $\det (D\xi_t^{-1}) >0$ by $\det(D \xi_0^{-1}) = 1$ and continuity due to Proposition \ref{Prop_Psi}, which is why we drop the modulus compared to \eqref{eq:distr_comp}. 
	We claim that additionally
	\begin{equation}\label{Eq21}
	\sup_{t\in [0,T]}\| \widebar{\phi}(t,\cdot ) \|_{H^1(\R^d)} \,\lesssim_{{(\omega,k,T)}} \| \phi \|_{H^1(\R^d)},
	\end{equation}
	for almost all $\omega$, where $k$ is such that $\supp(\phi)\subset B_k(0)$.
	Indeed,  for $\phi$ satisfying the latter condition, we have that
	\begin{align*}&
		\sup_{t\in [0,T]} \int_{\R^d} \bigl|\widebar{\phi}(t,y)\bigr|^2 \, \dd y \\&\quad =\, 
		\sup_{t\in [0,T]} \int_{\R^d} \bigl| \bigl[\phi^2 \cdot {\det}^{-1}(D\xi_t) \bigr] (\xi_t^{-1}(y))
		 \det (D\xi_{t}^{-1}(y))
		\bigr| \, \dd y
		\\&\quad =\, 
		\sup_{t\in [0,T]} \int_{\R^d} \bigl|\phi^2(x) \cdot {\det}^{-1}(D\xi_t(x) )\bigr| \, \dd x
		\\&\quad 
		\le \, \|\phi\|_{L^2(\R^d)}^2 \sup_{(t,x)\in [0,T]\times B_k(0)} |{\det}^{-1}(D\xi_t (x))| \,\lesssim_{(\omega,k,T)} \,  \|\phi\|_{L^2(\R^d)}^2,
	\end{align*}
	by the transformation rule and the continuity of
	\begin{equation}\label{Eq27}
	(D\xi_{\cdot}(\cdot))^{-1} \colon [0,T]\times \R^d \to \R^{d\times d},
	\end{equation}
	guaranteed by Theorem  \ref{Thm:Ck}.
	As a second step to verify \eqref{Eq21}, we proceed similarly for the derivative and calculate 
	\begin{align*}&
		\nabla_y \widebar{\phi}(t,y) \,=\, D\xi_t^{-1}(y) \bigl[
		\nabla_x \bigl(
		\phi\cdot {\det}^{-1}(D\xi_t)
		\bigr)
		\bigr](\xi_t^{-1}(y))
		\\&\quad =\, D\xi_t^{-1}(y) \bigl[
		 {\det}^{-1}(D\xi_t) \nabla_x\phi
		\bigr](\xi_t^{-1}(y))\,+\, 
		D\xi_t^{-1}(y) \bigl[\phi
		\nabla_x \bigl(
		{\det}^{-1}(D\xi_t)\bigr) 
		\bigr)
		\bigr](\xi_t^{-1}(y))
		\\&\quad =\, I_1(t,y) \,+\, I_2(t,y).
	\end{align*}
Since 
\[
\|\nabla_y \widebar{\phi}(t,\cdot ) \|_{L^2(\R^d;\R^d)} \,\le \, 
\|I_1(t,\cdot) \|_{L^2(\R^d;\R^d)}\,+\, 
\|I_2(t,\cdot) \|_{L^2(\R^d;\R^d)},
\]
it suffices to estimate both terms separately. Regarding $I_1$ we compute
\begin{align*}&
	\sup_{t\in [0,T]} \int_{\R^d} 
	|I_1(t,y)|^2
	\,\dd y
	\\&\quad =\, 
	\sup_{t\in [0,T]} \int_{\R^d} 
	\bigl|(D\xi_t)^{-1} (\xi_t^{-1}(y))\nabla_x \phi(\xi_t^{-1}(y)) \bigr|^2
	\bigl|
	{\det}^{-1}(D\xi_t)
	 (\xi_t^{-1}(y)) 
	\det (D\xi_t^{-1}(y))
	\bigr|
	\,\dd y
	\\&\quad =\, 
		\sup_{t\in [0,T]} \int_{\R^d} 
	\bigl|(D\xi_t)^{-1}(x)\nabla_x \phi(x)\bigr|^2
	\bigl|
	{\det}^{-1}(D\xi_t(x) )
	\bigr| 
	\,\dd x
	\\&\quad \le \, 
	\|\nabla \phi\|_{L^2(\R^d;\R^d)}^2
	\sup_{(t,x)\in [0,T]\times B_k(0)}
	\bigl|(D\xi_t)^{-1}(x)\bigr|^2
	\bigl|
	{\det}^{-1}(D\xi_t(x))
	\bigr| 
	\\
	&\quad \lesssim_{(\omega,k,T)} 
	\|\nabla \phi\|_{L^2(\R^d;\R^d)}^2.
\end{align*}
where we used again the continuity of \eqref{Eq27} in the latter estimate. 
For $I_2$, we first observe that
\begin{align*}
	\partial_{x_i} {\det}^{-1}(D\xi_t)\,=\, -{\det}^{-1}(D\xi_t)\tr(( {D \xi_t})^{-1} \partial_{x_i} D \xi_t )
\end{align*}
by Jacobi's formula and the chain rule, 
which we write with a slight abuse of notation as
\[
	\nabla_x {\det}^{-1}(D\xi_t)\,=\, -{\det}^{-1}(D\xi_t)\tr(( {D \xi_t})^{-1} D^2 \xi_t ).
\]
 Thus, we can  proceed to 
\begin{align*}&
		\sup_{t\in [0,T]} \int_{\R^d} 
	|I_2(t,y)|^2
	\,\dd y
	\\&\quad =\,
	\sup_{t\in [0,T]}
	\int_{\R^d} 
	\bigl|\bigl[\phi(D\xi_t)^{-1}\tr(( {D \xi_t})^{-1} D^2 \xi_t ) \bigr](\xi_t^{-1}(y)) \bigr|^2 
	\bigl|
	\det (D\xi_t^{-1}(y))
	\bigr|^2
	\,\dd y
		\\&\quad =\,
	\sup_{t\in [0,T]}
	\int_{\R^d} 
	\bigl|\phi(x)(D\xi_t)^{-1}(x)\tr(( {D \xi_t})^{-1} D^2 \xi_t ) (x)\bigr|^2 
	\bigl|
	{\det}^{-1}(D\xi_t(x))
	\bigr|
	\,\dd x
	\\&\quad \le \, 
	\| \phi\|_{L^2(\R^d;\R^d)}^2
	\sup_{(t,x)\in [0,T]\times B_k(0)}
	\bigl|(D\xi_t)^{-1} (x)\tr(( {D \xi_t})^{-1} D^2 \xi_t )(x)\bigr|^2
	\bigl|
	{\det}^{-1}(D\xi_t(x))
	\bigr| 
	\\
	&\quad \lesssim_{(\omega,k,T)} 
	\| \phi\|_{L^2(\R^d;\R^d)}^2,
\end{align*}
where the last estimate also uses the continuity of $D^2 \xi$ due to Theorem \ref{Thm:Ck}, and \eqref{Eq21} follows.

\emph{Step 2 (The right-hand side of \eqref{eq:ItoWentzell1} is almost surely  continuous in $H^{-1}_{\loc}(\R^d)$). }
We observe that by the previous step, for each $l\in \N$
\begin{align}\begin{split}\label{Eq11}
	\|  \eta_l (f_t\circ \xi_t)\|_{H^{-1}(\R^d)} \,&=\, 
	\sup_{\|\phi\|_{ H^1{(\R^d)}}\le 1 } | \langle \eta_l (f_t\circ \xi_t),\phi\rangle|
	\,\le \, \sup_{\substack{\|\phi\|_{ H^1{(\R^d)}}\le C_{\eta,l} \\ \supp(\phi) \subset 2l }}
	 |\langle f_t\circ \xi_t,\phi\rangle|\\& =\, 
 \sup_{\substack{\|\phi\|_{ H^1{(\R^d)}}\le C_{\eta,l} \\ \supp(\phi) \subset 2l }}
 |\langle f_t,\bar{\phi}\rangle|\,\lesssim_{(\omega, \eta,l,T)} \|\eta_{\lambda(\omega,l ,T)} f_t\|_{H^{-1}(\R^d)},
	\end{split}
\end{align}
for almost all $(\omega ,t )$, where $\lambda(\omega,l,T)$ is chosen sufficiently large and $\bar{\phi}$ is as in \eqref{Eq12}. By integrating  in time, we deduce that
$f \circ \xi \in L^1([0,T]; H_{\loc}^{-1}(\R^d)) $, and the same line of arguments yields that also all other integrands on the right-hand side of \eqref{eq:ItoWentzell1} are (stochastically) integrable, which is therefore continuous in $H^{-1}_{\loc}(\R^d)$. 
	
\emph{Step 3 (The left-hand side of \eqref{eq:ItoWentzell1} is  almost surely  continuous in $H^{-1}_{\loc}(\R^d)$).} 	To verify the latter we let $\epsilon>0$ and decompose 
\begin{align*}&
	\| \eta_l (u_t\circ \xi_t -  u_s\circ \xi_s)  \|_{H^{-1}(\R^d)}
	\\&
	\quad \le \, 	\| \eta_l ( u_t\circ \xi_t -  (\vp *u_t)\circ \xi_t)  \|_{H^{-1}(\R^d)} 
	\, +\, \| \eta_l ( (\vp *u_t)\circ \xi_t -  (\vp *u_s)\circ \xi_s)  \|_{H^{-1}(\R^d)} 
	\\&\qquad +\, \| \eta_l ((\vp *u_s)\circ \xi_s -  u_s\circ \xi_s)  \|_{H^{-1}(\R^d)} \,=\, A_1 \,+\, A_2\,+\, A_3,
\end{align*}
where $\vp$ is a smooth, compactly supported convolution kernel.
Regarding $A_3$, we use \eqref{Eq11} to bound 
\begin{align*}
	\| \eta_l ((\vp *u_s)\circ \xi_s -  u_s\circ \xi_s)  \|_{H^{-1}(\R^d)}\,\le \, C_{(\omega, \eta,l,T)} \|\eta_{\lambda} (\vp *u_s -  u_s)\|_{H^{-1}(\R^d)},
\end{align*}
for $\lambda = \lambda(\omega,l,T)$ from the previous step 
and we choose $\vp$ in a way such that the right-hand side is less than $\epsilon/4$. As $\vp* u$ is continuous in time and space and $(t,\xi_t(x)) \to (s,\xi_s(x))$ as $t\to s$ for all $x\in \R^d$, we can find $\delta>0$ such that $A_2<\epsilon/4$ for $ |t-s|<\delta$ by $L^2(\R^d) \hookrightarrow H^{-1}(\R^d)$ and the dominated convergence theorem. Lastly, for $A_1$ we use \eqref{Eq11} and further expand 
\begin{align*}&
	\| \eta_l ( u_t\circ \xi_t -  \vp *u_t\circ \xi_t)  \|_{H^{-1}(\R^d)} \,\le \, C_{(\omega, \eta,l,T)} \|\eta_{\lambda(\omega,l,T)} (u_t -  \vp *u_t)\|_{H^{-1}(\R^d)} 	
	\\&\quad \le \,  C_{(\omega, \eta,l,T)}\bigl( \|\eta_{\lambda} (u_t -  u_s)\|_{H^{-1}(\R^d)} \,+\, \|\eta_{\lambda} (u_s-  \vp* u_s)\|_{H^{-1}(\R^d)}	\bigr)
\\	&\qquad +\, C_{(\omega, \eta,l,T)} \| \eta_\lambda ( \vp* u_s -  \vp *u_t)  \|_{H^{-1}(\R^d)} \,=\, A_{11}\,+\,  A_{12}\,+\,  A_{13}.
\end{align*}
While $A_{12} <\epsilon/4$ by our choice of $\vp$, we can  guarantee $A_{11}+  A_{13} <\epsilon/4$ by decreasing $\delta$ if necessary.

\emph{Step 4 (Conclusion).} Since both sides of \eqref{eq:ItoWentzell1} are continuous processes in $H^{-1}_{\loc}(\R^d)$ it suffices to prove that for a fixed time instance $t\in [0,T]$ they coincide almost surely when tested against some $\phi \in C_c^\infty(\R^d)$. The desired equality reads 
	\begin{equation}\label{eq:weakItoWentzell}
	\begin{aligned}
		\lb u_t,\overline{\phi}_t\rb   &= \lb u_0, \overline{\phi}_0\rb + \int_0^t \lb f_s, \overline{\phi}_s\rb \dd s + \sum_{n\geq 1}\int_0^t \lb g_{n,s},\overline{\phi}_s\rb\, \dd w^n_s\\ & \quad -
		\sum_{j=1}^d  \int_0^t \lb u_s, \partial_j ({\mu}^j_s \overline{\phi}_s)\rb\, \dd s - \sum_{n\geq 1}\sum_{j=1}^d \int_0^t \lb u_s,\partial_j( {\sigma}^j_{n,s}\overline{\phi}_s)\rb\,  \dd w^n_s
		\\ &\quad - \sum_{n\geq 1} \sum_{j=1}^d  \int_0^t \lb g_{n,s}, \partial_j ({\sigma}_{n,s}^j \overline{\phi}_s)\rb \,\dd  s
	\,-\,\frac12 \sum_{n\geq 1} \sum_{i,j=1}^d  \int_0^t \lb \partial_i u_s, \partial_j ({\sigma}_{n,s}^i{\sigma}_{n,s}^j \overline{\phi}_s) \rb \,\dd s,
	\end{aligned}
\end{equation}
	where we used the transformation rule and the notation \eqref{Eq12}. 
	To this end, we multiply \eqref{eq:ItoWentzell_pw} with $\phi(x)$ for $\phi\in C_c^\infty(\R^d)$ and observe that the stochastic Fubini theorem \cite[Lemma 2.7]{KrylovIto}, see also \cite[Theorem 2.2]{stoch_Fubini_Veraar}, is applicable over the compact set $\supp(\phi)$. Using additionally the transformation rule, we observe that \eqref{eq:weakItoWentzell} is valid if we replace $u$, $f$, and $g$ by the convolved versions $u^\kappa$, $f^\kappa$, and $g^\kappa$ introduced at the beginning of this proof. Using that the test functions in \eqref{eq:weakItoWentzell} are compactly supported with 
	\begin{align*}
		\overline{\phi} \, \in \, &L^\infty(0,T;H^1(\R^d)),\\
		\partial_j (\mu \overline{\phi}) \, \in \, & L^\infty(0,T;L^2(\R^d)),\\
		\partial_j ( \sigma_n^j \overline{\phi}) \, \in \, &L^\infty(0,T;L^2(\R^d;\ell^2)),\\
		\partial_j ( \sigma_n^j \sigma_n^i\overline{\phi} )\, \in \, &L^\infty(0,T;L^2(\R^d;\ell^1))
	\end{align*}
	by the first step, we see that the almost sure convergence
	\begin{align*}
		u^\kappa \,\to\, u,\qquad &\text{in } L^2( 0,T ; H^{1}_{\loc}(\R^d)), \\
		f^\kappa \, \to \, f ,\qquad &\text{in } L^1(0,T; H^{-1}_{\loc}(\R^d) ), \\
		g^\kappa \, \to \, g ,\qquad &\text{in } L^2(0,T; L^2_{\loc}(\R^d;\ell^2) ) \\
	\end{align*}
	imply the claim. 
\end{proof}\section{Stochastic De Giorgi--Nash--Moser estimates}
\label{app:DeGiorgi_Nash_Moser}
In this section, we state and prove our main results, Theorem \ref{t:DeGiorgiNashMoser} and Theorem \ref{t:DeGiorgiNashMoser_quantitative}, concerning the  \emph{H\"older regularity of weak solutions} to  linear second-order stochastic PDEs 
\begin{equation}
	\label{eq:SPDE_lin}
	\begin{cases}
		\displaystyle{
			\dd u \, - \, Au \,\dd t \, = \, f \, \dd t \,+\, \bigl(B_n u \,+\,g_n\bigr) \,\dd w^n_t,} &\text{ on }[0,\tau]\times \Tor^d,\\
		u(0)\,=\,u_0,&\text{ on }\Tor^d,
	\end{cases}
\end{equation}
where we recall that we use the \emph{Einstein summation convention}
and the shorthand notation 
\begin{align*}
	A u &\,= \,\partial_i(a^{ij}\partial_j u) +  a^i \partial_i u + a^0 u,
	\qquad  B_n u \,= \, b^i_n \partial_i u \, + \,b^0_n u,
	\qquad  f  \,=\, f^0 \,+\,  \partial_i f^i,
\end{align*}
from \eqref{Eq20}. 
The following assumption regarding  $(A,B)$ stands throughout.

\begin{assumption}[Parabolicity\,\&\,Boundedness]
	\label{ass:DeGiorgi_Nash_Moser}
	We suppose that 
    the following is satisfied:
	\begin{enumerate}[(i)]
		\item\label{it:am_DeGiorgi}
		For all $i,j\in \{1, \ldots, d\}$, $a^{ij}:[0,T]\times\O\times \Tor^d\to \R$ and $b^i:[0,T]\times\O\times \Tor^d\to \ell^2$,
		are $\Progress\otimes\Borel(\Tor^d)$-measurable and there exists an $M,\nu\in (0,\infty)$ such that a.s.
		\begin{align}\begin{split}\label{eq:ellip}\sum_{i,j=1}^d |a^{ij}(t,x)|+ \sum_{i=1}^d \|(b^i_n(t,x))_{n\geq 1}\|_{\ell^2} \,&\leq\, M 
		\\	\biggr( a^{ij}(t,x)-\frac{1}{2} b^i_n(t,x) b^j_n(t,x) \biggl) {\eta }_i {\eta }_j\,&\geq \, \ellip | {\eta}|^2,
	\end{split}
		\end{align}
		for all $ {\eta }\in \R^d$, $t\in [0,T]$ and $ x\in \Tor^d$.
		\item For all $i\in \{1, \ldots, d\}$, $ a^i, a^0:[0,T]\times\O\times \Tor^d\to \R$ and $ b^0:[0,T]\times\O\times \Tor^d\to \ell^2$,
		are $\Progress\otimes\Borel(\Tor^d)$-measurable such that a.s.
		\[
        |a^0(t,x)|\,+\,\sum_{i=1}^d |a^i(t,x)|\, + \,\|b^0(t,x)\|_{\ell^2}\,\le\,  M,
		\]
		for all  $t\in [0,T]$ and $ x\in \Tor^d$.
		\end{enumerate}
\end{assumption}
Under the previous assumption,  existence and uniqueness of weak solutions to \eqref{eq:SPDE_lin} for suitable $(u_0,f,g)$ can be obtained, for example, using monotone operator theory \cite[Chapter 4]{LiuRock}. 
We remark that one can remove the common moment assumption that $u_0\in L^2(\Omega; L^2(\T^d))$ by localizing to  $\{  \| u_0 \|_{L^2(\T)}^2 \le k\} \in \F_0$ for $k\in \N$ and similar assumptions on $(f,g)$ by a stopping time argument. This is summarized in the following proposition.
\begin{proposition}[Well-posedness of \eqref{eq:SPDE_lin}]\label{prop:Existence_Uniqueness}
	Let $u_0 \colon \Omega \to L^2(\T^d)$ be $\F_0$-measurable, $\tau$ a stopping time  and 
	$f^0, \ldots, f^d\colon [0,\tau]\times\O\times \Tor^d\to \R $ and $g \colon [0,\tau]\times\O\times \Tor^d\to \ell^2$ be  $\Progress\otimes\Borel(\Tor^d)$-measurable with a.s. 
	\[
	f^0, 
\ldots, f^d\in L^2(0,\tau;L^2(\Tor^d)) \ \ \text{and} \ \  g\in L^2(0,\tau;L^2(\Tor^d;\ell^2)).\]
	Then there exists a unique progressively measurable $u \in  L^2(0,\tau;H^1(\Tor^d))\cap C([0,\tau];L^2(\Tor^d))$ such that for all $\phi\in C^\infty(\T^d)$ a.s.
	\begin{align*}
		\langle
		u_t,  \phi\rangle \,-\, \langle u_0,\phi \rangle  \,=\, & \int_0^t \big(- \langle a^{ij} \partial_j u ,\partial_i \phi\rangle \, + \, \langle  
		a^i \partial_i u +a^0 u ,\phi\rangle \, + \, \langle f^0, \phi \rangle \, -\,  \langle f^i ,\partial_i \phi\rangle\big) \,\dd s
\\&
+\, \int_0^t \langle  b^i_n \partial_i u \, + \,b^0_n u \,+\, g_n ,\phi \rangle \,\dd w^n,\qquad t\in [0,\tau].
\end{align*}
\end{proposition}
In the following, we call the unique $u$ provided by the preceding proposition \emph{the solution to \eqref{eq:SPDE_lin}}. 
Boundedness of solutions to \eqref{eq:SPDE_lin} under the above Assumption \ref{ass:DeGiorgi_Nash_Moser} was proved in \cite{DG15_boundedness}, as we reviewed in the introductory Section \ref{sec:intro}. 

For later use, we recall the following result, which follows from \cite[Theorem 1]{DG15_boundedness} by applying the standard $\omega$-localization technique mentioned above.

\begin{proposition}[Boundedness of solutions]\label{prop:bddness}
	In addition to the assumptions of Proposition \ref{prop:Existence_Uniqueness}, we assume that a.s. $u_0\in L^\infty(\T^d)$,
		\begin{equation}\label{Eq24}
	f^0\in L^{p/2}(0,\tau;L^{q/2}(\Tor^d)),\quad \text{ and }\quad
	f^1, \ldots, f^d\in L^p(0,\tau;L^q(\Tor^d))
    \ \ \text{and} \ \  g\in L^p(0,\tau;L^q(\Tor^d;\ell^2)),\end{equation}
	for some $p,q\in (2,\infty)$ with $\frac{2}{p}+\frac{d}{q}<1$ and let $u$ be the solution to \eqref{eq:SPDE_lin}. Then,  a.s., $u\in L^\infty([0,\tau]\times \T^d)$, and we have the estimate
    \begin{align}\label{bddness_quant}
        \E\|u\|_{L^\infty([0,\tau]\times \T^d)}^r \,\le \, C
        \E\biggl[
        \|u_0\|_{L^\infty(\T^d)}^r + \|f^0 \|_{L^{p/2}(0,\tau;L^{q/2}(\Tor^d))}^r + \sum_{i=1}^d
        \|f^i\|_{ L^p(0,\tau;L^q(\Tor^d))}^r +\| g \|_{L^p(0,\tau;L^q(\Tor^d;\ell^2))}^r
        \biggr]
    \end{align}
    for any $r>0$, where $C$ depends only on $(d,T,M,\nu, p,q,r)$.
\end{proposition}

Before going further, let us emphasize that the conditions in \eqref{Eq24} with $\frac{2}{p}+\frac{d}{q}<1$ are optimal in an $L^p(L^q)$-scale to obtain boundedness of the solution to \eqref{eq:SPDE_lin} even in the case of the stochastic heat equation, i.e., $A u =\Delta u$ and $B u=0$. For the deterministic part, this is well-known, while for the stochastic part, this follows from the argument used in Step 2 in the proof of Theorem \ref{t:DeGiorgiNashMoser} below. 
Interestingly, the conditions \eqref{Eq24}  with $\frac{2}{p}+\frac{d}{q}<1$ are also sufficient for obtaining H\"older continuity of the solution to \eqref{eq:SPDE_lin} as shown below.

The boundedness of solutions in Proposition \ref{prop:bddness} was subsequently 
extended to almost sure continuity of solutions at each point $(t_0,x_0)$ in the interior of the space-time domain in \cite{DG17_continuity}, see Section \ref{sec:intro} for a more extensive  literature review. 
Our goal is to improve upon this by showing that the solution lies a.s. in some H\"older class (which may depend on $\omega$). In addition to the assumption that $u_0$ is H\"older continuous, we also require additional regularity from the transport noise coefficients  $b^i$, while for the diffusion coefficient $a^{ij}$ the already imposed parabolicity and boundedness are sufficient. As the need for additional regularity of the noise coefficients comes from the use of the stochastic method of characteristics as discussed in Subsection \ref{ss:proof_strat}, it is an intriguing question whether alternative approaches can eliminate this assumption, see Subsection \ref{ss:open_problems} for a discussion. 
\begin{theorem}[Stochastic De Giorgi--Nash--Moser estimates -- qualitative version]
	\label{t:DeGiorgiNashMoser}
	Let Assumption \ref{ass:DeGiorgi_Nash_Moser} be satisfied, and let $\tau \le T$ be a stopping time. Suppose that for some $\delta\in (0,1)$ 
	and $R_b<\infty$ we have a.s.
	\[\|b^i(t,\cdot)\|_{C^{3+\delta}(\T^d;\ell^2)}\leq R_b,\qquad t\in [0,T], \,i\in \{1,\ldots, d\}.\]
	Moreover, assume that there exist $\g_0>0$ and $p,q\in (2,\infty)$ such that $\frac{2}{p}+\frac{d}{q}<1$ and $f^0, \ldots, f^d,g$ are progressively measurable satisfying a.s.\
	\begin{equation}\label{eq:assumption_rhs}
	f^0\in L^{p/2}(0,\tau;L^{q/2}(\Tor^d)),\quad \quad
	f^1, \ldots, f^d \in L^p(0,\tau;L^q(\Tor^d)) \ \ \text{and} \ \  g\in L^p(0,\tau;L^q(\Tor^d;\ell^2)),\end{equation}
	and $u_0\in C^{\gamma_0}({\Tor^d})$ a.s.\ is strongly $\F_0$-measurable. 
	Then, there exists a sequence of $(\overline{\g}_m)_m\subset (0,\infty)$ depending only on $(d, M,\nu,p,q)$ for which the  solution $u$ to \eqref{eq:SPDE_lin} satisfies the following with $\g_m:={(\overline{\g}_m\wedge\g_0)}/{3}$: 
    
    There exist stopping times $(\tau_m)_m$ depending only on $\tau$ and $b$ such that $\lim_{m\to \infty}	\P(\tau_m=\tau)=1$ and
	\begin{equation}
    \label{eq:qual_statement}
	u\in C^{ \g_m }((0,\tau_m)\times \Tor^d)\ \ \text{ a.s.\ for all } m\ge 1.
	\end{equation}
\end{theorem}

While answering our main question regarding the H\"older continuity of solutions to \eqref{eq:SPDE_lin} positively,
the previous statement has the shortcoming that the sequence of stopping times $\tau_m$ 
may depend on the specific noise coefficient $b$. The following version addresses this issue by providing control over the probability that $u$ is uniformly Hölder continuous with exponent $\delta_k$, depending only on the size, but not the specific form, of the data.

\begin{theorem}[Stochastic De Giorgi--Nash--Moser estimates -- quantitative  version]
\label{t:DeGiorgiNashMoser_quantitative}
        Under the assumptions of Theorem \ref{t:DeGiorgiNashMoser} set
	\begin{align*}
	\data_{u_0,f,g}
	\,:=\,&\|u_0\|_{C^{\g_0}(\Tor^d)}\,+\,\|f^0\|_{L^{p/2}(0,\tau;L^{q/2}(\Tor^d))}\\
	&\,+\, \textstyle\sum_{j=1}^d \|f^j\|_{L^p(0,\tau;L^q(\Tor^d))} \,+\,
	\|g\|_{L^p(0,\tau;L^q(\Tor^d;\ell^2))}.
	\end{align*}
    Then there exist sequences $(\delta_k)_k,(\psi_k)_k\subset(0,\infty)$ depending only on $(d,T,M,\nu,\gamma_0,p,q,R_b)$ such that $\lim_{k\to \infty}\psi_k=\infty$ and
	$$
	\P\big(\|u\|_{C^{\delta_k}((0,\tau)\times \T^d)}> k\big) \,\leq\, 
	 \psi_k^{-1}\,+\,10\,\P(\data_{u_0,f,g}>\psi_k),\quad \text{ for all } k\,\ge\, 1.
	$$ 
    \end{theorem}
The above shows that the set on which $u$ is not uniformly Hölder continuous can be made arbitrarily small by choosing 
$k$ sufficiently large, with this choice depending only on the norm of the data and the constants in Assumption \ref{ass:DeGiorgi_Nash_Moser}. 
Comments and possible improvements regarding the fact that in both versions of the stochastic  De Giorgi--Nash--Moser estimates the H\"older exponent may become small on a set of small probability can be found in Section \ref{sec:unif_constants}. 
The remainder of the current section is devoted to the proof of Theorems \ref{t:DeGiorgiNashMoser}--\ref{t:DeGiorgiNashMoser_quantitative}. 

\subsection{Proofs of Theorems \ref{t:DeGiorgiNashMoser}--\ref{t:DeGiorgiNashMoser_quantitative}}
\label{ss:proofs_main}
We begin by proving the qualitative version of the stochastic De Giorgi--Nash--Moser estimates---an outline of the main idea was given in Subsection \ref{ss:proof_strat}.

\begin{proof}[Proof of Theorem \ref{t:DeGiorgiNashMoser}]
We start with some reductions. Firstly, by extending $f^i$ and $g$ trivially to $[0,T]$, we can assume that $\tau=T$
 in the first place. Secondly, by 
 Proposition \ref{prop:bddness}  $u$ is a.s.\ bounded so that we can replace $(f^0, g )$ by $(f^0 +a^0 u,g + b^0 u)$ and set in the following $a^0 = b^0 = 0$. Next, we define the periodization $\tilde{u} (x) = u(\bar{x})$ whenever $\bar{x} = x+ \Z^d $ and obtain in this way $\tilde{u} \in C([0,T] ; L_{\loc}^2(\R^d))\cap L^2(0,T; H^{1}_{\loc}(\R^d))$ on the whole space. Using that 
 \[
 \langle \tilde{u} , \phi  \rangle \,=\, \biggl\langle u, \sum_{k\in \Z^d} \phi(\cdot + k) \biggr\rangle
 \]
	 for $\phi \in C_c^\infty(\R^d)$, we can also define $\tilde{f} \in H^{-1}_{\loc}(\R^d)$ for $f\in H^{-1}(\T^d)$ 
     so that the resulting periodized equation 
	\begin{equation}
				\dd\wt{u} \, -\, \wt{A} \wt{u} \, \dd t 
				=\, \wt{f}\, \dd t \,+\,\bigl( \wt{B} \wt{u} \,+ \,\wt{g}_n\bigr)\, \dd w^n 
	\end{equation}holds
	in $H^{-1}_{\loc}(\R^d)$.
	With a slight abuse of notation, we write again $u$ instead of $\wt{u}$ and similar below.
	
	{\em Step 1 (Reduction to $b^i=0$, $i\in \{1, \ldots, d\}$).}
	Introducing $\mu_t^i = \frac12  b_{n}^j(t,\cdot )\partial_j b_{n}^i(t,\cdot)$ and $\sigma_{n,t}^i= -b_n^i(t,\cdot)$, we let $\xi_t$ be the stochastic flow of $C^2$-diffeomorphisms solving \eqref{eq:SDE_2} provided by Theorem \ref{Thm:Ck}. We point out that by periodicity of the coefficient functions and pathwise uniqueness for \eqref{eq:SDE_2} we have, a.s.,
	\begin{equation}\label{eq:periodic_xi}
		\xi(x+j) \,=\, \xi(x) + j,\quad D\xi(x) \,=\, D\xi(x+j), \qquad j\in \Z^d,\,x\in \R^d, \, t\in [0,T],
	\end{equation}
    as discussed in Subsection \ref{sec:periodic_flows}.
	According to Proposition \ref{prop:ItoWentzel} the (therefore periodic) composition $v = u \circ \xi$ satisfies
	\begin{align*}
		\dd v  \,=\, &
		\bigl((A u)\circ \xi \, + \, f\circ \xi\bigr) \,\dd t  \, + \, (b_n^i\partial_i u + g_{n})\circ \xi \,\dd w^n\,+\, 
		(\mu^i  \partial_i u )\circ \xi\,   \dd t 
		\\ & -\,( b_n^i \partial_i u ) \circ \xi\,\dd w^n \,- \,\bigl((b_{n}^i\partial_i (b_n^j \partial_j u))\circ \xi \,+\, (b_{n}^i\partial_i g_{n})\circ \xi\bigr) \, \dd t\,+\,\frac12 ( b_n^i b_n^j \partial_{ij} u) \circ \xi \,\dd t
        \\=\,&
        \bigl((A u)\circ \xi \, + \, f\circ \xi\bigr) \,\dd t  \, + \, g_{n}\circ \xi \,\dd w^n\,+\, 
		(\mu^i  \partial_i u )\circ \xi\,   \dd t 
		\\ & - \,\bigl((b_{n}^i\partial_i (b_n^j \partial_j u))\circ \xi \,+\, (b_{n}^i\partial_i g_{n})\circ \xi\bigr) \, \dd t\,+\,\frac12 ( b_n^i b_n^j \partial_{ij} u) \circ \xi \,\dd t
        ,
	\end{align*}
	where the transport noise term canceled out as desired. We also remark that $\partial_j( f(\xi )) \,=\,   \partial_j \xi^i (\partial_i f)(\xi)$ and thus $(\partial_i f)(\xi) = \psi^j_i\partial_j (f(\xi )) $ where 
	\begin{equation}\label{eq:def_psi}
	\psi:= (D\xi)^{-1},
	\end{equation}
     if $f\in L^1_{\loc}(\R^d)$ is weakly differentiable. Since 
	$
	D\xi \in C_{\loc}^{\alpha/2-, (2+\alpha)-}([0,T]\times \R^d)
	$ by Theorem \ref{Thm:Ck} we have $\psi \in C_{\loc}^{\alpha/2-, (2+\alpha)-}([0,T]\times \R^{d\times d})$ a.s.\ and hence the latter extends to $f\in L^2_{\loc}(\R^d)$, in which case $(\partial_j f ) (\xi)$ is understood as the distributional composition \eqref{eq:distr_comp}. This allows us to rewrite
	\begin{align*}&
		(\mu^i \partial_iu)\circ \xi \,=\, \mu^i(\xi) \psi_i^j \partial_j v,\qquad (a^i\partial_iu)\circ \xi \,=\, a^i(\xi) \psi_i^j \partial_j v,
		\\& 
		(\partial_i (a^{ij} \partial_j u))\circ \xi \,=\, \psi_i^k \partial_k ( (a^{ij }\partial_j u)\circ \xi )\,=\, \partial_k (\psi_i^k 
		 a^{ij}(\xi) \psi_j^l \partial_l v
		) \,-\, (\partial_k \psi_i^k) a^{ij}(\xi) \psi_j^l \partial_l v,
		\\&- (b_{n}^i\partial_i (b_n^j \partial_j u))\circ \xi  \,+\,  \frac12 ( b_n^i b_n^j \partial_{ij} u) \circ \xi\,=\,-\bigl( b_n^i (\partial_i b^j_n )\partial_j u\bigr)\circ\xi  \, -\,  \frac12 ( b_n^i b_n^j \partial_{ij} u)\circ\xi 
		\\&\quad = \,
		\frac{1}{2}\bigl(
		-\,\bigl( b_n^i (\partial_i b^j_n )\partial_j u\bigr)\,+\, (\partial_i b_n^i) b^j_n\partial_j u
		\,-\,\partial_i (b_n^i b_n^j \partial_{j} u)
		\bigr)\circ\xi
		\\&\quad =\,  - \,\frac{1}{2} b_n^i(\xi) (\partial_i b^j_n )(\xi) \psi^k_j\partial_k v\,+\,\frac{1}{2}(\partial_i b_n^i) (\xi) b^j_n(\xi)\psi_j^k \partial_k v
		\\&\qquad
		-\frac{1}{2}\partial_k (\psi_i^k 
		b_n^{i}(\xi)b^j_n(\xi) \psi_j^l \partial_l v
		) \,+\,\frac{1}{2} (\partial_k \psi_i^k) b^i_n(\xi)b^j_n(\xi) \psi_j^l \partial_l v.
	\end{align*}
	We define the coefficients accordingly 
	\begin{align}\begin{split}\label{eq:defi_As}
		\alpha^{ij} \,&=\, \psi_k^i \biggl(
		a^{kl}(\xi) \,-\, \frac{1}{2}b^k_n(\xi)b_n^l (\xi)
		\biggr) \psi^{j}_l,
		\\
		\alpha^i \, &=\, \mu^j (\xi)\psi^{i}_j \,+\, a^j (\xi)\psi^{i}_j \,-\, \frac{1}{2}b_n^j(\xi)(\partial_j b_n^k)(\xi) \psi^i_k \,+\,\frac12 (\partial_j b_n^j (\xi)) b_n^k(\xi)\psi_k^i \,+\,\frac12 (\partial_k \psi_j^k) b^j_n(\xi)b^l_n(\xi) \psi_l^i,
	\end{split}\end{align}
	so that 
	\begin{equation}\label{Eq13}
	\dd v\,=\, \Bigl(\partial_i (\alpha^{ij} \partial_j v) \,+\, \alpha^i \partial_i v \,+\, f^0\circ \xi \,+\, (\partial_i f^i) \circ \xi \,-\,(b_{n}^i\partial_i g_{n})\circ \xi\Bigr) \,   \dd t \, + \,  g_{n}\circ \xi \,\dd w^n.
	\end{equation}
	To also rewrite the inhomogeneities, we observe that
	\begin{align*}
&	(\partial_i f^i ) \circ \xi \,=\, \psi_i^j \partial_j ( f^i (\xi) ) \,=\, \partial_j(\psi_i^j f^i (\xi) ) \,-\, \partial_j(\psi_i^j) f^i (\xi) ,\\
	&
	(b_{n}^i\partial_i g_{n})\circ \xi\,=\, b_n^i(\xi)\psi_i^j \partial_j(g_n(\xi))\,=\, 
	\partial_j( b_n^i(\xi)\psi_i^j g_n(\xi)) \,-\, \partial_j( b_n^i(\xi)\psi_i^j) g_n(\xi)
	\end{align*} 
	and introduce
	\begin{align}\begin{split}\label{eq:new_rhs1}
		F^0 \,&=\, f^0 (\xi) \,-\,  \partial_j(\psi_i^j) f^i (\xi)\,+\, \partial_j( b_n^i(\xi)\psi_i^j) g_n(\xi),
		\\
		F^i \,&=\, \psi_j^i f^j (\xi)\,-\, b_n^j(\xi)\psi_j^i g_n(\xi)  , \qquad i\in \{1, \ldots d\},
		\\
		G_n  \,&=\, g_n ( \xi) ,
	\end{split}
	\end{align}
	so that \eqref{Eq13} becomes
	\[
	\dd v\,=\, \Bigl(\partial_i (\alpha^{ij} \partial_j v) \,+\, \alpha^i \partial_i v \,+\, F^0 \,+\, \partial_iF^i \Bigr) \,   \dd t \, + \,  G_n  \,\dd w^n.
	\]
	\emph{Step 2 (Reduction to $G_n = 0$).}
	We let furthermore 
	$h$ be the solution to the  heat equation 
	\begin{equation}
    \label{eq:h_equation}
		\begin{cases}
			\displaystyle{\dd h\, - \,\Delta h \,\dd t\,=\,  \one_{\square} G_n \,  \dd w_t^n}, & \text{ on }\R^d,\\
			h(0)=0, & \text{ on }\R^d,
		\end{cases}
	\end{equation}
	where we denote $\square=[-1,2)^d$. We notice that, while modifying $G_n$ in a non-periodic way may seem unnatural, it is necessary to control its possibly infinite mass. Moreover, it suffices to prove properties of $v$ on an open set containing the unit cube, like $Q$, because we can transfer properties to the whole space by periodicity.
    To derive the regularity of $h$, we note that
	by \cite[Example 10.1.5]{Analysis2} and the fact that $(1-\Delta)^{1/2}: H^{1,q}(\R^d)\to L^q(\R^d)$ is an isomorphism, the operator $-\Delta: H^{1,q}(\R^d)\subseteq  H^{-1,q}(\R^d)\to H^{-1,q}(\R^d)$ has a bounded $H^{\infty}$-calculus of angle $0$. We also observe
	\begin{align}\begin{split}\label{Eq17}&
		\int_{\square} \biggl(\sum_{n\ge 1} |G_n(\xi(x))|^2\biggr)^{q/2} \, \dd x\,\eqsim_d \, 
		\int_{\T^d} \biggl(\sum_{n\ge 1} |G_n(\xi(x))|^2\biggr)^{q/2} \, \dd x
		\\& \quad = \, \int_{\T^d} \biggl(\sum_{n\ge 1} |g_n(y)|^2\biggr)^{q/2} |\det D\xi^{-1} (y) | \, \dd y\,\lesssim_{\omega,T} \, \int_{\T^d} \biggl(\sum_{n\ge 1} |g_n(y)|^2\biggr)^{q/2}  \, \dd y,
	\end{split}
	\end{align}
	independently of $t\le T$ by \eqref{eq:periodic_xi} and the continuity of $D\xi$ and hence we have
	\[
	 \|G \|_{\ell^2}\,\in \, L^p(0,T;L^q(\square)),\qquad \text{a.s.}
	\]
     Therefore,
	combining either \cite[Corollary 7.4]{NVW13} or \cite[Theorem 1.2]{MaximalLpregularity} with a standard shift argument (cf.\ \cite[Section 5]{AV19}), we obtain that  
	\begin{equation}\label{Eq19}h\,\in \,H^{\theta,p}(0,T;H^{1-2\theta,q}(\R^d))\qquad \text{a.s.,}
	\end{equation}
	for all $\theta\in [0,\frac{1}{2})$. 
	Since $p>2$ and $\frac{2}{p}+\frac{d}{q}<1$ we can choose $\eta>0$ such that 
	\begin{equation}\label{Eq14}
	\frac{1}{p}\,+\,\eta\,<\, \frac{1}{2}  \qquad \text{ and }\qquad \frac{2}{p}\,+\,\frac{d}{q}\,<\,1\,-\,2\eta,
\end{equation}
	and set $\theta=\frac{1}{p}+\eta$. The above choice and Sobolev embeddings yield
	\begin{align}\begin{split}
			\label{eq:embeddings}
		H^{\theta,p}(0,T;H^{1-2\theta,q}(\R^d))&\,\embed\, 
		C^{\eta}(0,T;H^{1-2\theta,q}(\R^d))
		 \\&\,\embed  \, C^{\eta}(0,T;C^{\eta'}(\R^d)) \,
		\embed\, 
		C^{\gamma}((0,T)\times \R^d),
		\end{split}
	\end{align}
	where $\eta':=1-2\theta-\frac{d}{q}>0$ by \eqref{Eq14} and $\gamma= \eta\wedge \eta'$ only depends on $(d,p,q)$. 
	We introduce a new process  $z = v-h$  which solves 
		\begin{equation}\label{Eq16}
	\dd z\,=\, \Bigl(\partial_i (\alpha^{ij} \partial_j z) \,+\, \alpha^i \partial_i z \,+\, \overline{F}^0 \,+\, \partial_i\overline{F}^i \Bigr) \,   \dd t 
	\end{equation}
	on $\square$
	for 
	\begin{align}\begin{split}\label{eq:new_rhs2}
		\overline{F}^0 \,&=\, F^0 \,+\, \alpha^i  \partial_i h,
	\\	\overline{F}^i \, &=\,F^i \,+\,  \alpha^{ij} \partial_j h \,-\, \partial_i h ,  \qquad i\in \{1, \ldots, d\}.
\end{split}\end{align}

	\emph{Step 3 (Application of the deterministic De Giorgi--Nash--Moser theorem).}  As we have transformed the original SPDE \eqref{eq:SPDE_lin} for $u$ into the random PDE \eqref{Eq16} for $z$, we aim to apply the deterministic result \cite[Theorem III.10.1]{LSU} to $z$ on the set $\square$. To this end, we will provide stopping times $\tau_m$ and a sequence $M_m,\nu_m \in (0,\infty)$ such that a.s.
	\begin{enumerate}[(i)]
		\item \label{Item1} $z(0) \in C^{\gamma_0}(\square)$,
		\item \label{Item2} $z\in L^2(0,T; H^1(\square))\cap C([0,T]; L^2(\square))\cap L^\infty([0,T]\times \square)$,
		\item \label{Item3} $\nu_m | {\eta}|^2 \le \alpha^{ij}(t,x) {\eta }_i  {\eta}_j \le M_m | {\eta }|^2 $ for all $  {\eta}\in \R^d$, $t\in [0,\tau_m]$ and $ x\in \Tor^d$,
		\item \label{Item100} We have $
		\overline{F}^0\in L^{p/2}(0,T;L^{q/2}(\square))$ and   $\alpha^i ,\overline{F}^i \in L^p(0,T;L^q(\square))$ for $i>0$.
	\end{enumerate}
	Then an application of \cite[Theorems 10.1, Chapter III]{LSU} will imply 
	\begin{equation}\label{Eq29}
	z\,\in \, C^{(\beta_m\wedge \gamma_0 )/2, \beta_m \wedge\gamma_0 }( (0,\tau_m)\times{\square}_* )
	\end{equation}
	for ${\beta}_m =  \beta(d,M_m,\nu_m, p,q)$ and $\square_*= [-1/2, 3/2)^d$. We remark that passing from $Q$ to the smaller cube $Q_*$ is necessary since we do not have any control on the H\"older norm of $z$ on the boundary $[0,T]\times\partial Q_*$. At the same time, since $Q_*$ contains the unit cube, it is still large enough to extrapolate properties of the periodic function $v$ to the whole space. In any case, part \eqref{Item1} of  the above assumptions
	follows from $z(0) = v(0) = u(0)$. That $z\in L^\infty([0,T]\times \square)$ is a consequence of Proposition \ref{prop:bddness} together with \eqref{Eq19} and \eqref{eq:embeddings}. Using now 
	\begin{align*}
	&	\int_{\square} | \nabla v|^2 \,\dd x\,\eqsim_d \,
			\int_{\T^d} | D \xi \nabla u(\xi)|^2 \,\dd x
			\\&\quad  =\, \int_{\T^d} |D\xi^{-1}(y) \nabla u(y) |^2 |\det D\xi^{-1}| \,\dd y \,\lesssim_{\omega,T} \, \| \nabla u\|_{L^2(\T^d; \R^d)}^2
	\end{align*}
uniformly in $t\le T$ as follows from \eqref{eq:periodic_xi}
combined with an analogous estimate to \eqref{Eq17} we deduce that $v\in L^2(0,T; H^1(\square))$. That $v$ is continuous in $L^2(\square)$ can be obtained based on \eqref{Eq17} analogously to the third step from the proof of Proposition \ref{prop:ItoWentzel} and therefore \eqref{Item2} holds too. To make sure that \eqref{Item3} is satisfied as well, we set
\begin{equation}\label{eq:tau_m_definition}
\tau_m \,=\, \inf\Bigl\{ t\in [0,T] \,\Big|\, \sup_{x\in  [0,1]^d}|\psi(t,x)| \,+\, |D\xi (t,x)| \ge m    \Bigr\},
\end{equation}
 $\nu_m = \nu/m^2$ and $M_m= m^2M$. Then \eqref{eq:defi_As} implies that \eqref{Item3} is indeed fulfilled. Lastly, we observe that ${F}^i$ as defined in \eqref{eq:new_rhs1} satisfies the same integrability as $f^i$ required in \eqref{eq:assumption_rhs} by estimates similar to \eqref{Eq17} and the ($\omega$-wise) boundedness of all additional prefactors. Also the terms entering in \eqref{eq:new_rhs2} respect the condition \eqref{eq:assumption_rhs} by \eqref{Eq19} for $\theta=0$ so that \eqref{Item100} holds too.
	
	\emph{Step 4 (Conclusion).}
	Combining \eqref{Eq19}, \eqref{eq:embeddings} and \eqref{Eq29} we find  
	$$
	v \in C^{({\overline{\g}}_m\wedge \gamma_0)/2,{\overline{\gamma}}_m\wedge \gamma_0 }( (0,\tau_m)\times{\square}_*  ) 
	$$ 
	for $\overline{\gamma}_m =  \gamma\wedge \beta_m $ and by periodicity $v \in C^{({\overline{\g}}_m\wedge \gamma_0)/2,{\overline{\g}}_m\wedge \gamma_0  }( (0,\tau_m)\times \R^d) $. Invoking Proposition \ref{Prop_Psi}, we  have also $\xi^{-1} \in C_{\loc}^{1/2-,1}([0,T]\times \R^d)$ so that $u(t,x) = v(t, \xi^{-1}(t,x)) $ lies a.s.\ in 
    $C^{({\overline{\g}}_m\wedge \gamma_0)/2-,{\overline{\g}}_m\wedge \gamma_0  }( (0,\tau_m)\times \T^d) $.
    Upon defining $\g_m:={(\overline{\g}_m\wedge\g_0)}/{3}$,  the desired \eqref{eq:qual_statement} follows.
    \end{proof}
    
     \begin{proof}[Proof of Theorem \ref{t:DeGiorgiNashMoser_quantitative}]
     We divide the proof into several steps. 

\emph{Part I (Reductions).} As in the previous proof, by an extension by zero on $[\tau,T]$ of the inhomogeneities $f$ and $g$, we assume that $\tau=T$. Next, we argue that the proof of Theorem \ref{t:DeGiorgiNashMoser_quantitative} can be reduced to the proof of a more tractable estimate.
     
     \emph{Step I-A. To prove Theorem \ref{t:DeGiorgiNashMoser_quantitative} it suffices to show the existence of sequences $(\delta_k)_k,(\psi_k)_k\subset(0,\infty)$ depending only on $(d,T,M,\nu,\gamma_0,p,q,R_b)$ such that $\lim_{k\to \infty}\psi_k=\infty$, and for all solutions $u$ to \eqref{eq:SPDE_lin} with $a^0=0$ and $b^0=0$ it holds that}
	\begin{equation}
	\label{eq:claim_1_reduction0_proof_quantitative}
	\P\big(\|u\|_{C^{\delta_k}((0,T)\times \T^d)}> k\big) \,\leq\, 
	 \psi_k^{-1}\,+\,5\,\P(\data_{u_0,f,g}>\psi_k),\quad \text{ for all } k\,\ge\, 1.
	\end{equation}

	Assume that \eqref{eq:claim_1_reduction0_proof_quantitative} holds in the case $ a^0=0$ and $b^0=0$ and recall the reduction to this situation by replacing $(f^0,g)$ by $(f^0+a^0u, g + bu)$ employed at the beginning of the proof of Theorem \ref{t:DeGiorgiNashMoser}. 
	To estimate the resulting contribution of $u$, we provide a bound on the tail probability of $\|u\|_{L^\infty([0,T]\times \T^d)}$, which follows by Lenglart's domination \cite{Lenglart} from Proposition \ref{prop:bddness}. 
	For the reader's convenience, we include the details here.
	For  $\ell>0$, let $\lambda$ be the stopping time given by
	$$
	\lambda:=\inf\{t\in [0,T]\,:\, \data_{u_0,f,g}(t)\geq \ell\}\quad \text{ with }\quad \inf\emptyset:=T,
	$$ 
	where $\data_{u_0,f,g}(t)$ is as in the statement of Theorem \ref{t:DeGiorgiNashMoser_quantitative} with $\tau$ replaced by $t$. 
	Let $u^{(\lambda)}$ be the solution to \eqref{eq:SPDE_lin} with $(f,g)$ replaced by $(\one_{[0,\lambda]}f,\one_{[0,\lambda]}g)$. By Proposition \ref{prop:Existence_Uniqueness}, it holds that $u^{(\lambda)}=u$ on $[0,\lambda]$.
	Thus,
	\begin{equation}
	\label{Eq54}
	\begin{aligned}
	\P(\|u\|_{L^\infty([0,T]\times \T^d)}\geq \eta )
	&\leq \P(\|u\|_{L^\infty([0,T]\times \T^d)}\geq \eta,\, \lambda= T)+ \P(\lambda<T)\\ 
	&\leq \P(\|u^{(\lambda)}\|_{L^\infty([0,T]\times \T^d)}\geq \eta)+ \P(\data_{u_0,f,g}\geq \ell)\\
	&\leq \frac{C}{\eta}\E\|u^{(\lambda)}\|_{L^{\infty}([0,T]\times \T^d)}+ \P(\data_{u_0,f,g}\geq \ell)\\
	&\leq \frac{C\ell}{\eta}+ \P(\data_{u_0,f,g}\geq \ell).
	\end{aligned}
	\end{equation}
	where in the last estimate we used  Proposition \ref{prop:bddness} and the definition of $\lambda$.
	We emphasize that working with $u^{(\lambda)}$ instead of $u$ allows us to show immediately that the constant in the above estimates is independent of the stopping time, as long as $\lambda \le T$.
	Choosing, e.g., $\ell=\sqrt{\eta}$ in the above, we arrive at
	\begin{equation}
	\label{eq:boundedness_u_estimate}
	\P(\|u\|_{L^\infty((0,T)\times \T^d)}\geq {\eta} )\leq  C/\sqrt{\eta}+ \P(\data_{u_0,f,g}\geq \sqrt{\eta}).
	\end{equation}
	Hence, if \eqref{eq:claim_1_reduction0_proof_quantitative} holds with $a^0=0$ and $b^0=0$, then
     	\begin{align*}
	&\P\big(\|u\|_{C^{\delta_k}((0,T)\times \T^d)}> k\big) \\
	&\,\leq\, 
	 \psi_k^{-1}\,+\,5\,\P(\data_{u_0,f+a^0 u,g+b^0 u}>\psi_k)\\
	 &\,\leq\, 
	 \psi_k^{-1}\,+\,5\,\P(\data_{u_0,f,g}>\psi_k/2)\,+\,5\,\P(\|a^0u\|_{L^{p/2}(0,T;L^{q/2}(\T^d))}+\|b^0u\|_{L^{p/2}(0,T;L^{q/2}(\T^d;\ell^2))}>\psi_k/2)\\
	 &\,\leq\, 
	 \psi_k^{-1}\,+\,5\,\P(\data_{u_0,f,g}>\psi_k/2)\,+\,5\,\P(\|u\|_{L^{\infty}([0,\tau]\times \T^d)}>\psi_k/(2C_{0}M)),
	\end{align*}
	where $C_{0}$ depends only on $p$ and $T$, and we used the boundedness of $a^0$ and $b^0$ as in Assumption \ref{ass:DeGiorgi_Nash_Moser}. Now the claim in Step 1 follows from \eqref{eq:boundedness_u_estimate}, up to relabelling the sequence $(\psi_k)_{k}$.
	
	In the rest of the proof, we assume $a^0=0$ and $b^0=0$ without further mentioning it.
	
   \emph{Step I-B. To prove the claim in \eqref{eq:claim_1_reduction0_proof_quantitative} it is enough to show the existence of a constant $C>0$ and sequences $(\g_m)_m,(\chi_m)_m$ of strictly positive numbers depending only on $(d,T,M,\nu,\gamma_0,p,q,R_b)$, such that }
 \begin{equation}
    \label{eq:claim_estimated_quantitative}
\P\bigl(\|u\|_{C^{\g_m}((0,T)\times \T^d)}> \chi_m\bigr)
\leq C /{m}\,+\,5\, \P(\data_{u_0,f,g}>{m}),\qquad m\ge 1.
 \end{equation}
 
Without loss of generality, we can assume that $\chi_m$ increases to $\infty$ and takes values in the integers. Thus, $\psi_k:=\sup \{m\,:\, \chi_m\leq k\}$, with the convention that $\sup\emptyset:=1$, is increasing and satisfies $\lim_{k\to \infty}\psi_k=\infty$.
In particular, $\chi_{\psi_k}\leq k$ for all $k\geq 1$ and the above implies  \eqref{eq:claim_1_reduction0_proof_quantitative} with $\delta_k=\gamma_{\psi_k}$, since then 
$$
\P\bigl(\|u\|_{C^{\delta_k}((0,T)\times \T^d)}> k\bigr)
\leq 
\P\bigl(\|u\|_{C^{\gamma_{\psi_k}}((0,T)\times \T^d)}> \chi_{\psi_k}\bigr),
$$
 and $C/\psi_k \leq \psi_k^{-1/2}$ for sufficiently large $k$, upon replacing $\psi_k$ by its square root.

\emph{Part II (Proof of \eqref{eq:claim_estimated_quantitative}).}
It remains to prove \eqref{eq:claim_estimated_quantitative}, for which we collect some facts for fixed $m\geq 1$.
With the same notation from the proof of Theorem \ref{t:DeGiorgiNashMoser}, we
    define new stopping times 
    \[
\wt{\tau}_m\,=\, \inf\Bigl\{ t\in [0,\tau_m ] \,\Big|\, \| \xi^{-1}\|_{C^{3/8,3/4}((0,t)\times {Q} ;\R^d)} {
	+ \sup_{x\in [0,1]^d}|D^2 \xi (t,x) |}  \, \ge\, m    \Bigr\} ,
\]
   where $\inf\emptyset:=\tau_m$ and $\tau_m$ is as in \eqref{eq:tau_m_definition}. We remark that we choose the exponents $(3/8,3/4)$ for concreteness and that one could equally well take any pair $(\kappa,2\kappa)$ for $\kappa\in[1/3,1/2)$. Then, exploiting that $u(t,x) = v(t, \xi^{-1}_t(x)) $ and periodicity we 
    observe 
    \begin{align}\begin{split}
    \label{eq:inequality_C_m_definition_quantitative_estimate}
    \| u \|_{C^{\gamma_m} ((0,\wt{\tau}_m)\times \T^d)} \,&\leq C_m\, \|v\|_{C^{(\overline{\g}_m\wedge \gamma_0)/2,  \overline{\g}_m\wedge \gamma_0}((0,\wt{\tau}_m)\times \R^d)}\\
    \,&=C_m\, 
    \|v\|_{C^{(\overline{\g}_m\wedge \gamma_0)/2,  \overline{\g}_m\wedge \gamma_0}((0,\wt{\tau}_m)\times {Q}_*)} 
    \\&
      \leq \,C_m\big( \|z\|_{C^{(\beta_m\wedge \gamma_0 )/2, \beta_m \wedge\gamma_0 }( (0,\wt{\tau}_m)\times{Q}_* )}\,+\, \|h\|_{C^\gamma ((0,\wt{\tau}_m)\times {Q}_*)}\big)
    \end{split}\end{align}
    for $\gamma_m  = (\overline{\g}_m\wedge \gamma_0) /3$,   $C_m\in [1,\infty)$ depending only on $m$, and $\beta_m$ as in \eqref{Eq29}. 
     In the above, we used the notation introduced in Steps 2 and 3 in the proof of Theorem \ref{t:DeGiorgiNashMoser}.
    
    Additionally, recalling that $\psi D\xi = \id_{\R^d}$ by \eqref{eq:def_psi}, and based on the observation  \[
    (\partial_i \psi ) D \xi \,=\,  \partial_i ( \psi D\xi ) - \psi \partial_i D\xi \qquad\Rightarrow \qquad 
        (\partial_i \psi ) \,=\,- \psi (\partial_i D\xi )\psi,
    \]
    we can also estimate 
    \begin{equation}\label{Eq52}
    \|\partial_i \psi^i_j\|_{L^\infty([0,\wt{\tau}_m]\times Q)} \,\le\, C_m,\qquad i,j=1,\dots, d,
    \end{equation}
    up to enlarging $C_m$ if necessary.
    We proceed by bounding the two terms on the right-hand side of \eqref{eq:inequality_C_m_definition_quantitative_estimate} individually.

    \emph{Step II-A (a-priori estimate on $h$). There exists $\gamma>0$ depending only on $(d,p,q)$ for which the following assertion holds. There exist a sequence $(N_m)_m$ depending only on $(d,p,q)$ and a constant $C$ depending only on $(d,T,M,\nu,p,q)$ such that, for all $m,\eta\geq 1$:}
		\begin{align}\label{Eq60}
	\P\bigl( \|h\|_{C^{\gamma}( (0,\wt{\tau}_m)\times{\square } )}> \eta\bigr) &\le  N_m \eta^{-p} + \P( \data_{u_0,f,g} > m),
	\\ \label{Eq61}
	\P\bigl( \|h\|_{L^p(0,\wt{\tau}_m;H^{1,q}(Q))}> \eta\bigr) &\le  N_m \eta^{-p} + \P( \data_{u_0,f,g} >  m ).
	\end{align}

    Before going into the proof of the above, we emphasize that the choice of the $p$-th power in the claimed estimates is not essential for our purposes, and other choices can be made. To prove the claimed estimate, we argue as in Step 1. 
    For each stopping time $\lambda\leq \wt{\tau}_m$, let us denote by $h^{(\lambda)}$ the solution to \eqref{eq:h_equation} with $G$ replaced by the process $\one_{[0,\lambda]}G$. Note that $h=h^{(\lambda)}$ on $[0,\lambda]\times \O$, where $h$ is the solution to \eqref{eq:h_equation}. Hence, from the stochastic maximal $L^p$-regularity estimates as used in \eqref{Eq19}-\eqref{eq:embeddings} and the bound \eqref{Eq17},  it follows
    \begin{align}\begin{aligned}\label{Eq62}
\E\Bigl[\|h\|_{C^{\gamma}((0,\lambda)\times \R^d)}^p + \|h\|_{L^p(0,\lambda; H^{1,q}(\R^d))}^p\Bigr]
&\leq
\E\Bigl[\|h^{(\lambda)}\|_{C^{\gamma}((0,T)\times \R^d)}^p + \|h^{(\lambda)}\|_{L^p(0,T; H^{1,q}(\R^d))}^p\Bigr]
\\
&\leq C \E\|\one_{[0,\lambda]}G\|_{L^p(0,T;L^q(Q;\ell^2))}^p\\
&\leq C_m \E\|g \|_{L^p(0,\lambda;L^q(\T^d;\ell^2))}^p ,
\end{aligned}
    \end{align}
    where $C_m$ depends only on $(d,p,q)$ and $m\geq 1$, and we used that $\lambda\leq \wt{\tau}_m$.
    As in Step 1, the use of the stopping time $\lambda$ allows us to immediately obtain the independence of the constant in the above estimate on the stopping time $\lambda$ as long as $\lambda \le \wt{\tau}_m$.
    
 For each $m\geq 1$, we employ the above estimate with
 $$
\lambda:= \inf\{t\in [0,\wt{\tau}_m]\,:\, \|g\|_{L^p(0,\lambda;L^q(\T^d;\ell^2))}\geq  m\},\quad \text{ with }\quad \inf\emptyset:=\wt{\tau}_m,
 $$
 resulting in the bound 
\begin{align}\begin{split}\label{Eq51}
   \P(\|h\|_{C^\gamma((0,T)\times \T^d)}\geq \eta)&\,\leq\,
   \P(\|h\|_{C^\gamma((0,T)\times \T^d)}\geq \eta, \, \lambda=\wt{\tau}_m)+ \P(\lambda <\wt{\tau}_m)\\
   &\,\leq\,
   \P(\|h\|_{C^\gamma((0,\lambda)\times \T^d)}\geq \eta)+ \P(\|g\|_{L^p(0,\wt{\tau}_m;L^q(\T^d;\ell^2))} \geq  m)\\
   &\,\leq\,\frac{ C_m \, m^p}{\eta^p}+\P(\|g\|_{L^p(0,\tau;L^q(\T^d;\ell^2))} \geq  m)
   \\
   &\,\leq\, \frac{C_m\, m^p}{\eta^{p}} +\P(\data_{u_0,f,g}\geq  m).
   \end{split}
\end{align}
for $m,\eta\ge 1$. Hence, \eqref{Eq60} follows by setting $N_m= C_m m^p$. The above argument and \eqref{Eq62} imply analogously \eqref{Eq61}. 

	\emph{Step II-B (A-priori estimate on $z$).
		There exist a constant $C$ depending only on $(d,T,M,\nu,p,q)$ and a family of positive constants $(C_{m,\eta}^*)_{m,\eta}$
		depending on ${(d,T,M,\nu ,p,q,R_b)}$ such that 
	\begin{align}\label{Eq159}
	\P \Bigl(
	\|z\|_{C^{(\beta_m\wedge \gamma_0 )/2, \beta_m \wedge\gamma_0 }( (0,\wt{\tau}_m)\times {Q}_* )} > C_{m,\eta}^*
	\Bigr)
	 \le 2N_m \eta^{-p}+ C/\sqrt{\eta} +
	2 \P(\data_{u_0,f,g} > \sqrt{\eta}) +2 \P(\data_{u_0,f,g} > {m}),
\end{align}for all $m,\eta\ge 1$,
	where $N_m$ is from the previous step.
}

    Recall that $\square=[-1,2)^d$. 
    We start with the observation that
    \[
    \| z\|_{L^\infty([0,\wt{\tau}_m]\times Q) } \le
    \| u\|_{L^\infty([0,\wt{\tau}_m]\times \T^d) } + 
    \| h\|_{L^\infty([0,\wt{\tau}_m]\times Q) } ,
    \]
  and notice that the terms on the right-hand side are estimated by \eqref{Eq54} and \eqref{Eq60}. 
    With our aim to apply \cite[Theorem III.10.1]{LSU}, we also observe that regarding the data for the equation \eqref{Eq16} satisfied by $z$, as defined in  \eqref{eq:defi_As}, \eqref{eq:new_rhs1} and \eqref{eq:new_rhs2}, we have
    \begin{align*}
    	\|\alpha^i\|_{L^\infty([0,\wt{\tau}_m ] \times Q)} &\le C_{m,M,R_b},\\
    	\|\overline{F}^0 \|_{L^{p/2}(0,\wt{\tau}_m;L^{q/2}(Q))}&\le 
    	C_{m,M,R_b} \biggl(
    	\|f^0\|_{L^{p/2}(0,\wt{\tau}_m;L^{q/2}(\T^d)) } +\sum_{i=1}^d \|f^i\|_{L^{p/2}(0,\wt{\tau}_m;L^{q/2}(\T^d)) } \\&  \qquad\qquad + \|g\|_{L^{p/2}(0,\wt{\tau}_m;L^{q/2}(\T^d;\ell^2)) } 
    	\biggr)
    	+
    	C_M\| h\|_{L^{p/2}(0,\wt{\tau}_m;H^{1,q/2}(Q))},
    	\\
    	\|\overline{F}^i \|_{L^{p}(0,\tau;L^{q}(Q))}&\le C_{m,M,R_b} \biggl(
    	\| f^i\|_{{L^{p}(0,\wt{\tau}_m;L^{q}(\T^d))}} + \| g \|_{{L^{p}(0,\wt{\tau}_m;L^{q}(\T^d;\ell^2))}}
    	\biggr)
    \\&\qquad 	+ C_M\| h\|_{L^{p}(0,\wt{\tau}_m;H^{1,q}(Q))}
    	,\qquad i\in \{1,\dots, d\}
    	,
    \end{align*}
    where we involved in particular \eqref{Eq52}. 
    Thus, applying  \cite[Theorem III.10.1]{LSU} we find 
    \begin{align}
    \|z\|_{C^{(\beta_m\wedge \gamma_0 )/2, \beta_m \wedge\gamma_0 }( (0,\wt{\tau}_m)\times{Q}_* )} \,\le\, C_{m,\eta}^*
\end{align}
where $\beta_m $ depends on $(d,m,M,\nu,p,q)$ and $C^*_{m,\eta}$ on $(d,T,M,\nu ,p,q,R_b,m,\eta)$, on the event   
    \begin{align*} 
    \biggl\{\max\Bigl\{   \| u\|_{L^\infty([0,\wt{\tau}_m]\times Q) },  \| h\|_{L^\infty([0,\wt{\tau}_m]\times Q) } ,	\data_{u_0,f,g} ,\| h\|_{L^{p}(0,\wt{\tau}_m;H^{1,q}(Q))} \Bigr\}
    	\le \eta\biggr\},\qquad \eta\ge 1.
    \end{align*}    
	Involving now \eqref{eq:boundedness_u_estimate}, \eqref{Eq60} and \eqref{Eq61} 
	we find that
	\begin{align*}
	&
		\P \Bigl(
		\|z\|_{C^{(\beta_m\wedge \gamma_0 )/2, \beta_m \wedge\gamma_0 }( (0,\wt{\tau}_m)\times {Q}_* )} > C_{m,\eta}^*
		\Bigr)\\
		 &\quad\le \P\Bigl( \max\Bigl\{  \| u\|_{L^\infty([0,T]\times Q) }, \| h\|_{L^\infty([0,\wt{\tau}_m]\times Q) } ,	\data_{u_0,f,g} ,\| h\|_{L^{p}(0,\wt{\tau}_m;H^{1,q}(Q))} \Bigr\}> \eta\Bigr)\\
		&\quad \le 2N_m \eta^{-p}\,+\,C/\sqrt{\eta}\,+\, 
		2\, \P(\data_{u_0,f,g} > \sqrt{\eta}) \,+\, 2\, \P(\data_{u_0,f,g} > {m}),
	\end{align*}
	where  $N_m$ and $C$ are from the previous step and \eqref{eq:boundedness_u_estimate}, respectively.

	\emph{Step II-C (Conclusion).}  To show \eqref{eq:claim_estimated_quantitative}, we fix $m\ge 1$ and let $\chi>0$ be determined  later.  We can estimate
	\begin{align*}
		&\P\bigl(\|u\|_{C^{\g_m}((0,T)\times \T^d)}> \chi\bigr) \\
	&\quad \leq \,
	\P\bigl(\|u\|_{C^{\g_m}((0,T)\times \T^d)}> \chi,\, \wt{\tau}_{m}=  T\bigr)\, +\,  \P(\wt{\tau}_{m}<T)\\
	&\quad  \leq \,
	\P\bigl(\|u\|_{C^{\g_m}((0,\wt{\tau}_m)\times \T^d)}> \chi\bigr)\, +\,  \P(\wt{\tau}_{m}<T)\\
		&\quad  \leq \,
	\P\bigl( \|z\|_{C^{(\beta_m\wedge \gamma_0 )/2, \beta_m \wedge\gamma_0 }( (0,\wt{\tau}_m)\times{Q}_* )} > \chi/(2C_m)\bigr)\, +\,\P\bigl(\|h\|_{C^\gamma ((0,\wt{\tau}_m)\times {Q}_*)}> \chi/(2C_m)\bigr)\,+\,   \P(\wt{\tau}_{m}<T),
	\end{align*}
	based on  \eqref{eq:inequality_C_m_definition_quantitative_estimate}. While we bound the first two quantities on the right-hand side in the previous steps, we observe regarding the last term that
	\begin{align*}&
		\P(\wt{\tau}_{m}<T)
		\\&\quad \leq\, \P\big(\|\psi\|_{L^\infty([0,T]\times Q; \R^{d\times d})} \,+\, \|D\xi \|_{L^\infty([0,T]; C^1( Q; \R^{d\times d}))}\,+\,\|\xi^{-1}\|_{C^{3/8,3/4}((0,T)\times Q;\R^d))} > m\big)\\
		\,&\quad \leq\,m^{-1}\big(\E\|\psi\|_{L^\infty([0,T]\times Q;\R^{d\times d})} \,
		+\, \E\|D\xi \|_{L^\infty([0,T]; C^1( Q; \R^{d\times d}))}\,+\,\E\|\xi^{-1}\|_{C^{3/8,3/4}((0,T)\times Q;\R^d)}\big)
		\,\leq\,C m^{-1},
	\end{align*}
	due to Propositions \ref{prop:periodic_flows}--\ref{cor:quant_inverse} with $C$ depending only on $(T,R_b)$. Together with \eqref{Eq60} and \eqref{Eq159} we conclude that for any $\eta \ge 1$, as soon as 
	\begin{equation}\label{Eq53}
		\chi \ge 4 C_m C_{m,\eta}^*,
	\end{equation}
	we have the estimate
		\begin{align*}
		\P\bigl(\|u\|_{C^{\g_m}((0,T)\times \T^d)}> \chi\bigr)  \,
		\leq \,&Cm^{-1} \,+\,(2C_m)^p N_m \chi^{-p}+2N_m\eta^{-p}\, +\, C/\sqrt{\eta} \\&+\,2\,\P(\data_{u_0,f,g} > \sqrt{\eta} ) \,+\,3\,\P(\data_{u_0,f,g} > m ),
		\end{align*}
	up to enlarging $C$.
	Obtaining \eqref{eq:claim_estimated_quantitative} amounts now for $m\ge 1$ to choose first $\eta_m = \max\{ m^2, (N_mm)^{1/p}\}$ and then $\chi_m = \max\{ C_m(N_m m)^{1/p} , 4C_m C_{m,\eta_m}^*\}$.
\end{proof}

\section{Flow method and uniform H\"older exponent}\label{sec:unif_constants}
In the De Giorgi–Nash–Moser estimates of Theorems \ref{t:DeGiorgiNashMoser} and \ref{t:DeGiorgiNashMoser_quantitative}, we establish that solutions to \eqref{eq:SPDE_lin} are almost surely Hölder continuous. However, the associated Hölder exponent may become arbitrarily small on a set of small probability. In this section, we investigate conditions under which the flow-based approach can yield a uniform H\"older exponent instead. To this end, we start by observing that the possibility that $\overline{\g}_m\to 0$ in Theorem \ref{t:DeGiorgiNashMoser} stems from the presence of the flow via $\psi = (D\xi)^{-1}$ in the diffusion coefficient $\alpha^{ij}$, as defined in \eqref{eq:defi_As}, and the resulting loss of control over the 
\emph{ellipticity ratio} of the transformed equation \eqref{Eq16}. 
\begin{remark}[Ellipticity ratio]\label{rem:ell_ratio}
    The   H\"older regularity of the solution $u$ to the parabolic PDE
    $\partial_t u = \nabla \cdot (a \nabla u )$ guaranteed by the deterministic De Giorgi--Nash--Moser estimates depends on $a$ satisfying $  \nu |y|^2\le y^\top a(t,x)y \le M |y|^2$   or all $y\in \R^d$ only through the ellipticity ratio $M/\nu$.
This is because
the rescaled $u_r(t,x) = u(rt,x)$  solves the equation  $\partial_t u_r = \nabla \cdot ( (ra) \nabla u_r )$, and while the upper and lower bound on $r a$ change individually, their ratio remains the same.  At the same time, we have $u\in C^\gamma $ iff $ u_r\in C^\gamma$ for each $r,\gamma>0 $ so that one can use the convention to apply \cite[Theorem III.10.1]{LSU} to $u_{\nu}$, which has lower bound $1$ and upper bound $M/\nu$. We remark that while
    this qualitative statement also applies in the presence of lower-order terms and a right-hand side, any quantitative bound is affected by the rescaling of time. 
    \end{remark}

Before stating sufficient conditions to achieve a uniform Hölder exponent in Theorem \ref{t:DeGiorgiNashMoser}, we present an explicit example of a flow for which the ellipticity ratio of the transformed equation \eqref{Eq16} becomes arbitrarily large with positive probability. 

\begin{proposition}[Exploding ellipticity ratio]
\label{prop:blow_up_elliptic_ratio}
We consider in $d=2$ the SPDE 
\begin{align}\begin{split}\label{eq:SPDE_example}
    \dd u\,=\, \bigl(\partial_1 (a^{11} \partial_1 u) \,+\, \partial_2 (a^{22} \partial_2 u) \bigr)\,\dd t \,&-\, \sin(2\pi x_1)\partial_{1} u \,\dd w_t^1 \,-\, 
    \cos(2\pi x_1)\partial_1 u \,\dd w_t^2\\&-\, \sin(2\pi x_2)\partial_2 u \,\dd w^3_t \,-\, \cos(2\pi x_2)\partial_2 u\, \dd w^4_t,
\end{split}\end{align}
where $a^{11},a^{22} \colon[0,T]\times\O\times \Tor^d\to \R$ are $\Progress\otimes\Borel(\Tor^d)$-measurable such that for some $\delta \in (0,1)$, $\P$-a.s.,
\begin{equation}
    \label{Eq_assumption_prior_a_example}
    \frac{1}{2} +\delta  \,\le\, a^{ii}(t,x)  \,\le\, \frac{1}{2} + \delta^{-1},
\end{equation}
for all $i\in \{1,2\}$, $t\in [0,T]$ and $x\in \R^d$.
Then, defining for  the 
diffusion coefficient $\alpha^{ij}$ of the transformed equation \eqref{Eq16}  the  quantity
\begin{equation}
	\label{Eq59}
\Lambda(t,x):=
{\sup_{\eta\in \R^2 : |\eta| = 1 } |\eta_i \alpha^{ij}(t,x) \eta_j | }  \bigg/ {   \inf_{\eta\in \R^2 : |\eta| = 1 } |\eta_i \alpha^{ij}(t,x) \eta_j | },
\end{equation}
we have 
$
\P(\Lambda(t,x) >k)>0
$
for all  $(t,x)\in (0,T] \times\T^d$ and $k>0$.
\end{proposition}

The above result shows that the presence of sequences $\overline{\g}_m,\delta_k \searrow 0$  in  Theorems \ref{t:DeGiorgiNashMoser}--\ref{t:DeGiorgiNashMoser_quantitative} cannot be improved using our approach. 
We note that, while this highlights limitations of the flow method, we do not give an upper bound on the H\"older regularity of $u$ solving \eqref{eq:SPDE_example}, for which one usually needs to work with a specific solution, cf.\ \cite[Example 1, pp.\ 68--70]{book_homgen}. 

 An analysis of the proof of Theorem \ref{t:DeGiorgiNashMoser} yields conversely the following sufficient conditions for the ellipticity ratio of \eqref{Eq16} to be uniformly bounded.
In its statement we denote by 
$$[X,\tilde{X} ] = (X\cdot \nabla ) \tilde{X} -( \tilde{X}\cdot \nabla  ) X$$ 
the \emph{Lie bracket} of two sufficiently regular vector fields $X,\tilde{X}\colon \T^d\to\R^d$.

\begin{proposition}[Sufficient conditions for uniform H\"older exponent]\label{prop:sufficient_uniform_gamma}
Let the assumptions of Theorem \ref{t:DeGiorgiNashMoser} be satisfied.
Then the  sequence $(\overline{\g}_m)_{m}$ in Theorem \ref{t:DeGiorgiNashMoser} can be taken constant whenever the stochastic flow of diffeomorphisms $\xi$ satisfies
\begin{equation}\label{eq:cond_number}\biggl(
\sup_{t\in [0,T],x\in \T^d} | (D\xi_t(x))^{-1}|\biggr)\biggl/ \biggl( \inf_{t\in [0,T],x\in \T^d} \frac{1}{ |D\xi_t(x)| } \biggr)  \,\le \, C, 
\end{equation}
almost surely for some deterministic constant $C<\infty$. The latter applies in particular in the following cases.
\begin{enumerate}[{\rm(1)}]
   \item \label{Cond_uniform_3}
   {\rm(Constant noise coefficients)} For each $n\in \N$,  $b_n(t,\omega, x)$ depends only on $(t,\omega)$.
    \item  {\rm (Commuting flows)} \label{Cond_uniform_2} For some $N\in \N$ we have $b_n = 0$ for all $n>N$, while for  $k\ne n\le N$ it holds:
    \begin{itemize}
        \item  The commutator  
$[b_k,b_n]$ vanishes.
    \item Each $b_n(t,x,\omega) = b_n(x)$  depends only on space.
        \item For $\phi_t^n(x)$ being the solution to  $
         \dot{\phi}_t^n(x) \,=\, -b_n(\phi_t^n(x))$ starting from $ \phi_0^n(x)\,=\, x$,
        we have 
        \begin{equation}\label{Eq55}
        \biggl(
        \sup_{t\in [0,\infty),x\in \T^d} | (D\phi_t^n (x))^{-1}|\biggr)\biggl/ \biggl( \inf_{t\in [0,\infty),x\in \T^d} \frac{1}{ |D\phi_{t}^n(x)| } \biggr)  \,<\,\infty.
        \end{equation}
    \end{itemize}
\end{enumerate}
\end{proposition}
\begin{remark}[Commuting flows]\begin{enumerate}[(i)]\item We comment on the role of the assumptions of \eqref{Cond_uniform_2}: That $[b_k,b_n] =0 $ and $b_n = b_n(x)$ 
implies that the individual deterministic flows $\phi^k$ and $\phi^n$ commute with each other, allowing for an explicit description of the stochastic flow of diffeomorphisms $\xi$ solving the Stratonovich equation $\dd \xi =- \sum_{n=1}^N b_n(\xi) \circ \dd w^n $, cf.\ \cite[Example III.3.5]{Kunita} and the proof of Proposition \ref{prop:sufficient_uniform_gamma} below. The requirement \eqref{Eq55} is therefore directly linked to \eqref{eq:cond_number} and is, e.g., satisfied for time periodic  $\phi^n_t$. 

\item 
Noise coefficients for which \eqref{Cond_uniform_2} is satisfied are, for example, any
$b_n(x) =  \widetilde{b}_n(x) e_n$, $n\le N=d$ for positive functions $\wt{b}_n \colon \T \to \R $ bounded away from zero. Indeed, then the commutation relation is trivially satisfied while each $\phi^n$ has a finite (in time) period. 

\item 
Lastly, due to the explicit form of the stochastic flow used in the proof of Proposition \ref{prop:sufficient_uniform_gamma}, one sees that the $C^2$-regularity of $\xi$ in space, which is the key in the proofs of Theorems \ref{t:DeGiorgiNashMoser}--\ref{t:DeGiorgiNashMoser_quantitative}, holds as soon as $\phi^n$ has $C^2$-regularity in space for each $n\le N$.  For the latter, it suffices to assume that $b_n \in C^2(\T^d ; \R^d)$, cf.\ \cite[Theorem V.4.1]{Hartman_book}, so that one can assume less regularity than in the general setting considered in Section \ref{app:DeGiorgi_Nash_Moser} as well as Remark \ref{remark_improved_condition}.
\end{enumerate}
\end{remark}

\subsection{Proofs of Propositions \ref{prop:blow_up_elliptic_ratio} and \ref{prop:sufficient_uniform_gamma}}\label{ss:proofs_uniform_gamma}
This subsection is devoted to the proofs of the special cases of the flow method stated above.
\begin{proof}[Proof of Proposition \ref{prop:blow_up_elliptic_ratio}]
\label{ex:explosion_elliptic_ratio}
The equation \eqref{eq:SPDE_example} corresponds to the choice of coefficients 
\begin{align*}
    b^1_1(x)\,&=\, -\sin(2\pi x_1) ,\quad b^2_1(x)\,=\, 0,\\
    b^1_2(x)\,&=\, -\cos(2\pi x_1) ,\quad b^2_2(x)\,=\,0,\\
    b^1_3(x)\,=\, 0 ,\quad b^2_3(x)\,&=\,  -\sin(2\pi x_2), \\
    b^1_4(x)\,=\, 0 ,\quad b^2_4(x)\,&=\, -\cos(2\pi x_2), 
\end{align*}
with $a^{12} =a^{21} =0$
in  \eqref{eq:SPDE_lin}, in which case \begin{equation}\label{eq32}
\bigl(a^{ij}(t,x)-\frac{1}{2} b^i_n(x) b^j_n(x)\bigr)_{i,j} = \delta_{ij} (a^{ij}(t,x)-1/2),
\end{equation}
so that all assumptions of Theorem \ref{t:DeGiorgiNashMoser} are satisfied.
The corresponding stochastic flow of diffeomorphisms $\xi_t(x)=(\xi^1_t(x), \xi^{2}_t(x))$ is  in the proof of Theorem \ref{t:DeGiorgiNashMoser} defined as the solution to
\begin{align*}&
    \dd\xi_t^1(x) \,=\, \sin(2\pi \xi^1_t(x)) \,\dd w^1_t \,+\,  \cos(2\pi \xi^1_t(x)) \,\dd w^2_t,\qquad \xi_0^1(x)\,=\, x_1 ,
    \\
    &
    \dd\xi_t^2(x) \,=\, \sin(2\pi \xi^2_t(x)) \,\dd w^3_t \,+\,  \cos(2\pi \xi^2_t(x)) \,\dd w^4_t,\qquad \xi_0^2(x)\,=\, x_2 ,
\end{align*}
where the drift term precisely vanishes. Differentiating the above equation with respect to $x$ yields that
\begin{align}\begin{split}\label{Eq18}&
    \dd (\partial_1\xi^1_t(x))\,=\, 2\pi \cos(2\pi \xi^1_t(x))  (\partial_1\xi^1_t(x) ) \,\dd w_t^1 \,-\, 2\pi \sin(2\pi \xi^1_t(x))  (\partial_1\xi^1_t(x) )\, \dd w_t^2, \\&
    \dd (\partial_2\xi^2_t(x))\,=\, 2\pi \cos(2\pi \xi^2_t(x))  (\partial_2\xi^2_t(x) )\, \dd w_t^3 \,-\, 2\pi \sin(2\pi \xi^2_t(x))  (\partial_2\xi^2_t(x) ) \,\dd w_t^4 ,
    \end{split}
\end{align}
both starting from 1, 
and $\partial_{2}\xi_t^1 =  \partial_{1}\xi_t^2=0 $. By Levy's characterization theorem, the emerging drivers
\begin{align*}&
    W^1_t(x) \,=\, \int_0^t \cos(2\pi \xi^1_s(x))  \,\dd w^1_s\,-\, 
    \int_0^t \sin(2\pi \xi^1_s(x))  \,\dd w^2_s,\\&
    W^2_t(x) \,=\, \int_0^t \cos(2\pi \xi^2_s(x))  \,\dd w^3_s\,-\, 
    \int_0^t \sin(2\pi \xi^2_s(x))  \,\dd w^4_s,
\end{align*}
are for each $x\in \T^d$ a pair of independent Brownian motions. 
Then \eqref{Eq18} reads 
\begin{align*}
    \dd (\partial_1\xi^1_t(x))\,=\, 2\pi (\partial_1\xi^1_t(x)) \,\dd W^1_t(x),\qquad \dd (\partial_2\xi^2_t(x))\,=\, 2\pi (\partial_2\xi^2_t(x)) \,\dd W^2_t(x),
\end{align*}
and the latter equations are  solved by the geometric Brownian motions
\[
\partial_1\xi^1_t(x)\,=\, \exp\bigl({2\pi W^1_t(x) - 2\pi^2 t}\bigr),\qquad \partial_2\xi^2_t(x)\,=\, \exp\bigl({2\pi W^2_t(x) - 2\pi^2 t}\bigr).
\]
Consequently, we can explicitly 
 evaluate the transformed diffusion matrix \eqref{eq:defi_As} using also \eqref{eq32} as
\[
\alpha(t,x) \,=\, \left(\begin{array}{cc}
  (a^{11}(t,\xi_t(x)) -1/2)  \exp\bigl({4\pi^2 t-4\pi W^1_t(x)}\bigr)   &  0 \\
   0  & (a^{22}(t,\xi_t(x)) -1/2) \exp\bigl({4\pi^2 t- 4\pi W^2_t(x)}\bigr)
\end{array}\right).
\]
and we bound the quantity \eqref{Eq59} from below by
\begin{align*}&
\Lambda(t,x)\,\ge\, \delta^2 \exp\Bigl(
4\pi \Bigl(
\max\{ W^1_t(x) , W^2_t(x)\} \,-\, \min\{ W^1_t(x) , W^2_t(x)\} 
\Bigr)
\Bigr),
\end{align*}
where $\delta>0$ is the constant from \eqref{Eq_assumption_prior_a_example}. 
Due to the independence of  $(W^1_t(x), W^2_t(x))$, the claim follows. 
\end{proof}
\begin{proof}[Proof of Proposition \ref{prop:sufficient_uniform_gamma}]
We start by proving the general assertion that it suffices to check \eqref{eq:cond_number} to ensure that one can take  $\overline{\gamma}_m = \g $ to be constant in the statement of Theorem \ref{t:DeGiorgiNashMoser}.  
    Following Remark \ref{rem:ell_ratio}, it suffices for the latter to check that the ellipticity ratio of the transformed diffusion matrix 
    \[
    \psi_k^i \biggl(
		a^{kl}(\xi) \,-\, \frac{1}{2}b^k_n(\xi)b_n^l (\xi)
		\biggr) \psi^{j}_l\,=\, \bigg[(D\xi)^{-1, \top} \biggl(a(\xi) - \frac{1}{2}(b(\xi) , b(\xi))_{\ell^2 }  \biggr) D\xi^{-1} \bigg]_{i,j},
    \]
    see \eqref{eq:defi_As},
    stays uniformly bounded. To this end, we estimate
    \begin{align*}&
{\sup_{t\in [0,T],x\in \T^d} \sup_{\eta\in \R^d : |\eta| = 1 } \biggl|\eta^\top (D\xi_t(x) )^{-1,\top} \biggl(a(\xi_t(x)) - \frac{1}{2}(b(\xi_t(x) ) , b(\xi_t(x) ))_{\ell^2}  \biggr) (D\xi_t(x))^{-1} \eta\biggr| } 
\\&\quad =\, {\sup_{t\in [0,T],x\in \T^d} \sup_{\tilde{\eta}\in \R^d : |\tilde{\eta}| 
\le |( D\xi_t(x))^{-1}| } \biggl|\tilde{\eta}^\top \biggl(a(\xi_t(x)) - \frac{1}{2}(b(\xi_t(x)) , b(\xi_t(x)))_{\ell^2}  \biggr) \tilde{\eta} \biggr| } \,\\&\quad \le\, M \sup_{t\in [0,T],x\in \T^d} |( D\xi_t(x))^{-1}|^2 ,
    \end{align*}
    and  
    \begin{align*}&
        \inf_{t\in [0,T],x\in \T^d} \inf_{\eta\in \R^d : |\eta| = 1 } \biggl|\eta^\top (D\xi_t(x) )^{-1, \top} \biggl(a(\xi_t(x)) - \frac{1}{2}(b(\xi_t(x) ) , b(\xi_t(x) ))_{\ell^2 }  \biggr) (D\xi_t(x))^{-1} \eta \biggr| 
        \\&\quad \ge \, 
        \inf_{t\in [0,T],x\in \T^d} \inf_{\tilde{\eta}\in \R^d : |\tilde{\eta}| \ge  1/|D\xi_t(x)| } \biggl|\tilde{\eta}^\top \biggl(a(\xi_t(x)) - \frac{1}{2}(b(\xi_t(x)) , b(\xi_t(x)))_{\ell^2 }  \biggr) \tilde{\eta} \biggr|
        \\&\quad  \ge\,  \nu \inf_{t\in [0,T],x\in \T^d} { |D\xi_t(x)|^{-2} },
    \end{align*}
    where $M,\nu$ are the constants appearing in \eqref{eq:ellip}. Thus, whenever \eqref{eq:cond_number} holds, the ellipticity ratio of $\alpha$ remains uniformly in $(t,\omega)$ bounded as desired. 

    Secondly, we check that the above condition \eqref{eq:cond_number} is indeed satisfied if \eqref{Cond_uniform_3} or \eqref{Cond_uniform_2}   holds.  In the former situation \eqref{Cond_uniform_3}, we can  represent the flow explicitly as 
    \[
    \xi_t(x) \,=\, x \,-\, \sum_{n\ge 1} \int_0^t  b_{n,s} \,\dd w^n_s .
    \]
    Because the stochastic integrals are independent of $x$,  the flow is therefore a random, time-dependent, shift of the initial value. Since then $D\xi_t = \id_{\R^d}$, the condition \eqref{eq:cond_number} is satisfied.
    
    Concerning the situation \eqref{Cond_uniform_2}, we observe that due to the commutation assumption, we have 
    \[
    \Phi_{t_1,\dots, t_N}(x) \, = \,\phi_{t_1}^1 \circ \dots \circ \phi_{t_N}^N(x) \,=\, \phi_{t_{l_1}}^{l_1} \circ \dots \circ \phi_{t_{l_N}}^{l_N}(x) ,
    \]
    for any permutation $(l_1,\dots , l_N)$ of $(1,\dots, N)$ and $t_1,\dots t_N \ge 0$. The latter allows us to efficiently compute derivatives of the above with respect to $t_n$, for any $n\le N$. Indeed, choosing a permutation with $l_1=n$, we find 
    \[
     \frac{\dd }{\dd t_{l_1} } \Phi_{t_1,\dots, t_N}\,=\, \biggl(\frac{\dd }{\dd t_{l_1} } \phi_{t_{l_1}}^{l_1} \biggr) \circ \dots \circ \phi_{t_{l_N}}^{l_N} \,=\, -b_{l_1}\biggl( \phi_{t_{l_1}}^{l_1} \circ \dots \circ \phi_{t_{l_N}}^{l_N}  \biggr)\,=\, -b_{l_1}( \Phi_{t_1,\dots, t_N}).
    \]
    A repetition of this argument shows that, moreover
    \[
    \frac{\dd^2 }{\dd t_{l_1}^2 } \Phi_{t_1,\dots, t_N} \,=\, (b_{l_1}\cdot \nabla b_{l_1}) ( \Phi_{t_1,\dots, t_N}),
    \]
     and an application of It\^o's formula yields therefore
    \[
    \dd \Phi_{w_t^1,\dots, w_t^N} \,=\, \frac{1}{2}\sum_{n=1}^N (b_n\cdot \nabla b_n) (\Phi_{w_t^1,\dots, w_t^N}) \,\dd t \,-\, \sum_{n=1}^N b_n (\Phi_{w_t^1,\dots, w_t^N})\,\dd w^n.
    \]
    Consequently, $\Phi_{w_1,\dots, w_N}(x)$ coincides with the stochastic flow $ \xi_t(x)$ defined in the proof of Theorem \ref{t:DeGiorgiNashMoser}. Differentiation with respect to $x$ yields then
    \[
    D \xi_t(x)\,=\, D \Phi_{w^1_t,\dots, w_t^N}(x) \,=\, D \phi_{w_t^1}^1\Bigl(\phi_{w_t^2}^2 \circ \dots \circ \phi_{w_t^N}^N(x)\Bigr)  \dots D \phi_{w_t^N}^N(x),
    \]
    so that the general condition \eqref{eq:cond_number} follows from our assumption \eqref{Eq55} on the individual $\phi^n$ for $n\le N$.
\end{proof}

\section{Global smooth solutions to quasilinear SPDEs}
\label{s:application_to_quasi}
In this last section, we apply our main result, Theorem \ref{t:DeGiorgiNashMoser}, to establish the existence of global, regular solutions to quasilinear SPDEs of the form
\begin{equation}
	\label{eq:SPDE_quasi}
	\begin{cases}
		\displaystyle{
			\dd U \, - \, \nabla \cdot (A(t,\cdot,U)\cdot \nabla U) \,\dd t \, = \, \textstyle{\sum}_{n\ge 1} (B_n(t,\cdot)\cdot \nabla) U \,\dd w^n_t,} &\ \ \ \text{ on }  [0,\infty) \times  \Tor^d,\\
		U(0)\,=\,U_0,&\ \ \  \text{ on }\Tor^d.
	\end{cases}
\end{equation}
 On the variable diffusion coefficient $A$, the transport noise coefficients $B_n$ and the initial value $U_0$, we impose the following, for which we recall that the space $\Lip$ refers to the \emph{space of  Lipschitz functions}.

\begin{assumption}[Coefficients]
\label{ass:quasilinear} There exist $C<\infty$, $\delta\in (0,1)$ and $ \nu >0  $ such that the  following holds:
   \begin{enumerate}[(i)]
    \item \label{I1} $B=(B_n)_{n\geq 1}:\R_+\times \O\times \T^d\to \ell^2(\R^d)$ is $\Progress\otimes \Borel(\T^d)$-measurable and a.s. for all $t\in [0,\infty)$
    $$\|B(t,\cdot)\|_{C^{3+\delta}(\T^d;\ell^2(\R^d))}\le C  .
    $$
    \item  \label{I2}$A:\R_+\times \O\times \T^d\times \R\to \R^{d\times d}$ is $\Progress\otimes \Borel(\T^d\times \R)$-measurable, and for all $N\geq 1$ there exists $C_N>0$ such that, a.s.\ for all $t\in [0,\infty)$ and $x\in \T^d$,
\begin{align*}
\|A(t,x,\cdot )\|_{{\Lip(-N,N)}}\le C_N.
\end{align*}
    \item A.s. for all $t\in[0,\infty)$, $x\in \T^d$, $y\in \R$ and $\eta\in \R^d$, we have
    $$ A(t,x,y)\eta\cdot\eta- \frac{1}{2}\Big(\sum_{n\geq 1} | B_n(t,x) \cdot \eta|^2\Big)
    \geq \nu |\eta|^2.
    $$
\end{enumerate}
\end{assumption}

\begin{assumption}[Initial data]
\label{ass:quasilinear_data}
Let $p,q\in (2,\infty)$ and $\kappa\in [0,\frac{p}{2}-1)$ be such that $2\frac{1+\kappa}{p}+\frac{d}{q}<1$ and let $U_0\in L^0_{\F_0}(\O;B^{1-2\frac{1+\kappa}{p}}_{q,p}(\T^d))$.
\end{assumption}  
Inspecting the above condition on the initial value reveals that one can equivalently demand that
\begin{equation}\label{Eq:equivalent_assumption}
	\text{$U_0\colon\Omega \to C^{\gamma_0}(\T^d)$ is strongly $\F_0$-measurable for some $\g_0 >0 $,}
\end{equation} 
 as we require in Theorem \ref{t:global_quasilinear_intro}.
Indeed if Assumption \ref{ass:quasilinear_data} is fulfilled, it suffices to take $\gamma_0 \le 1- 2\frac{1+\kappa}{p} -\frac{d}{q}$ and use the Sobolev embedding theorem to obtain \eqref{Eq:equivalent_assumption}. Conversely, if we have \eqref{Eq:equivalent_assumption},
we can first choose $p\in (2,\infty)$ and $\kappa\in [0,\frac{p}{2}-1)$ such that $0< 1-2\frac{1+\kappa}{p} \le \gamma_0$ and subsequently $q$ sufficiently large.
Here, we chose the formulation from Assumption \ref{ass:quasilinear_data} as it accommodates the $L^p(w_\kappa;L^q)$-setting of stochastic PDEs developed by the first and third named authors in \cite{AV19_QSEE1,AV19_QSEE2}.
Indeed, following \cite{ALV21}, $B^{1-2\frac{1+\kappa}{p}}_{q,p}(\T^d)$ is the  optimal trace space when studying parabolic SPDEs in time-weighted spaces, where we recall that $w_{\kappa}(t):=|t|^\kappa$ for $\kappa>0$ and
$$
\textstyle
\|f\|_{L^p(0,T,w_{\kappa};\mathscr{X})}^p:=\int_{0}^T \|f(t)\|_{\mathscr{X}}^p w_{\kappa}(t) \,\dd t, \qquad  \text{for }p\in [1,\infty) 
$$
and a Banach space $\mathscr{X}$. For more details, we also refer to the recent survey \cite{AV25_survey}. 

Returning to the study of the quasilinear SPDE \eqref{eq:SPDE_quasi}, we define solutions as follows. 

\begin{definition}[Solutions to \eqref{eq:SPDE_quasi}]\label{defi:sol}
Let Assumptions \ref{ass:quasilinear} and \ref{ass:quasilinear_data} be satisfied,
 $\sigma$ be a stopping time and 
$
U:(0,\sigma)\times \O\to W^{1,q}(\T^d)
$
be progressively measurable. 
\begin{enumerate}[{\rm(1)}]
\item We say that $(U,\sigma)$ is a \emph{local $(p,\kappa,q)$-solution} to \eqref{eq:SPDE_quasi} if there exists a sequence of stopping times $(\sigma_j)_{j\geq 1}$ for which $\sigma_j \to\sigma$ a.s. and for all $j\ge 1$
\begin{align}
\label{eq:regularity_U_local}
U\in L^p(0,\sigma_j,w_{\kappa};W^{1,q}(\T^d))\cap C([0,\sigma_j];B^{1-2\frac{1+\kappa}{p}}_{q,p}(\T^d)), \qquad \text{a.s.,}
\end{align}
and
\begin{align}
\label{eq:integrated_quasi}
U(t) - U_0 =\int_0^t \nabla \cdot (A(\cdot,U)\cdot \nabla U)\,\dd s 
+\sum_{n\geq 1} \int_0^t (B_n \cdot\nabla) U\,\dd w^n_s,\qquad t\in [0,\sigma_j].
\end{align}
\item We call a local $(p,\kappa,q)$-solution $(U,\sigma)$ to \eqref{eq:SPDE_quasi} \emph{unique} if for any other local $(p,\kappa,q)$-solution $(V,\tau)$ to \eqref{eq:SPDE_quasi} it holds a.s.\ $U=V$ on $[0,\sigma\wedge \tau )$.
\item We call a local $(p,\kappa,q)$-solution $(U,\sigma)$ to \eqref{eq:SPDE_quasi} \emph{maximal} if for any other local $(p,\kappa,q)$-solution $(V,\tau)$ to \eqref{eq:SPDE_quasi} it holds $\tau\leq \sigma$ a.s.\ and $U=V$ on $[0,\sigma\wedge \tau )$.
\item We call a local $(p,\kappa,q)$-solution $(U,\sigma)$ to \eqref{eq:SPDE_quasi} \emph{global} if  a.s.\ $\sigma=\infty$. In this case, we simply write $U$ instead of $(U,\sigma)$.
\end{enumerate}
\end{definition}

Note that maximal solutions are, in particular, unique. Moreover, a unique global $(p,\kappa,q)$-solution $(U,\sigma)$ is a fortiori maximal. 
We recall that by the assumption $\kappa\in [0,\frac{p}{2}-1)$ we have $L^p(0,T;w_\kappa)\embed L^2(0,T)$ for all   $T<\infty$. Consequently, Assumption \ref{ass:quasilinear}~\eqref{I1}--\eqref{I2} and that local $(p,\kappa,q)$-solutions satisfy \eqref{eq:regularity_U_local} ensure the convergence of the deterministic and the sum of stochastic integrals in \eqref{eq:integrated_quasi} as $W^{-1,q}(\T^d)$-valued Bochner and $L^q(\T^d)$-valued It\^o integral, respectively. We remark also that due to space-time continuity on $[0,\sigma)\times \T^d$ resulting from \eqref{eq:regularity_U_local} together with the condition $2\frac{1+\kappa}{p}+\frac{d}{q}<1$ we refer to solutions as in Definition \ref{defi:sol} as \emph{regular}.
 
The following is our result on the global well-posedness of \eqref{eq:SPDE_quasi}.

\begin{theorem}[Global regular solutions to \eqref{eq:SPDE_quasi}]
\label{t:global_quasilinear}
Let Assumptions \ref{ass:quasilinear} and \ref{ass:quasilinear_data} be satisfied. Then there exists a unique global $(p,\kappa,q)$-solution $U$ to \eqref{eq:SPDE_quasi} 
satisfying 
\begin{align*}
U&\in L^{p}_{\loc}((0,\infty);W^{1,q} (\T^d)) \cap C((0,\infty);B^{1-\frac{2}{p}}_{q,p}(\T^d)) \ \text{ a.s.\ }
\end{align*}
Moreover, the solution
 regularizes instantaneously in time and space:
\begin{equation}
\label{eq:regularization_of_SPDE}
U\in C_{\loc}^{\vartheta,2\vartheta}((0,\infty)\times \T^d) \ \text{ a.s.\ for all }\vartheta <\tfrac{1}{2}.
\end{equation}
\end{theorem}

The regularity in \eqref{eq:regularization_of_SPDE} is (almost) optimal for $A$ as in Assumption \ref{ass:quasilinear}. 
The regularity in time is clear. As for the spatial regularity, from \eqref{eq:SPDE_quasi}  (see \cite[Subsection 2.4]{AS24_thin_film} for a related situation), one sees that $\nabla U(t,x)$ is as smooth as $A(t,x,U(t,x))$. The latter is only bounded as $A(t,x,y)$ can depend roughly on $x$, and therefore one can expect $\nabla U$ to be at most bounded in space. The reader is referred to the companion paper \cite{ASV25_quasi} for the bootstrap of regularity in case of additional smoothness for $x\mapsto A(\cdot,x,\cdot)$.

We remark that by a standard localization argument in the probability space, one obtains from Theorem \ref{t:global_quasilinear} also global regular solutions to \eqref{eq:SPDE_quasi} if only 
$$
\P(U_0 \text{ is H\"older continuous})=1,$$
for a strongly $\F_0$-measurable  $U_0\colon \Omega \to C(\T^d)$.
Indeed, it is enough to apply the above result with $U_0$ replaced by $\one_{\O_n} U_0$, where $\O_n:=\{\|U_0\|_{C^{1/n}(\T^d)}\leq n\}$ for $n\geq 1$ arbitrary.
Finally, if the diffusion coefficient $A$ is more regular in $(x,y)$ than required in Assumption \ref{ass:quasilinear}~\eqref{I2}, the regularity assertion \eqref{eq:regularization_of_SPDE} can be further improved. For more details on the latter, we refer the reader to the forthcoming work  \cite{ASV25_quasi}, where we also include lower order nonlinearities in \eqref{eq:SPDE_quasi} and treat systems of SPDEs.

\subsection{Proof of Theorem \ref{t:global_quasilinear}}\label{ss:proof_global_ql}
As a first step to prove Theorem \ref{t:global_quasilinear}, we state a result on the local well-posedness and blow-up criteria for \eqref{eq:SPDE_quasi}. The emergence of H\"older norms in the blow-up condition \eqref{eq:blow_up_criteria} clarifies the link to the De Giorgi--Nash--Moser estimates of Theorem \ref{t:DeGiorgiNashMoser}. 

\begin{proposition}[Local well-posedness]
\label{prop:local_quasi}
Let Assumptions \ref{ass:quasilinear} and \ref{ass:quasilinear_data} be satisfied.  Then 
there exists a maximal $(p,\kappa,q)$-solution $(U,\sigma)$ to \eqref{eq:SPDE_quasi} satisfying a.s.\ $\sigma>0$ and
\begin{align}
U&\in L^{p}_{\loc}([0,\sigma);W^{1,q} (\T^d)) \cap C([0,\sigma);B^{1-2\frac{1+\kappa}{p}}_{q,p}(\T^d)), \\
\label{eq:instantaneous_regularization_local}
U&\in C_{\loc}^{\vartheta,2\vartheta}((0,\sigma)\times \T^d) \ \text{ for all }\vartheta <\tfrac{1}{2},
\end{align}
as well as
\begin{equation}
\label{eq:blow_up_criteria}
\P\Big(t<\sigma<T,\, \sup_{s\in [t,\sigma)}\|U(s)\|_{C^{\varepsilon}(\T^d)}<\infty\Big)=0,
\end{equation}
for all $0<t<T<\infty$ and $\varepsilon \in (0,1)$.
\end{proposition}

As we will see in the proof of Theorem \ref{t:global_quasilinear}, the flexibility in taking 
$\varepsilon>0$ small plays a crucial role when applying our version of the De Giorgi--Nash--Moser estimates. This precisely counterbalances the possibility of the sequence of H\"older exponents $\g_m$ from Theorem \ref{t:DeGiorgiNashMoser} to degenerate.

Proposition \ref{prop:local_quasi} can be proved using the results on quasilinear stochastic evolution equations from \cite{AV19_QSEE1,AV19_QSEE2} in combination with \cite{AV21_SMR_torus}, in the same way as done by the first two authors in \cite[Proposition 2.12]{AS24_thin_film} for fourth-order quasilinear SPDEs. The independence of the blow-up criteria \eqref{eq:blow_up_criteria} of the choice of $\varepsilon>0$ follows as in \cite[Corollary 2.14]{AS24_thin_film} from the instantaneous regularization   \eqref{eq:instantaneous_regularization_local}, cf.  \cite[Theorem 2.10]{RD_AV23} for a similar situation. Details in a more general situation are also given in the companion paper \cite{ASV25_quasi}. 

Based on Proposition \ref{prop:local_quasi}, Theorem \ref{t:global_quasilinear} follows from our main result, Theorem \ref{t:DeGiorgiNashMoser}.

\begin{proof}[Proof of Theorem \ref{t:global_quasilinear}]
For $(U,\sigma)$ as in Proposition \ref{prop:local_quasi}, it suffices to show that $\sigma=\infty$ a.s.\ and, to this end, we set
$$
a^{ij}(t,\om,x):= \begin{cases}
	A^{ij}(t,\om,x,U(t,\om,x)), & t< \sigma, \\
	A^{ij}(t,\om,x,0 ), & t\ge\sigma,
	\\
\end{cases}
	\qquad  \text{on } [0,\infty )\times \O\times \T^d,
$$
where we a-priori do not have any control over the regularity of the coefficient $a$ for $t$ close to $\sigma$. Assumption \ref{ass:quasilinear} ensures, however, that $a$ is bounded and measurable, and that $a$ together with $b_n:=B_n$ is uniformly parabolic, implying therefore Assumption \ref{ass:DeGiorgi_Nash_Moser}~\eqref{it:am_DeGiorgi}. 

Given $T\in (0,\infty)$, there exists by Proposition \ref{prop:Existence_Uniqueness} a solution $u$ to the linear SPDE 
\begin{equation}\label{eq:linear_SPDE_quasi}
	\left\{
	\begin{aligned}
		&\dd u \, - \, \nabla \cdot (a\cdot \nabla u) \,\dd t \, = \,\textstyle{\sum}_{n\ge 1} (b_n\cdot \nabla) u \,\dd w^n_t, & & \text{ on }[0,T]\times \Tor^d,\\
		&u(0)\,=\,U_0,& & \text{ on }\Tor^d,
	\end{aligned}
	\right.
\end{equation}
and by uniqueness we must have a.s.\ $u=U$ on $[0,\sigma_j\wedge T]$ for all $j\ge 1$ and therefore on $[0,\sigma\wedge T)$ (here $\sigma_j$ are as in Definition \ref{defi:sol}).
By Assumption \ref{ass:quasilinear} and the equivalence of Assumption \ref{ass:quasilinear_data} and \eqref{Eq:equivalent_assumption} all conditions of Theorem \ref{t:DeGiorgiNashMoser} are fulfilled and we conclude that
\begin{equation}
	\label{eq:u_continuous_gamma_m}
	u\in C^{\gamma_m}((0,\tau_m)\times \T^d) \ \text{ a.s.,\ }
\end{equation}
for a sequence of positive numbers
$(\g_m)_{m\geq 1}$ and stopping times $(\tau_m)_{m\geq 1}$ with $\lim_{m\to \infty} \P(\tau_m=T)=1$. Using all this, we conclude that for any $t\in (0,T)$, we have 
\begin{align*}
\P(t<\sigma<T)
&= \lim_{m\to \infty}\P(t<\sigma<T,\, \tau_m=T)
\\
& = \lim_{m\to \infty}\,\P\bigl(t<\sigma<T,\, \tau_m=T , \, u\in C^{\gamma_m}((0,\tau_m)\times \T^d) ,\, u=U \text{ on } [0,\sigma\wedge T) \bigr)
\\&\le \liminf_{m\to\infty} \,
\P(t<\sigma<T , \, U\in C^{\gamma_m}((0,\sigma)\times \T^d)  ).
\end{align*}
Setting now $\epsilon = \gamma_m$ in the blow-up condition \eqref{eq:blow_up_criteria} we find that the right-hand side of the above equals $0$ and it remains to use that by $\sigma>0$ a.s.\
\[
\P ( \sigma<\infty ) = \P ( 0<\sigma<\infty ) = \lim_{r\to \infty} \P( 1/ r < \sigma <r ) = 0,
\]
to conclude the proof.
\end{proof}

\subsubsection*{Acknowledgments}
MS thanks the Mathematics Department Guido Castelnuovo at Sapienza University of Rome for the kind hospitality during his visit in May 2025, during which parts of this research were conducted. 

\def\polhk#1{\setbox0=\hbox{#1}{\ooalign{\hidewidth
  \lower1.5ex\hbox{`}\hidewidth\crcr\unhbox0}}} \def\cprime{$'$}

\end{document}